\documentclass[oneside,english,british,reqno]{amsart}
\usepackage[T1]{fontenc}
\usepackage[latin9]{inputenc}
\usepackage{geometry}
\usepackage{changepage}
\usepackage{amsthm}
\usepackage{amstext}
\usepackage{amssymb}
\usepackage{graphicx}
\usepackage{esint}
\usepackage{subfig}
\usepackage{mathtools}
\usepackage{babel}
\usepackage[breaklinks=true,a4paper=true,pagebackref=true]{hyperref}
\usepackage{color}
\usepackage{rotating}

\makeatletter
\numberwithin{equation}{section}
\numberwithin{figure}{section}
\theoremstyle{plain}
\newtheorem{thm}{\protect\theoremname}[section]
  \theoremstyle{plain}
  
  \theoremstyle{plain}
  \newtheorem{lem}[thm]{\protect\lemmaname}
  \theoremstyle{plain}
  \newtheorem{prop}[thm]{\protect\propname}
  \theoremstyle{remark}
  
  \theoremstyle{plain}
  
  \theoremstyle{definition}
  \newtheorem{defn}[thm]{\protect\definitionname}
  \theoremstyle{plain}
  \newtheorem{cor}[thm]{\protect\corollaryname}

\makeatother
\providecommand{\conjecturename}{Conjecture}
\providecommand{\definitionname}{Definition}
\providecommand{\factname}{Fact}
\providecommand{\lemmaname}{Lemma}
\providecommand{\propname}{Proposition}
\providecommand{\remarkname}{Remark}
\providecommand{\theoremname}{Theorem}
\providecommand{\corollaryname}{Corollary}

\def\EC{\mathcal{E}}
\global\long\def\cbr#1{\left\{  #1\right\}  }
\global\long\def\rbr#1{\left(#1\right)}
\global\long\def\R{\mathbb{R}}

\global\long\def\dd#1{{d}#1}

\global\long\def\eqdef{:=}
\def\eps{\varepsilon}
\def\T{\mathbb{T}}
\def\zero{{\bf 0}}
\def\P{\mathbb{P}}
\def\F{\mathcal{F}}
\def\E{\mathbb{E}}

 \DeclareMathOperator{\var}{Var}

\begin{document}

\title{Delocalization of two-dimensional random surfaces with hard-core constraints}

\author{Piotr Mi\l{}o\'{s}}
\thanks{Research of P.M. was partially supported by the Polish Ministry of Science and Higher Education
Iuventus Plus Grant no. IP 2011 000171.}
\address{Faculty of Mathematics, Informatics, and Mechanics, University of Warsaw, Banacha 2, \text{02-097 Warszawa}, Poland}
\email{pmilos@mimuw.edu.pl}
\urladdr{http://www.mimuw.edu.pl/~pmilos}

\author{Ron Peled}
\thanks{Research of R.P. is partially supported by an ISF grant and an IRG
grant.}
\address{School of Mathematical sciences, Tel-Aviv University, Tel-Aviv 69978, Israel.}
\email{peledron@post.tau.ac.il}
\urladdr{http://www.math.tau.ac.il/~peledron}

\begin{abstract}
We study the fluctuations of random surfaces on a two-dimensional
discrete torus. The random surfaces we consider are defined via a
nearest-neighbor pair potential which we require to be twice
continuously differentiable on a (possibly infinite) interval and infinity outside of
this interval. No convexity assumption is made and we include the
case of the so-called hammock potential, when the random surface is
uniformly chosen from the set of all surfaces satisfying a Lipschitz
constraint. Our main result is that these surfaces delocalize,
having fluctuations whose variance is at least of order $\log n$,
where $n$ is the side length of the torus. We also show that the
expected maximum of such surfaces is of order at least $\log n$. The
main tool in our analysis is an adaptation to the lattice setting of
an algorithm of Richthammer, who developed a variant of a
Mermin-Wagner-type argument applicable to hard-core constraints. We rely
also on the reflection positivity of the random surface model. The
result answers a question mentioned by Brascamp, Lieb and Lebowitz
\cite{Brascamp:1975vn} on the hammock potential and a question of
Velenik \cite{Velenik:2006kx}.
\end{abstract}

\maketitle

\section{Introduction}\label{sec:introduction}
In this paper we study the fluctuations of random surface models in
two dimensions. We consider the following family of models. Denote
by $\T_{n}^2$ the two-dimensional discrete torus in which the vertex
set is $\{-n+1,-n+2,\ldots,n-1,n\}^{2}$ and $(a,b)$ is adjacent to
$(c,d)$ if $(a,b)$ and $(c,d)$ are equal in one coordinate and
differ by exactly one modulo $2n$ in the other coordinate. Let $U$
be a potential, i.e, a measurable function $U:\R\to(-\infty,\infty]$
satisfying $U(x)=U(-x)$. The random surface model with potential
$U$, normalized at the vertex $\zero:=(0,0)$, is the probability
measure $\mu_{\T_n^2,\zero, U}$ on functions
$\varphi:V(\T_n^2)\to\R$ defined by
\begin{equation}\label{eq:mu_T_n_2_U_measure_def}
  d\mu_{\T_n^2, \zero, U}(\varphi) := \frac{1}{Z_{\T_n^2,\zero, U}} \exp\Bigg(-\sum_{(v,w)\in E(\T_n^2)}
  U(\varphi_v - \varphi_w)\Bigg) \delta_0(d\varphi_{\zero})\prod_{v\in V(\T_n^2)\setminus\{\zero\}}
  d\varphi_v,
\end{equation}
where the vertices and edges of $\T_n^2$ are denoted by $V(\T_n^2)$
and $E(\T_n^2)$ respectively, $d\varphi_v$ denotes Lebesgue measure
on $\varphi_v$, $\delta_0$ is a Dirac delta measure at $0$ and
$Z_{\T_n^2,\zero, U}$ is a normalization constant. For this
definition to make sense the potential $U$ needs to satisfy
additional requirements. It suffices, for instance (see
Lemma~\ref{lem:measure_well_defined} for additional details), that
\begin{equation}\label{eq:U_integral_cond}
\inf_x U(x) > -\infty\quad\text{and}\quad 0<\int
\exp(-U(x))dx<\infty.
\end{equation}

Suppose $\varphi$ is sampled from the measure $\mu_{\T_n^2, \zero,
U}$. The expectation of $\varphi$ is zero at all vertices by
symmetry. How large are the fluctuations of $\varphi$ around zero?
Let us focus on the variance of $\varphi$ at the vertex $(n,n)$. It is expected that this variance is of order $\log n$ under mild
conditions on $U$. This has been shown when the potential $U$ is
twice continuously differentiable with $U''$ bounded away from zero
and infinity, and certain extensions of this class, as discussed in
the survey paper \cite[Remarks 6 and 7]{Velenik:2006kx}.
Specifically, a lower bound of order $\log n$ has been established by Brascamp, Lieb and Lebowitz \cite{Brascamp:1975vn} when $U$ is twice continuously differentiable,
\[
    \int \exp(-\alpha U(x)) dx <+\infty, \forall {\alpha>0},\quad \lim_{|x|\to\infty} (|x| + |U'(x)|)\exp(-U(x)) = 0,
\]
and either of the following holds:
\begin{enumerate}
    \item $\sup_x U''(x)<\infty$ or
    \item $\sup_x |U'(x)|<\infty$ or
    \item $U$ is convex and $\int U'(x)^2\exp(-U(x))dx < \infty$.
\end{enumerate}
The class of potentials covered by their result can be further extended by taking suitable limits, as indicated in \cite{Brascamp:1975vn}. In addition, using arguments of Ioffe, Shlosman and Velenik \cite{Ioffe:2002fk} it is possible to derive qualitatively correct lower bounds for the variance for a class of, possibly discontinuous, potentials  satisfying
\[
  \|U - \tilde{U}\|_\infty < \eps
\]
for a small enough $\eps>0$  and some twice continuously differentiable $\tilde{U}$ satisfying $\sup_x \tilde{U}''(x) < \infty$.

The case of the \emph{hammock potential}, when $U(x) = 0$ for
$|x|\le 1$ and $U(x) = \infty$ for $|x|>1$, is explicitly mentioned
as open in \cite{Brascamp:1975vn} and \cite[Open problem
2]{Velenik:2006kx}. In this paper we prove a lower bound of order
$\log n$ on the variance for a wide class of potentials which
includes the hammock potential. A sample from the random surface
measure with the hammock potential is depicted in
Figure~\ref{2d_3d_100_cont_fig}, both in 2 and 3 dimensions.
\begin{figure}[t!]
\centering {\includegraphics[width=\textwidth, viewport=25 130 990 640, clip]{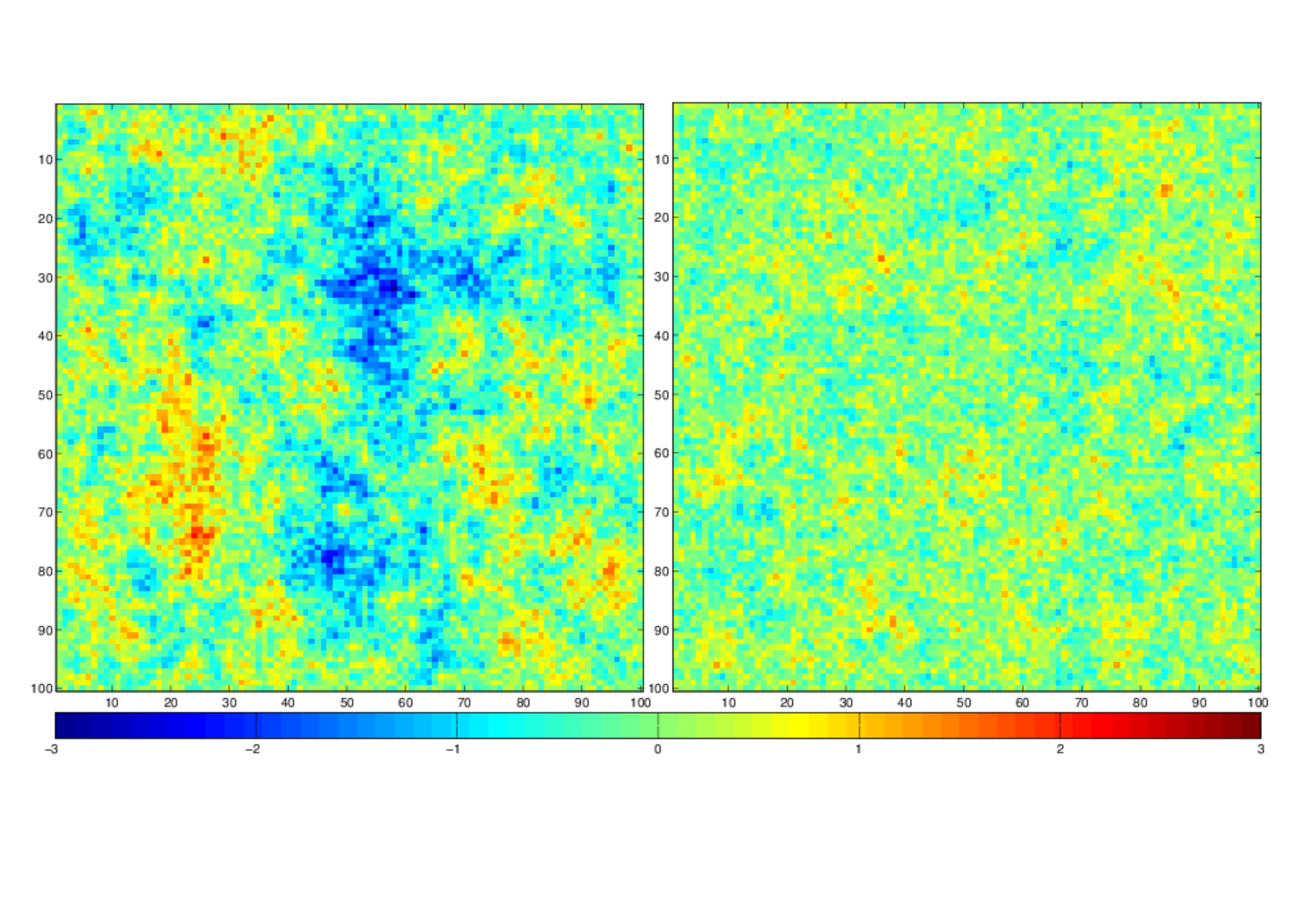}} {\it
\caption{Samples of the random surface measure with the hammock
potential, i.e., samples of a uniformly chosen Lipschitz function
taking real values which differ by at most one between adjacent
vertices. The left picture shows a sample on the 100 x 100 square
and the right picture shows the middle slice (at height 50) of a
sample on the $100\times100\times100$ cube, both conditioned to have
all boundary values in the $[-\frac{1}{2},\frac{1}{2}]$ interval.
Sampled using coupling from the past
\cite{ProppWilson96}.\label{2d_3d_100_cont_fig}} }
\end{figure}

We say that $U\in C^2(I)$ for an interval $I\subseteq\R$ if $U$ is
twice continuously differentiable on $I$. We consider the class of
potentials $U$ satisfying the following condition:
\begin{equation}\label{eq:potential_condition}
  \text{Either $U\in C^2(\R)$ or $U\in C^2((-K,K))$ for some $0<K<\infty$ and $U(x)=\infty$ when $|x|>K$}.
\end{equation}
This class includes the hammock potential as well as ``double well''
potentials, oscillating potentials with finite support (that is,
infinity outside of a bounded interval) and all smooth examples. In
the case that $U\in C^2((-K,K))$ we allow the possibility of a
discontinuity at the endpoints $-K$ and $K$. The following theorem
is the main result of this paper. Besides proving a lower bound on
the variance at the vertex $(n,n)$ we obtain estimates also for
other vertices, for small ball and large deviation probabilities and
for the maximum of the random surface.
\begin{thm}\label{thm:main}
Let $U:\R\to(-\infty,\infty]$ satisfy $U(x)=U(-x)$ and
conditions~\eqref{eq:U_integral_cond} and
\eqref{eq:potential_condition}. Let $n\ge 2$ and let $\varphi$ be
randomly sampled from $\mu_{\T_n^2, \zero, U}$. There exist
constants $C(U),c(U)>0$, depending only on $U$, such that for any
$v\in V(\T_n^2)$ with $\|v\|_1\ge (\log n)^2$ we have
\begin{alignat*}{2}
  &\var(\varphi_v) \ge c(U) \log(1+\|v\|_1),\\
  &\P(|\varphi_v|\le \delta\sqrt{\log(1+\|v\|_1)})\le C(U)\delta^{2/3},\quad&&\delta\ge \frac{1}{\sqrt{\log(1+\|v\|_1)}},\\
  &\P(|\varphi_v| \ge c(U)t\sqrt{\log(1+\|v\|_1)})\ge
  c(U)e^{-C(u)t^2},\quad&& 1\le t\le \frac{1+\sqrt{\|v\|_1}}{1+\log n}.
\end{alignat*}
In addition,
\begin{equation*}
  \P\left(\max_{v\in V(\T_n^2)} |\varphi_v|\ge c(U)\log n\right) \ge \frac{1}{2}.
\end{equation*}
\end{thm}
We remark that condition \eqref{eq:U_integral_cond} is mainly
required in this theorem for the probability
measure~\eqref{eq:mu_T_n_2_U_measure_def} to make sense. One may
replace it by other conditions of a similar nature. Additional
remarks may be found following
Theorem~\ref{thm:variance_lower_bound} below.

Our results can be viewed in a broader context of Mermin-Wagner-type arguments. Such arguments show, roughly, that continuous translational symmetry cannot be broken in one- or two-dimensional systems. For lattice models with compact spin spaces this implies that spins are uniformly distributed in the infinite volume limit. For lattice models with non-compact spin spaces, such as the random surface models we consider, such arguments prove delocalization and consequently non-existence of infinite volume Gibbs measures. We present now a non-exhaustive list of papers studying these phenomena. Such arguments were pioneered by Mermin and Wagner \cite{Mermin:1966fr} who worked in a quantum context and relied on the, so called, Bogoliubov inequalities. These techniques were later extended and transferred to a classical context - see e.g. Hohenberg \cite{Hohenberg:1967zr} and Brascamp, Lieb and Lebowitz \cite{ Brascamp:1975vn}. New techniques were developed by Dobrushin and Shlosman \cite{Dobrushin:1975ly, Dobrushin:1980ve}, McBryan and Spencer \cite{McBryan:1977fk} and Fr\"ohlich and Pfister \cite{Pfister:1981qy, Frohlich:1981qf}. The methods in all of the above papers require the potential to satisfy certain smoothness assumptions. Ioffe, Shlosman and Velenik \cite{Ioffe:2002fk} and Gagnebin and Velenik \cite{Gagnebin:2013lr} presented extensions to some classes of non-smooth potentials. These works left open the case of potentials taking infinite values and a solution to this problem came from Richthammer \cite{Richthammer:2007fk} who studied Gibbsian point processes in $\R^2$. Our approach follows closely his elaborate technique introduced for proving that all Gibbs states of such point process models are translation invariant, even in the presence of \emph{hard-core} constraints as in the hard sphere model. The main ingredient in Richthammer's approach is an algorithm designed for perturbing a given configuration in a prescribed manner while preserving the hard-core constraints. Our proof adapts this algorithm from the continuum to the graph setting and from the point process to the random surface context. The resulting adaptation is presented in some detail in Section~\ref{sec:Shift-Algorithm-and} and we hope that it will be useful in other contexts as well.

\subsection{Overview of the proof}

In order to illustrate our proof we first explain how to establish a
lower bound on fluctuations in the simpler case that the potential
$U$ satisfies that
\begin{equation}\label{eq:upper_bound_second_deriv_assumption}
\text{$U$ is twice continuously differentiable on $\R$ and }\sup_x
U''(x)<\infty,
\end{equation}
in addition to the condition \eqref{eq:U_integral_cond}. The methods of this section are similar to the one of \cite{Pfister:1981qy}. We then
provide details on the modification of this method, following the
approach of Richthammer, which we use for potentials satisfying
condition~\eqref{eq:potential_condition}.
\subsubsection{Delocalization argument for potentials with bounded
second derivative}\label{sec:Vanilla_Mermin_Wagner} Suppose the
potential $U$ satisfies
\eqref{eq:upper_bound_second_deriv_assumption} and the
condition~\eqref{eq:U_integral_cond}. Let $n\ge 1$ and let $\varphi$
be randomly sampled from $\mu_{\T_n^2, \zero, U}$. We will show that
\begin{equation}\label{eq:required_variance_bounds}
  \var(\varphi_{(n,n)})\ge c(U) \log n
\end{equation}
for some $c(U)>0$.

Write
\begin{equation*}
  g(\psi):=\frac{1}{Z_{\T_n^2, \zero, U}} \exp\Bigg(-\sum_{(v,w)\in E(\T_n^2)}U(\psi_v - \psi_w)\Bigg)
\end{equation*}
for the density of the measure $\mu_{\T_n^2, \zero, U}$ at the
configuration $\psi$. Define a function
$\tau:V(\T_n^2)\to[0,\infty)$ by
\begin{equation}\label{eq:tau_def_intro}
  \tau(v):=\frac{\log(1 + \|v\|_1)}{\sqrt{\log(2n+1)}},
\end{equation}
where $\|(v_1,v_2)\|_1:=|v_1|+|v_2|$. It is not difficult to check
that
\begin{equation}\label{eq:squared_gradients_tau}
  \sum_{(v,w)\in E(\T_n^2)} |\tau(v) - \tau(w)|^2 \le C
\end{equation}
for some absolute constant $C>0$. For a configuration
$\psi:V(\T_n^2)\to\R$ define the two shifted configurations
\begin{equation}\label{eq:t_plus_minus_def}
  \psi^+ := \psi + \tau,\quad \psi^- := \psi - \tau.
\end{equation}
A Taylor expansion of $U$, the assumption
\eqref{eq:upper_bound_second_deriv_assumption} and
\eqref{eq:squared_gradients_tau} imply that
\begin{align}\label{eq:geometric_average_densities}
  &\sqrt{g(\psi^+) g(\psi^-)}=\nonumber\\
  &= \frac{1}{Z_{\T_n^2, \zero, U}} \exp\bigg(-\frac{1}{2}\sum_{(v,w)\in E(\T_n^2)} U(\psi_v - \psi_w + \tau(v) - \tau(w)) + U(\psi_v - \psi_w - \tau(v) + \tau(w))\bigg) \ge\nonumber\\
  &\ge \frac{1}{Z_{\T_n^2, \zero, U}} \exp\bigg(-\sum_{(v,w)\in E(\T_n^2)} U(\psi_v - \psi_w) - \sup_x U''(x) \sum_{(v,w)\in E(\T_n^2)} (\tau(v) - \tau(w))^2\bigg) \ge c(U) g(\psi),
\end{align}
for some $c(U)>0$.

 We wish to convert the inequality
\eqref{eq:geometric_average_densities} into an inequality of
probabilities rather than densities. To this end define
\begin{align*}
  &E_a:=\{\psi:V(\T_n^2)\to\R\colon|\psi_{(n,n)}|\le a\},\quad a>0,\\
  &d\lambda(\psi):=\delta_0(d\psi_{\zero})\prod_{v\in
  V(\T_n^2)\setminus\{\zero\}}d\psi_v.
\end{align*}
Let $a>0$ and define
\begin{equation*}
  I:=\int_{E_a} \sqrt{g(\psi^+)g(\psi^-)}d\lambda(\psi).
\end{equation*}
On the one hand, by~\eqref{eq:geometric_average_densities},
\begin{equation}\label{eq:prob_estimate_1}
  I \ge c(U)\int_{E_a} g(\psi)d\lambda(\psi) = c(U)\P(|\varphi_{(n,n)}|\le a).
\end{equation}
On the other hand, the Cauchy-Schwartz inequality and a change of
variables using \eqref{eq:t_plus_minus_def} and the fact that
$\tau(\zero)=0$ yields
\begin{align}\label{eq:prob_estimate_2}
  I\le \left(\int_{E_a} g(\psi^+) d\lambda(\psi) \int_{E_a} g(\psi^-) d\lambda(\psi)\right)^{\frac{1}{2}} = \left(\P(|\psi_{(n,n)} - \tau((n,n))|\le a)\P(|\psi_{(n,n)} + \tau((n,n))|\le
  a)\right)^{\frac{1}{2}}.
\end{align}
Putting together \eqref{eq:prob_estimate_1} and
\eqref{eq:prob_estimate_2} and recalling \eqref{eq:tau_def_intro} we
obtain
\begin{align}\label{eq:prob_estimate_final}
  \left(\P(|\varphi_{(n,n)} - \sqrt{\log(2n+1)}|\le a)\P(|\varphi_{(n,n)} + \sqrt{\log(2n+1)}|\le
  a)\right)^{\frac{1}{2}}\ge c(U)\P(|\varphi_{(n,n)}|\le a).
\end{align}
Using the symmetry of the distribution of $\varphi$, the
arithmetic-geometric mean inequality and taking
$a:=\frac{1}{3}\sqrt{\log(2n+1)}$ in the last inequality we conclude
\begin{align*}
  &\frac{1}{2}\P\left(|\varphi_{(n,n)}|\ge\frac{2}{3}\sqrt{\log(2n+1)}\right)\\&=\frac{1}{2}\left[\P\left(\varphi_{(n,n)}\ge\frac{2}{3}\sqrt{\log(2n+1)}\right)
  + \P\left(\varphi_{(n,n)}\le-\frac{2}{3}\sqrt{\log(2n+1)}\right)\right] \ge\\
  &\ge\left[\P\left(\varphi_{(n,n)}\ge\frac{2}{3}\sqrt{\log(2n+1)}\right)\P\left(\varphi_{(n,n)}\le-\frac{2}{3}\sqrt{\log(2n+1)}\right)\right]^{1/2}\ge\\
  &\ge\left[\P\left(|\varphi_{(n,n)}-\sqrt{\log(2n+1)}|\le\frac{1}{3}\sqrt{\log(2n+1)}\right)\right.\\&\quad\left.\P\left(|\varphi_{(n,n)}+\sqrt{\log(2n+1)}|\le\frac{1}{3}\sqrt{\log(2n+1)}\right)\right]^{1/2}\ge\\
  &\ge c(U)\P\left(|\varphi_{(n,n)}|\le \frac{1}{3}\sqrt{\log(2n+1)}\right),
\end{align*}
from which we conclude $\E \varphi_{(n,n)}^2 \ge c'(U)\log n$ for
some $c'(U)>0$. The inequality \eqref{eq:required_variance_bounds}
follows as $\E \varphi_{(n,n)} = 0$ by symmetry.

\subsubsection{Modification of the argument for potentials
satisfying \eqref{eq:potential_condition}} For simplicity, assume
the potential $U$ satisfies $U\in C^2([-1,1])$ and $U(x)=\infty$
when $|x|>1$, as more general potentials satisfying
\eqref{eq:potential_condition} may be treated by similar arguments.
Let us say that a configuration $\psi:V(\T_n^2)\to\R$ is
\emph{Lipschitz} if $|\psi_v-\psi_w|\le 1$ whenever $(v,w)\in
E(\T_n^2)$. The measure $\mu_{\T_n^2,\zero,U}$ is supported on
Lipschitz configurations (satisfying $\psi(\zero)=0$) under our
assumption on $U$. The fundamental difficulty in applying the
previous argument to this case is that it may happen that although
$\psi$ is a Lipschitz configuration, one of the configurations
$\psi^+$ or $\psi^-$ defined by \eqref{eq:t_plus_minus_def} may fail
to be, in which case the
inequality~\eqref{eq:geometric_average_densities} will not be
satisfied. The solution we use for this problem is to replace the
configurations $\psi^+$ and $\psi^-$ in the previous argument by
$T^+(\psi)$ and $T^-(\psi)$, where
$T^+,T^-:\R^{V(\T_n^2)}\to\R^{V(\T_n^2)}$ are certain mappings,
termed addition algorithms in our paper, which share many of the
properties of the operations of adding and subtracting $\tau$ while
preserving the class of Lipschitz configurations. The definitions and properties of $T^+$ and $T^-$ are adapted from
the work of Richthammer \cite{Richthammer:2007fk} who showed that
all Gibbs states of point process models in $\R^2$ with hard-core
constraints, such as the hard sphere model, are translation
invariant. Our adaptation translates Richthammer's notions from the
continuum to the graph setting and from the point process to the
random surface context. The main properties of $T^+$ and $T^-$ are
detailed in Section~\ref{sec:addition_algorithm_prop}. We highlight
the possibility of defining these mappings for general graphs and
general addition functions $\tau$ as we believe these extensions to
be useful in other contexts and as they are captured with the same
definitions and proofs.

The mappings $T^+$ and $T^-$ are defined to satisfy
$T^-(\psi):=2\psi - T^+(\psi)$, just as in the definitions of
$\psi^+$ and $\psi^-$ in \eqref{eq:t_plus_minus_def}. It thus
suffices to define $T^+(\psi)$. Let us remark briefly on this
definition for a Lipschitz configuration $\psi$. Roughly speaking, a
certain $\psi$-dependent ordering on the vertices of the graph is
chosen. Then, for each vertex $v$ in this order, an amount between
$0$ and $\tau(v)$ is added to $\psi_v$ in such a way that the
Lipschitz property is maintained with respect to the previously
treated vertices in the chosen order. The amount added at vertex $v$
is chosen to vary continuously with the value $\psi_v$, in such a
way that the resulting operation is invertible.

Two difficulties arise when replacing $\psi^+$ and $\psi^-$ by
$T^+(\psi)$ and $T^-(\psi)$ in the argument of
Section~\ref{sec:Vanilla_Mermin_Wagner}. First, the change of
variables used in inequality~\eqref{eq:prob_estimate_2} relied on
the fact that the mappings $\psi\mapsto\psi+\tau$ and
$\psi\mapsto\psi-\tau$ preserve Lebesgue measure. When making a
change of variables from $T^+(\psi)$ and $T^-(\psi)$ to $\psi$ a
Jacobian factor enters, which needs to be estimated. Second, the
argument uses the fact that $\psi^+_{(n,n)}$ and $\psi^-_{(n,n)}$
differ significantly from $\psi_{(n,n)}$, by the amount
$\sqrt{\log(2n+1)}$. Thus we also need to show that the difference
of $T^+(\psi)_{(n,n)}$ and $T^-(\psi)_{(n,n)}$ from $\psi_{(n,n)}$
is close to $\sqrt{\log(2n+1)}$, at least for most configurations
$\psi$. It turns out that both these difficulties may be overcome if
we can control the following percolation-like process. We say an
edge $e=(v,w)\in E(\T_n^2)$ has \emph{extremal slope} for the
configuration $\psi$ if $|\psi_v - \psi_w|\ge 1-\eps$, for some
small $\eps>0$ fixed in advance. Sampling $\varphi$ randomly from
the measure $\mu_{\T_n^2, \zero, U}$, we denote by $\EC(\varphi)$
the random subgraph of $\T_n^2$ consisting of all edges with
extremal slope for $\varphi$. Both difficulties described above may
be overcome by showing that with high probability, the subgraph
$\EC(\varphi)$ is ``subcritical'' in the sense that its connected
components are small. Proving this turns out to be a non-trivial
task, which requires us to make use of reflection positivity
techniques, specifically, the chessboard estimate. We remark that
here (and only here) we rely essentially on the fact that $\T_n^2$
is a torus (i.e., has periodic boundary) and that the measure
$\mu_{\T_n^2, \zero, U}$ is normalized at the single vertex $\zero$.
Analogous estimates were also required in Richthammer's work
\cite{Richthammer:2007fk} but were provided by the underlying
Poisson process structure of the problem considered there, via
so-called Ruelle bounds.

\subsubsection{Reader's guide}
In Section~\ref{sec:Shift-Algorithm-and} we describe the mappings
$T^+$ and $T^-$ mentioned in the previous section. The section
begins by listing the main properties of $T^+$ and $T^-$, continues
with a precise definition of $T^+$ and proceeds to prove that the
required properties of $T^+$ indeed hold with this definition. In
Section~\ref{sec:reflection_positivity} we discuss reflection
positivity for random surface models and prove, via the chessboard
estimate, that the subgraph of edges with extremal slopes mentioned
in the previous section is ``subcritical'' with high probability.
Sections~\ref{sec:Shift-Algorithm-and} and
\ref{sec:reflection_positivity} address disjoint aspects of the
problem and may be read independently. In
Section~\ref{sec:proofs_for_main_theorems} we prove our main
theorem, Theorem~\ref{thm:main}, under alternative assumptions, by modifying the
argument presented in Section~\ref{sec:Vanilla_Mermin_Wagner} to
make use of the mappings $T^+$ and $T^-$ and extending it to provide
information also on small ball and large deviation probabilities and
on the maximum of the random surface. In the short
Section~\ref{sec:main_theorem_proof} we use the results of
Section~\ref{sec:reflection_positivity} to reduce
Theorem~\ref{thm:main} to the case discussed in
Section~\ref{sec:proofs_for_main_theorems}.
Section~\ref{sec:open_prob} contains a discussion of future research
directions and open questions.

\section{\label{sec:Shift-Algorithm-and}The addition algorithm and its properties}
In this section we define the addition algorithm $T^+$ which forms a
core part of our proof. The algorithm is an adaptation to the graph
setting of an algorithm of Richthammer \cite{Richthammer:2007fk} used in a
continuum setting. Our presentation adapts the proofs in \cite{Richthammer:2007fk} but emphasizes the applicability of the algorithm to general
graphs and general addition functions $\tau$.

\subsection{Properties of the addition
algorithm}\label{sec:addition_algorithm_prop} Here we describe the
properties of the addition algorithm which will be used by our
application. The algorithm itself is defined in the next section and
the fact that it satisfies the stated properties is verified in the
subsequent sections.

 Let $G=(V,E)$ be a finite, connected graph. We sometimes write $v\sim w$ to denote that $(v,w)\in E$. Let $\tau:V\to[0,\infty)$ and
$0<\eps\le\frac{1}{2}$ be given. We define a pair of measurable  mappings $T^+, T^-:\R^V\to\R^V$ related by the
equality
\begin{equation}\label{eq:T_+_T_-_relation}
  T^+(\varphi) - \varphi = \varphi -
  T^-(\varphi),\quad\varphi\in\R^V,
\end{equation}
and satisfying the following properties:
\begin{enumerate}
  \item \label{property:one-to-one_onto}$T^+$ and $T^-$ are one-to-one and onto.
  \item \label{property:maximal_increment}For every $\varphi\in\R^V$ and every $v\in V$,
  \begin{equation}\label{eq:T^+_T^-_increment_range}
    0\le T^+(\varphi)_v - \varphi_v = \varphi_v - T^-(\varphi)_v\le \tau(v).
  \end{equation}
  \item \label{property:Lipschitz_preservation} For every $\varphi\in\R^V$ and every $(v,w)\in E$,
  \begin{align*}
    \text{if }&|\varphi_v - \varphi_w|\ge 1\text{ then } T^+(\varphi)_v -
    T^+(\varphi)_w = T^-(\varphi)_v - T^-(\varphi)_w = \varphi_v -
    \varphi_w,\\
    \text{and if }&|\varphi_v - \varphi_w|< 1\text{ then } |T^+(\varphi)_v -
    T^+(\varphi)_w| < 1\;\;\text{ and }\;\; |T^-(\varphi)_v -
    T^-(\varphi)_w| < 1.
  \end{align*}
\end{enumerate}
The properties stated so far do not exclude the possibility that
$T^+$ is the identity mapping (implying the same for $T^-$ by
\eqref{eq:T_+_T_-_relation}). The next property shows that
$T^+(\varphi) - \varphi$ is close to $\tau$ under certain
restrictions on the set of edges on which $\varphi$ changes by more
than $1-\eps$. We require a few definitions.

Let $d_G$ stand for the graph distance in $G$.
The next two definitions concern the Lipschitz properties of $\tau$.
\begin{align}
\tau'(v,k)&:=\max\{\tau(v) - \tau(w)\colon w\in V,\, d_G(v,w)\le
k\}\label{eq:tau_prime_def},\\
L(\tau,\eps)&:=\max\left\{k\colon \forall v\in V,\,
\tau'(v,k)\le\frac{\eps}{2}\right\} - 1.\label{eq:L_tau_eps_def}
\end{align}
In the following definitions we consider the connectivity properties
of the subset of edges on which $\varphi$ changes by more than
$1-\eps$. For $\varphi\in\R^V$ define
\begin{equation}\label{eq:EC_def}
  \EC(\varphi):=\{(v,w)\in E\colon |\varphi_v-\varphi_w|\ge
  1-\eps\}
\end{equation}
and write, for a pair of vertices $v,w\in V$,
\begin{equation}\label{eq:EC_connectivity_def}
v\xleftrightarrow{\EC(\varphi)} w\;\text{ if $v$ is connected to $w$
by edges of $\EC(\varphi)$},
\end{equation}
where we mean in particular $v\xleftrightarrow{\EC(\varphi)} v$ for all $v\in V$. Let
\begin{align}\label{def:rAndm}
  r(\varphi, v)&:= \max\{d_G(v,w)\colon w\in V,\,
  v\xleftrightarrow{\EC(\varphi)}w\},\\
  M(\varphi)&:=\max\{d_G(v,w)\colon v,w\in V,\,v\xleftrightarrow{\EC(\varphi)}w\}.\nonumber
\end{align}
\begin{enumerate}
\setcounter{enumi}{3}
  \item\label{property:addition_lower_bound} If $\varphi$ satisfies $M(\varphi)\le L(\tau,\eps)$ then
  \begin{equation*}
    \forall v\in V,\quad T^+(\varphi)_v-\varphi_v = \varphi_v - T^-(\varphi)_v\ge \tau(v) -
    \frac{\eps}{2}.
  \end{equation*}
\end{enumerate}
Together with property~\eqref{property:maximal_increment} above this
shows that $T^+(\varphi) - \varphi$ and $\varphi - T^-(\varphi)$ are
approximately equal to $\tau$ when $M(\varphi)\le L(\tau,\eps)$. A
slightly stronger property is given in
Proposition~\ref{prop:bounding_shifts} below.

Our final property regards the change of measure induced by the
mappings $T^+$ and $T^-$. We bound the Jacobians of these mappings
when the subgraph $\EC(\varphi)$ does not contain many large
connected components.

Partition the vertex set $V$ into $V_0$ and $V_1$ by letting
\begin{equation}
  V_0:=\{v\in V\colon \tau(v) = 0\}\quad\text{and}\quad V_1:=V\setminus
  V_0.\label{eq:V_0_V_1_partition}
\end{equation}
Given a function $\theta:V_0\to\R$ we write
\begin{equation}\label{eq:Lebesgue_with_delta_measure_properties_section}
  d\mu_\theta(\varphi):=\prod_{v\in V_1} d\varphi_v \prod_{v\in V_0}
  \delta_{\theta_v}(d\varphi_v)
\end{equation}
for the measure on $\R^V$ given by product Lebesgue measure on the
subspace where $\varphi_v = \theta_v$, $v\in V_0$.
\begin{enumerate}
\setcounter{enumi}{4}
  \item\label{property:Jacobian_estimate} There are measurable functions $J^+:\R^V\to[0,\infty)$ and
  $J^-:\R^V\to[0,\infty)$ satisfying that for every
  $\theta:V_0\to\R$ and every $g:\R^V\to[0,\infty)$, integrable with respect to $d\mu_\theta$,
  \begin{equation}\label{eq:Jacobian_formula_properties_section}
\int g(T^+(\varphi))J^+(\varphi) d\mu_\theta(\varphi) = \int
g(T^-(\varphi))J^-(\varphi) d\mu_\theta(\varphi) = \int g(\varphi)
d\mu_\theta(\varphi).
\end{equation}
  Moreover, if $\varphi$ satisfies $M(\varphi)\le L(\tau,\eps)$ then
  \begin{equation*}
    \sqrt{J^+(\varphi)J^-(\varphi)}\ge \exp\left(-\frac{1}{\eps^2}\sum_{v\in V} \tau'\left(v,1+\max_{w\sim
v}r(\varphi,w)\right)^2\right).
  \end{equation*}
\end{enumerate}

\subsection{Description of the addition
algorithm}\label{sec:description_of_T^+} In this section we define
the mapping $T^+$ whose properties were discussed in the previous
section.

Let the graph $G=(V,E)$, function $\tau$ and constant $\eps$ be as
above. Fix an arbitrary total order $\preceq$ on the vertex set $V$. Define a Lipschitz ``bump'' function $f:\R\to\R$ by
\begin{equation}\label{eq:bump_function_def}
  f(x):=\begin{cases}
    0 & x\in(-\infty,-1]\\
    \frac{1+x}{\eps} & x\in[-1,-1+\eps]\\
    1 & x\in[-1+\eps,1-\eps]\\
    \frac{1-x}{\eps} & x\in[1-\eps,1]\\
    0 & x\in[1,\infty)
  \end{cases}.
\end{equation}
We also define a family of shifted and rescaled versions of $f$. For
a vertex $v\in V$ and $h,t\in\R$ let
\begin{equation}\label{eq:m_alternative_def}
  m_{v,h,t}(h'):= \begin{cases}
    \min\left(\tau(v) - t, \frac{\eps}{2}\right)f(h'-h)+t, &\text{ if }\tau(v)\geq t\\
    t, &\text{ if } \tau(v)<t
  \end{cases}.
\end{equation}
One should have in mind the case $\tau(v)\geq t$ and think of $m_{v,h,t}$ as being the same as $f$, scaled and
shifted to have maximum $\tau(v)$, minimum $t$ and to have its
``center'' at $h$. However, if the function just described has
Lipschitz constant more than $1/2$, we lower its maximum so that its
Lipschitz constant becomes $1/2$. For easy reference we record this
as
\begin{equation}\label{eq:m_Lipschitz_constant}
  \text{the function $m_{v,h,t}$ has Lipschitz constant at most
  $\frac{1}{2}$}.
\end{equation}
The case $\tau(v)<t$ is not used in the definition of $T^+$ below.
It is included here as it is technically convenient in the analysis
to have $m_{v,h,t}$ defined for all values of the parameters.

The definition of $T^+$ is based on the following algorithm. The
algorithm takes as input a function $\varphi\in\R^V$. It outputs
three sequences indexed by $1\le k\le |V|$:
\begin{enumerate}
\item A sequence $(P_k)$ which is a ordering of the vertices $V$,
that is, $\{P_k\} = V$.
\item A sequence $(s_k)\subseteq[0,\infty)$ with $s_k$ representing the amount to add to $\varphi$ at vertex
$P_k$.
\item A sequence $(\tau_k)$ of functions, $\tau_k:V\times\R\to\R$, which
will play a role in analyzing the Jacobian of the mapping $T^+$.
\end{enumerate}
The mapping $T^+$ is then defined by
\begin{equation}\label{eq:T_plus_def}
  T^+(\varphi) :=
\tilde{\varphi}\;\;\text{ with }\;\;\tilde{\varphi}_{P_k} :=
\varphi_{P_k} + s_k,\quad 1\le k\le |V|.
\end{equation}
\textbf{Addition algorithm:}\\
\noindent \underline{Initialization}. Set $\tau_1(v,h):=\tau(v)$ for all $v\in V$ and $h\in\R$.\\
\noindent \underline{Loop}. For $k$ between $1$ and $|V|$ do:
\begin{enumerate}
\item \label{alg:P_k_choice} Set $P_k$ to be the vertex $v$ in $V\setminus \{P_1, \ldots,
P_{k-1}\}$ which minimizes $\tau_k(v,\varphi_v)$. If there are
multiple vertices achieving the same minimum let $P_k$ be the
smallest one with respect to the total order $\preceq$.
\item \label{alg:s_k_def} Set $s_k:=\tau_k(P_k,\varphi_{P_k})$.
\item \label{alg:update_step} If $k<|V|$ set, for each $v\in V$ and
$h\in\R$,
\begin{equation}\label{eq:algorithm_update_tau}
\tau_{k+1}(v,h):=\begin{cases} \tau_{k}(v,h) & \text{ if
$v\in\{P_1,\ldots, P_k\}$ or $v\not \sim
P_k$}\\
\min(\tau_{k}(v,h), m_{v,\varphi_{P_{k}},s_k}(h)) & \text{ if
$v\notin\{P_1,\ldots, P_k\}$ and }v\sim P_{k}
\end{cases}.
\end{equation}
\end{enumerate}
 \newcommand{\scT}{0.4}

\begin{adjustwidth}{-1.0cm}{}   

\begin{tabular}{|c|c|c|}    
    \multicolumn{3}{c}{Table 1: An illustration of the action of the addition algorithm on a function defined on a 2x3 grid graph.}\\
\hline 
\begin{sideways}
\parbox{6cm}{
 \underline{Initialization.} The vertices are assigned initial requested shifts, $\tau_1:=\tau$.
\newline\newline 
\underline{Loop } 
\begin{enumerate}
	\item[(1)] The green (gray in b\&w) vertex is set to be $P_1$ and to be processed.
	\item[(2)] It is shifted by $s_1 := 0$.
	\item[(3)] The requested shifts are updated. In this example for all $w$: $\tau_2(w,\varphi_w) := \tau_1(w,\varphi_w)$.
\end{enumerate}

}      $\:$       
\end{sideways} &
\begin{sideways}
\parbox{6cm}{
 	\underline{Loop } 
	\begin{enumerate}
		\item[(1)] The green (gray in b\&w)  vertex is set to be $P_2$ and to be processed.
		\item[(2)] It is shifted by $s_2 := 0.2$.
		\item[(3)] The requested shifts are updated. For the top-most vertex $v$: {\small $\tau_3(v,\varphi_v):=0.6<1.1=\tau_2(v,\varphi_v)$.} For all other $w$:   $\tau_3(w,\varphi_w) := \tau_2(w,\varphi_w)$.
	\end{enumerate}  
	{\small The next vertex to be processed need not be adjacent to the previously processed vertices.}
}   
\end{sideways}
  & 
\begin{sideways}
\parbox{6cm}{
   	\underline{Loop } 
	\begin{enumerate}
		\item[(1)] The green (gray in b\&w)  vertex is set to be $P_3$ and to be processed.
		\item[(2)] It is shifted by $s_3 := 0.3$.
		\item[(3)] The requested shifts are updated. For the top-center vertex $v$: {\small $\tau_4(v,\varphi_v):=0.69<1.2=\tau_3(v,\varphi_v)$.} For all other $w$:   $\tau_4(w,\varphi_w) := \tau_3(w,\varphi_w)$.
	\end{enumerate}
}
\end{sideways}

\tabularnewline
\hline 
\includegraphics[angle=90,scale=\scT]{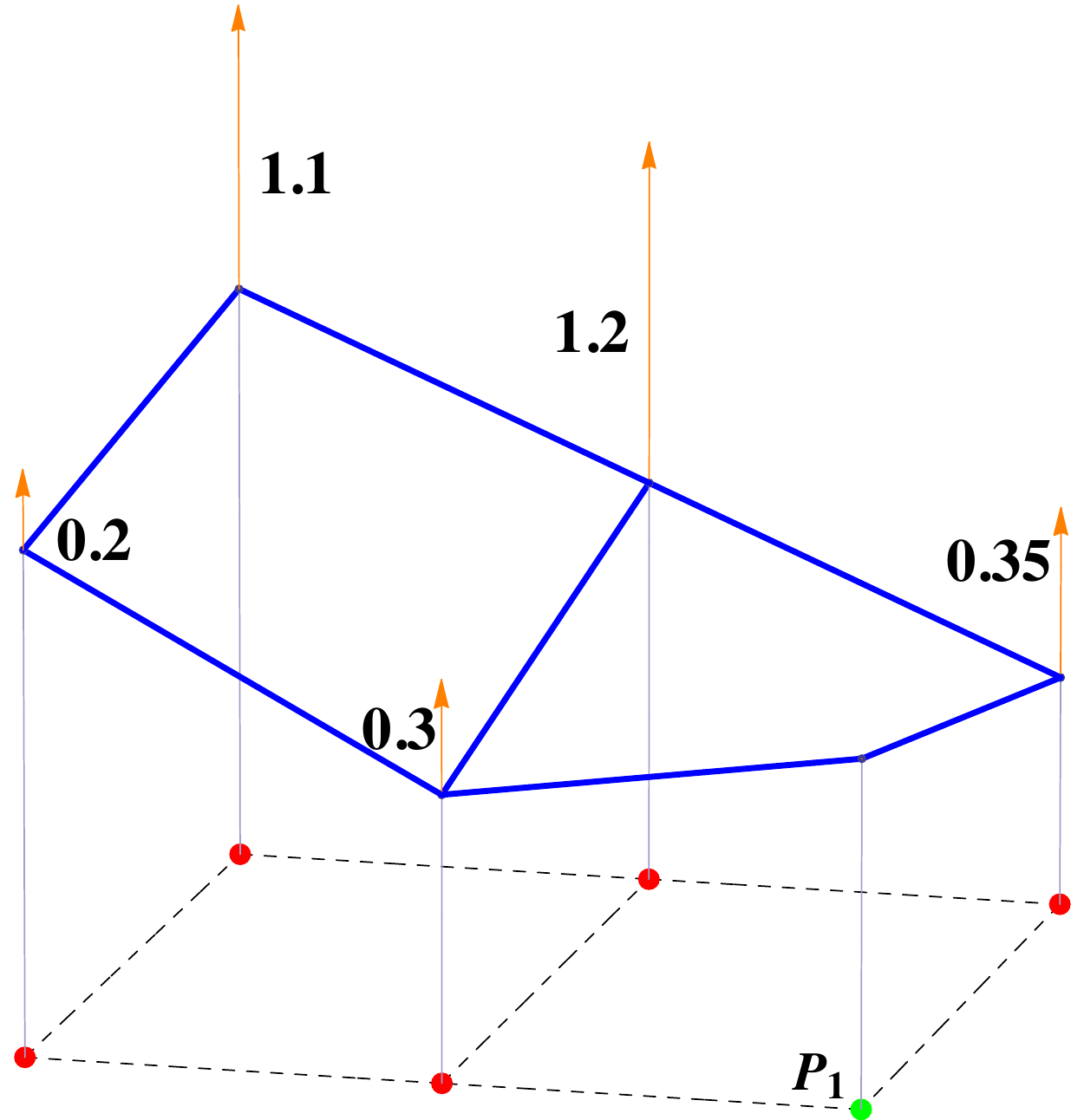} & \includegraphics[angle=90,scale=\scT]{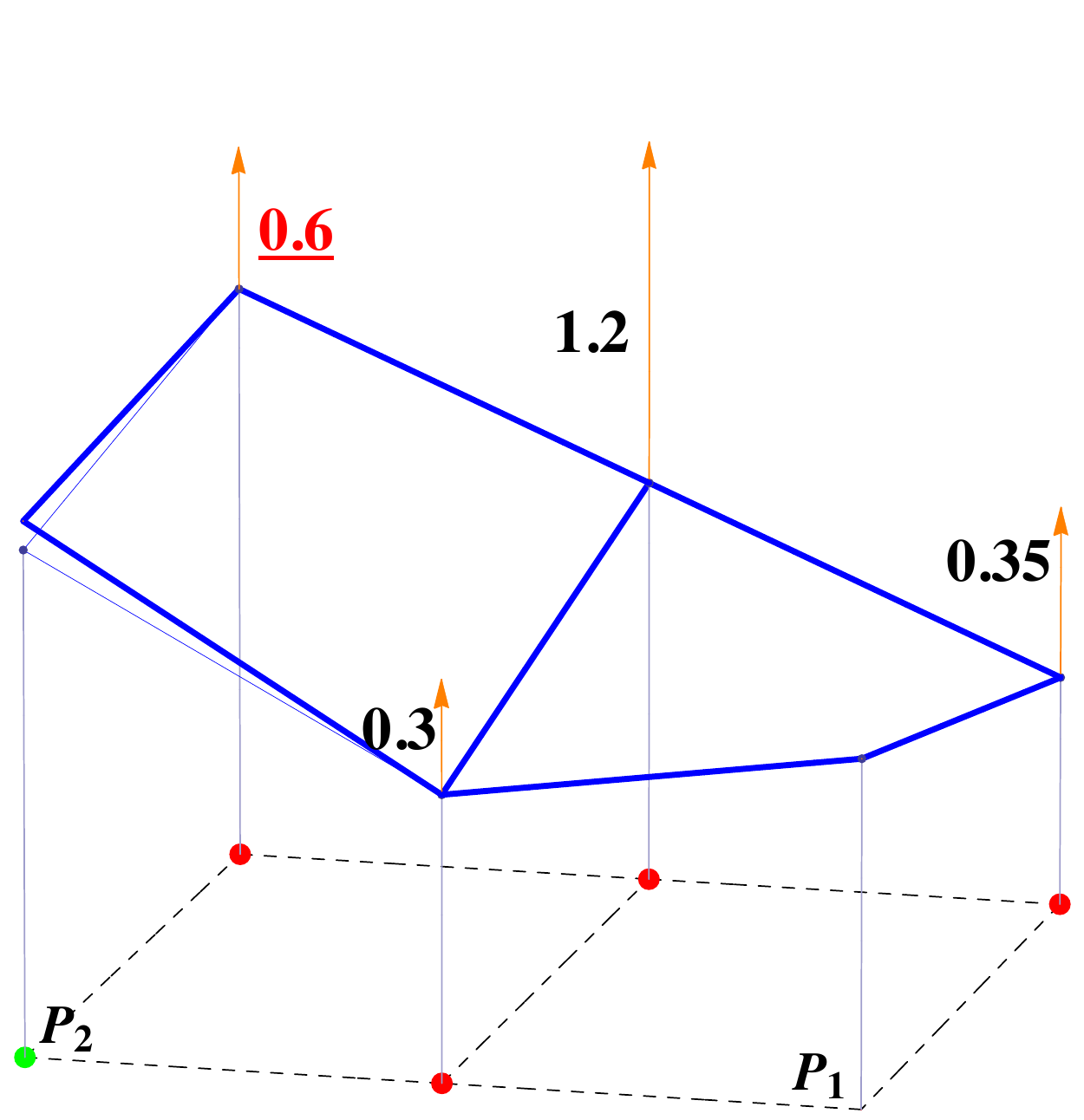} & \includegraphics[angle=90,scale=\scT]{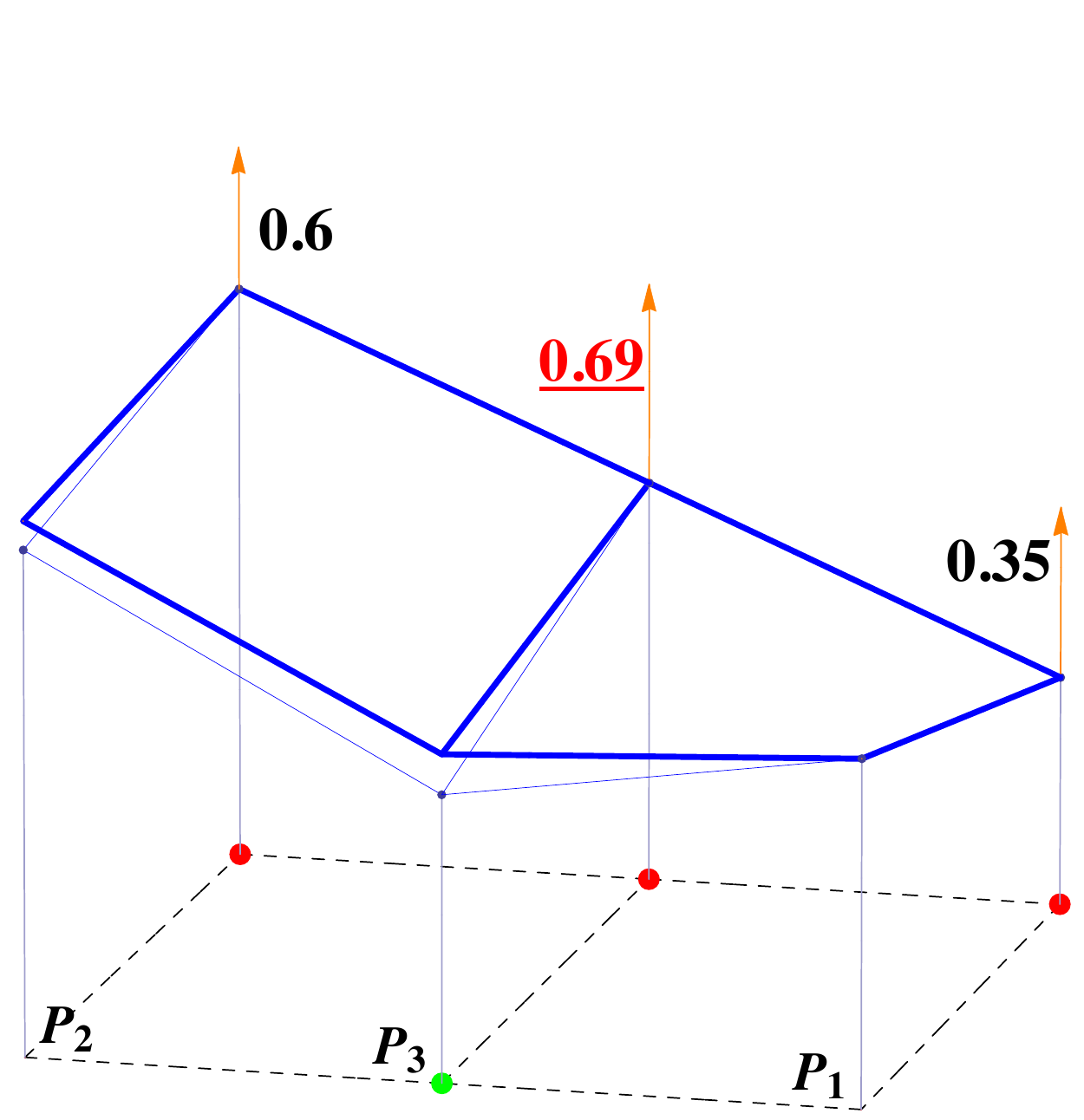}\\
\hline
\includegraphics[angle=90,scale=\scT]{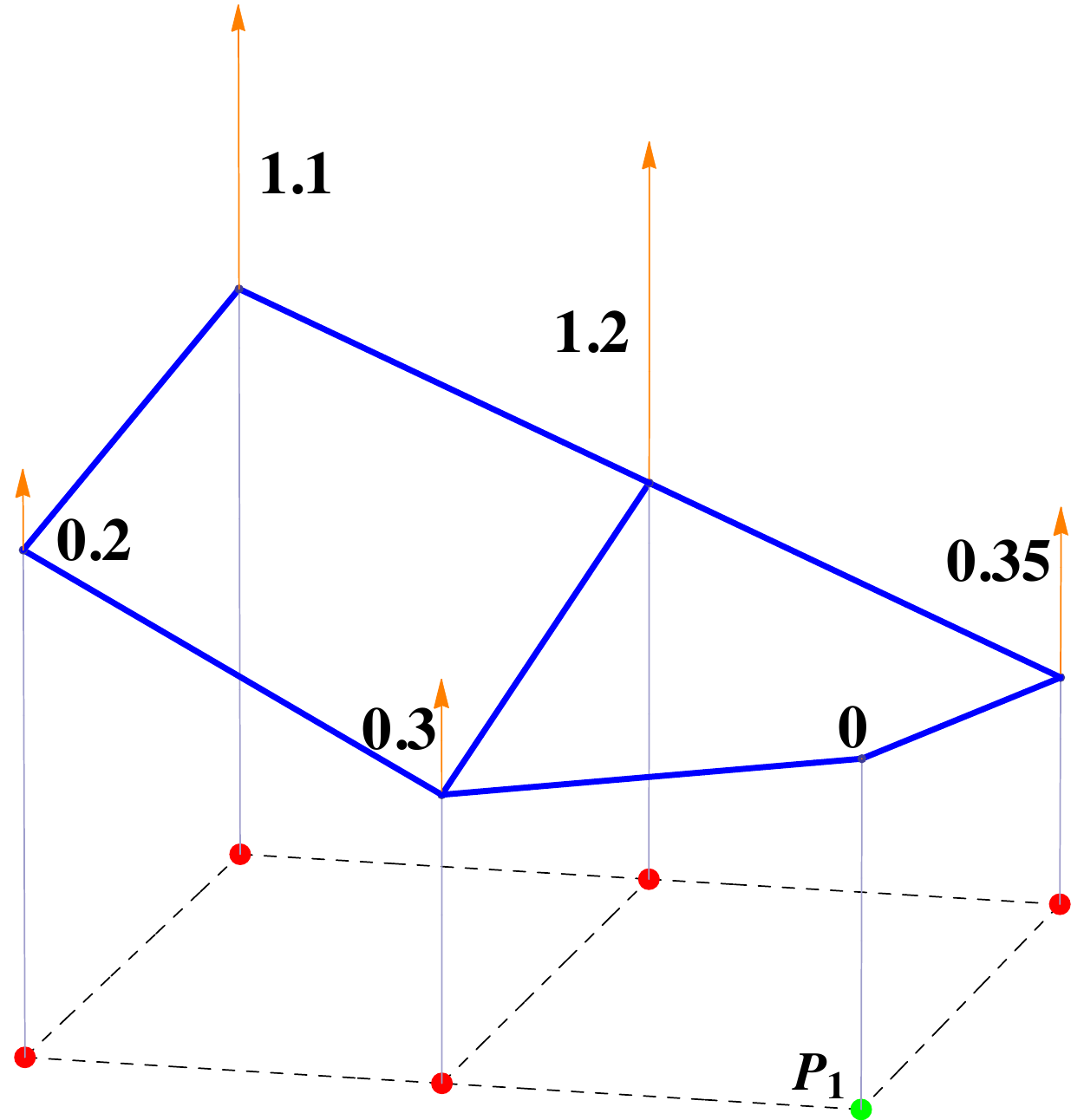} & \includegraphics[angle=90,scale=\scT]{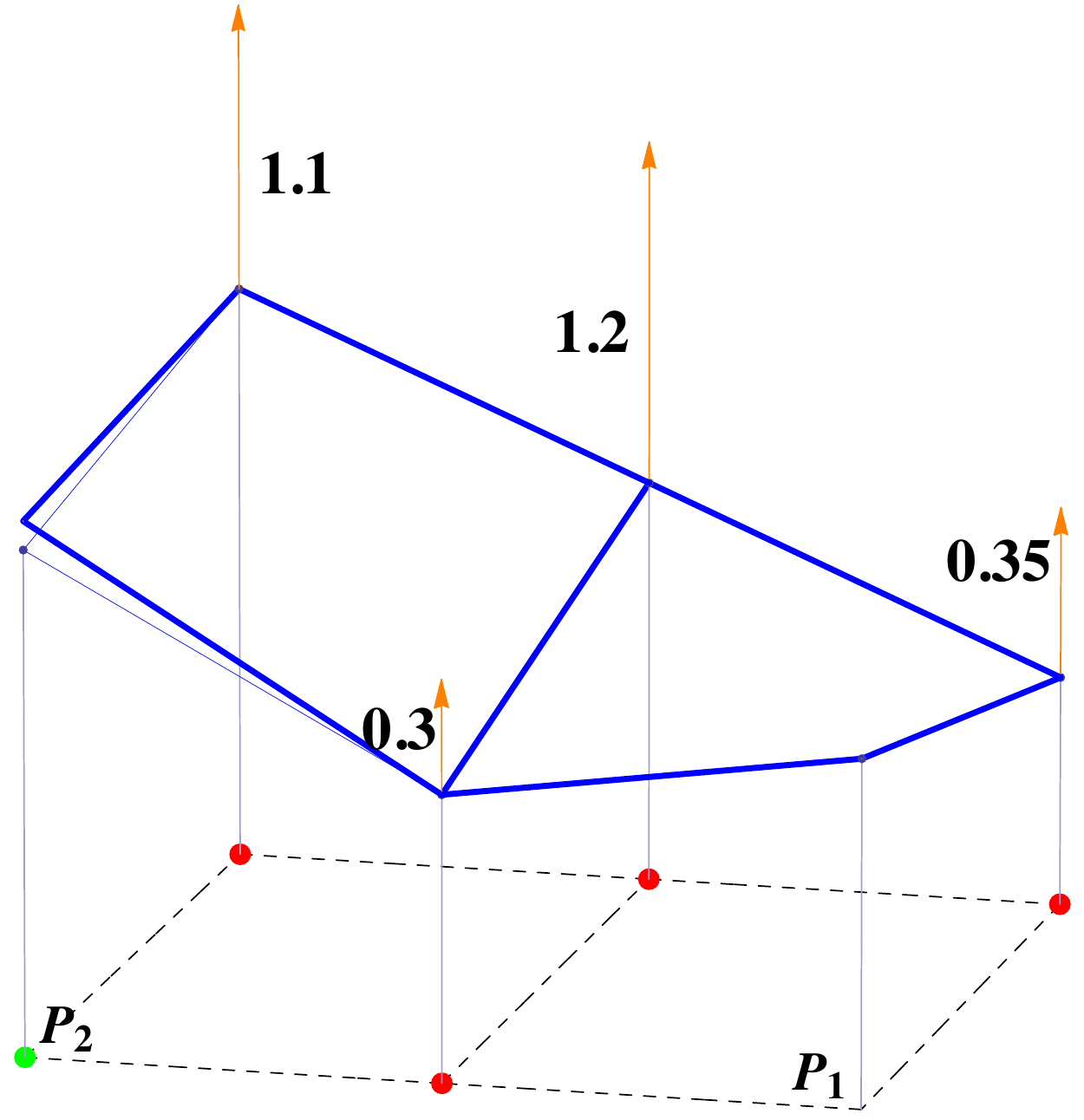} & \includegraphics[angle=90,scale=\scT]{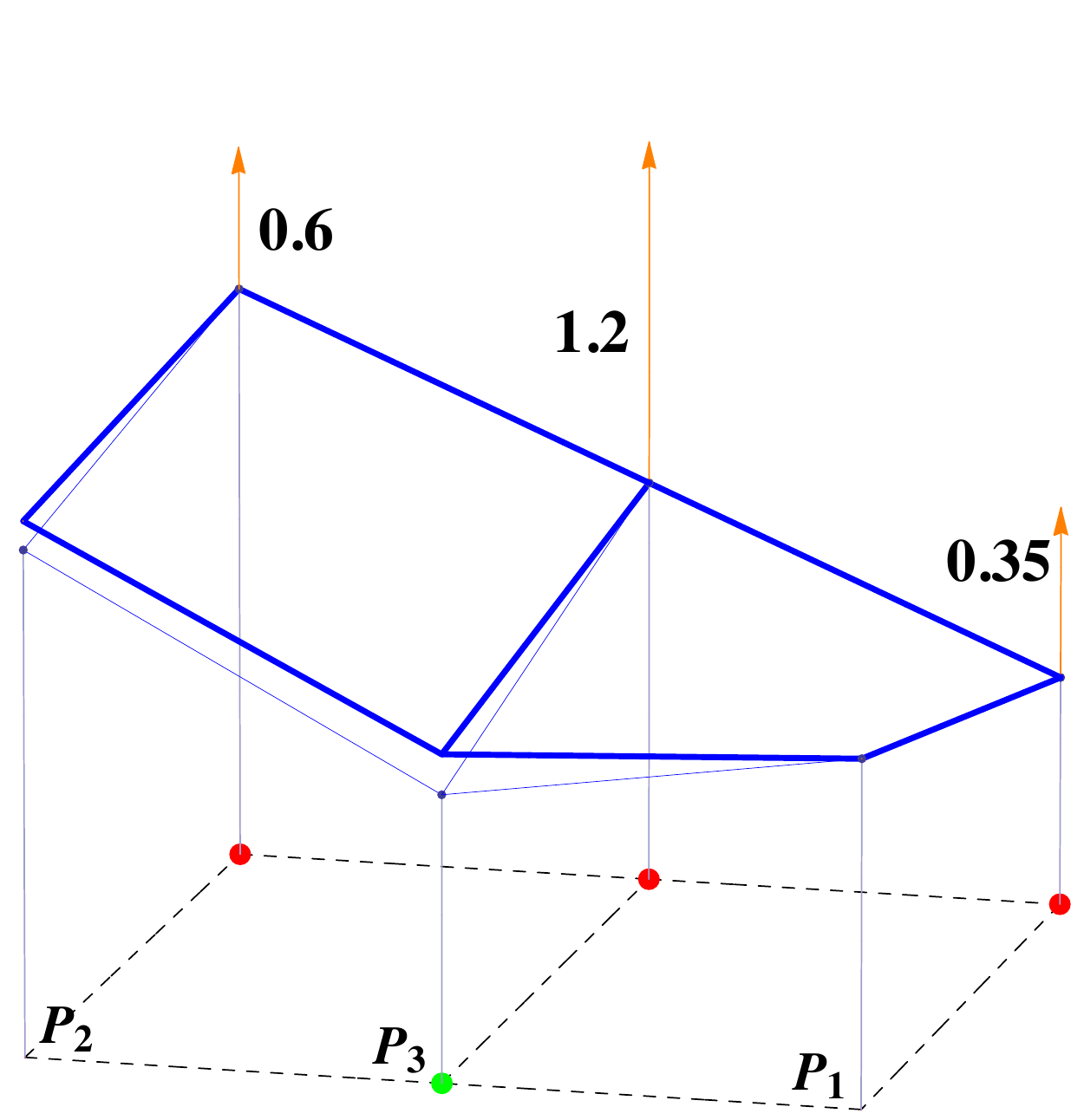}\\
\hline
\includegraphics[angle=90,scale=\scT]{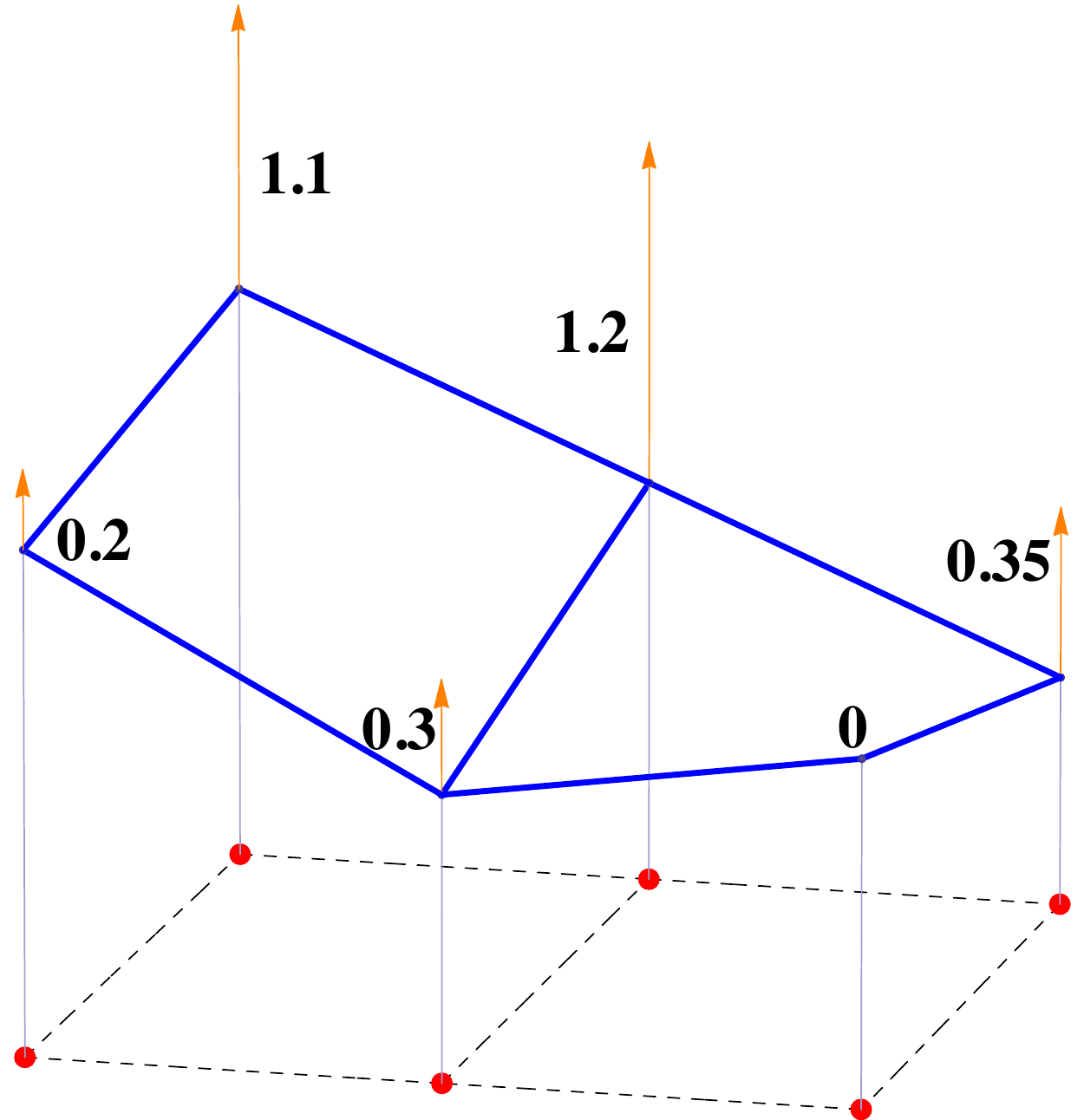} & \includegraphics[angle=90,scale=\scT]{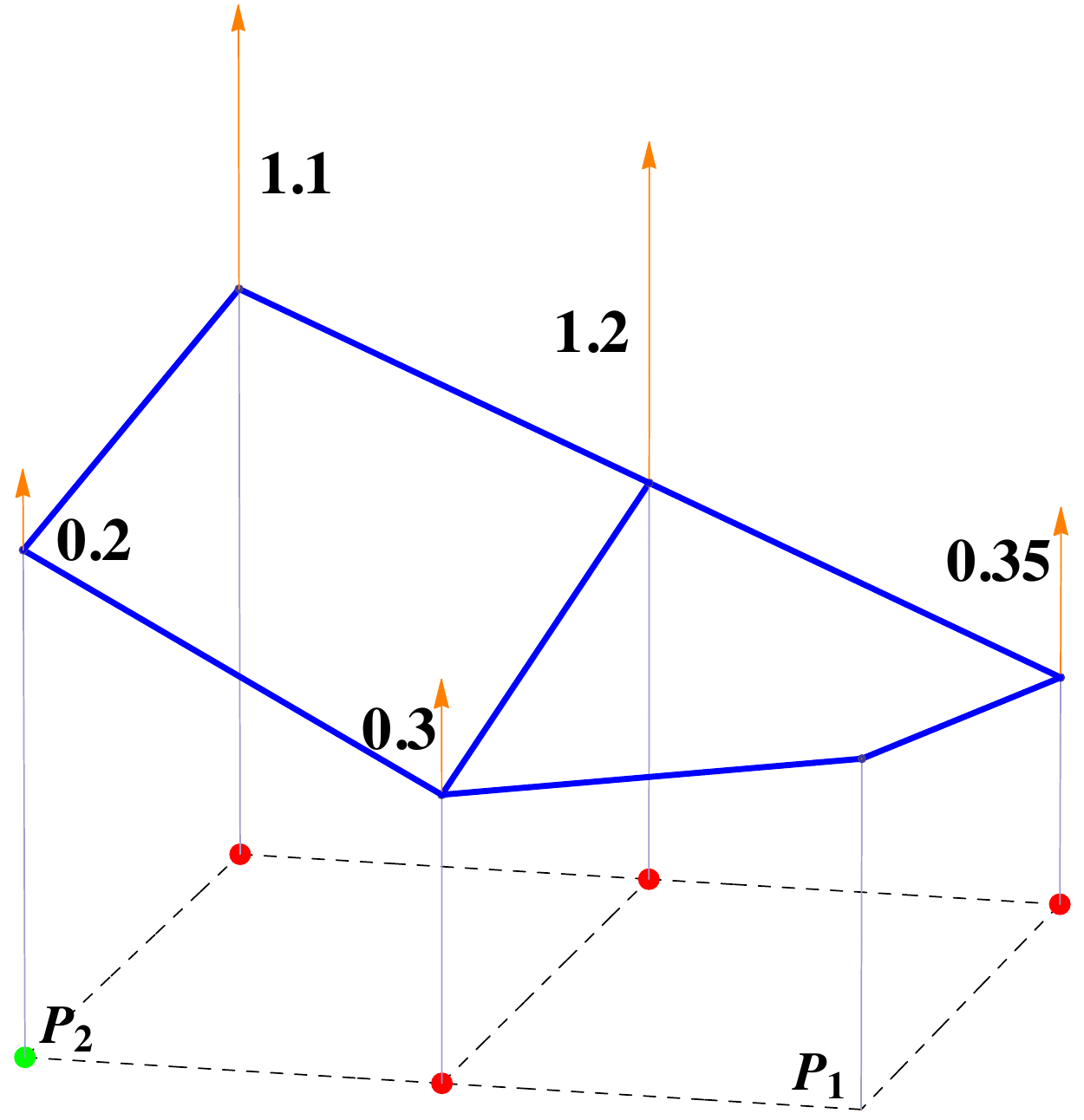} & \includegraphics[angle=90,scale=\scT]{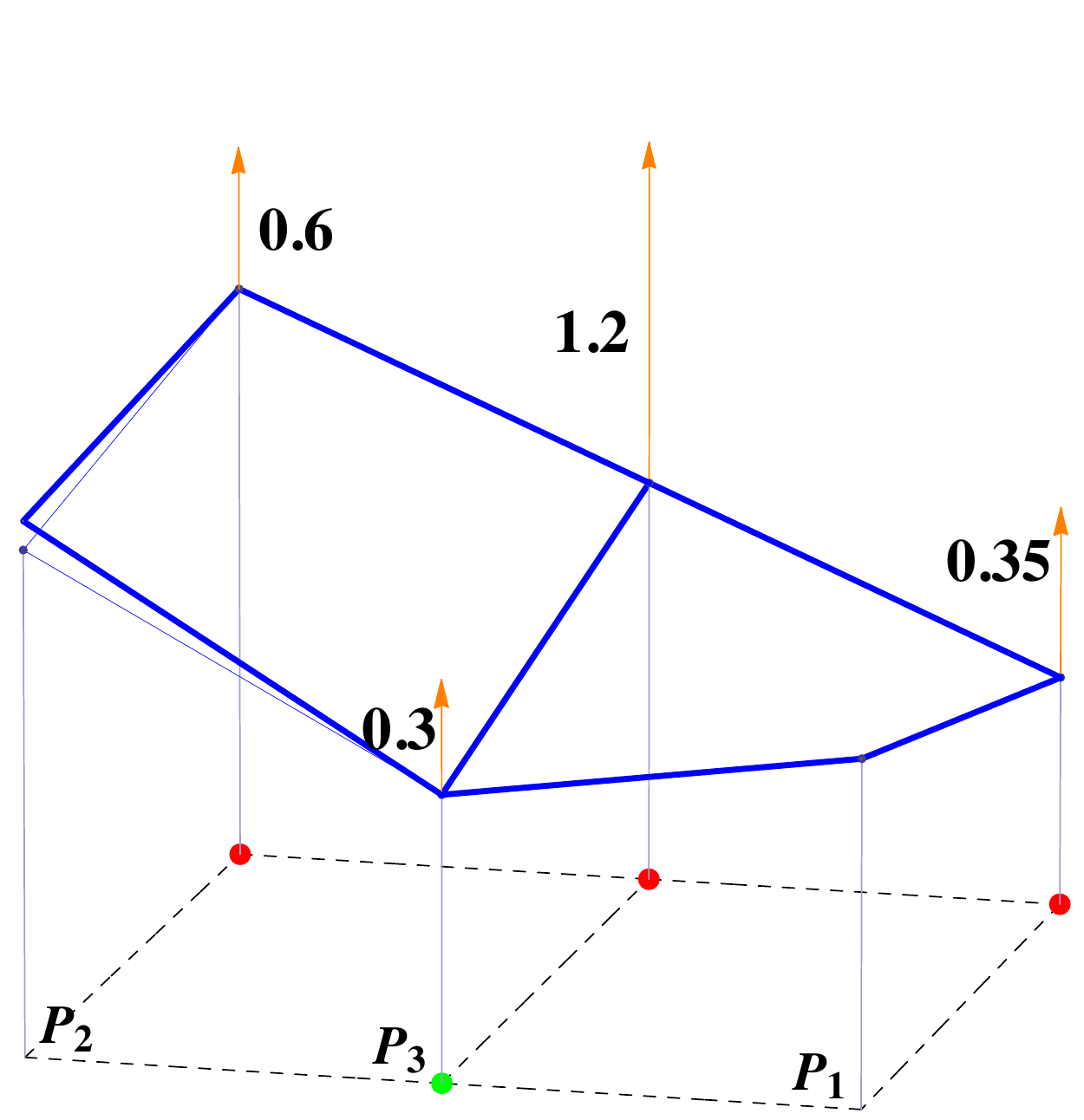}\\
\hline 
\end{tabular}
\newpage 
\begin{tabular}{|c|c|c|}
   \multicolumn{3}{c}{Table 1: An illustration of the action of the addition algorithm (cont.)} \\
\hline
\begin{sideways}
\parbox{6cm}{
 	\underline{Loop } 
	\begin{enumerate}
		\item[(1)] The green (gray in b\&w)  vertex is set to be $P_4$ and to be processed.
		\item[(2)] It is shifted by $s_4 := 0.35$.
		\item[(3)] The requested shifts are updated. For the top-center vertex $v$: {\small $\tau_5(v,\varphi_v):=0.5<0.69=\tau_4(v,\varphi_v)$.} For all other $w$:   $\tau_5(w,\varphi_w) := \tau_4(w,\varphi_w)$.
	\end{enumerate}   
	{\small The requested shift of a vertex may be decreased several times.}

}   $\:$       
\end{sideways} &
\begin{sideways}
\parbox{6cm}{
 	 \underline{Loop } 
		\begin{enumerate}
			\item[(1)] The green (gray in b\&w)  vertex is set to be $P_5$ and to be processed.
			\item[(2)] It is shifted by $s_5 := 0.5$.
			\item[(3)] The requested shifts are updated. For the top-left vertex $v$: {\small $\tau_6(v,\varphi_v):=0.55<0.6=\tau_5(v,\varphi_v)$.} For the processed vertices (all except $v$) $\tau_6:=\tau_5$.
		\end{enumerate}  
}
\end{sideways}
  & 
\begin{sideways}
\parbox{6cm}{
   	 \underline{Loop } 
		\begin{enumerate}
			\item[(1)] The green (gray in b\&w)  vertex is set to be $P_6$ and to be processed.
			\item[(2)] It is shifted by $s_6 := 0.55$.
		\end{enumerate}  
		\vspace{2pt}
		{\small  The consecutive shifts increase, \\$s_1 \leq s_2 \leq s_3 \leq s_4\leq s_5 \leq s_6.$\\}
		{~\\}
		\begin{center}
			The algorithm terminates!
		\end{center}
}
\end{sideways}

\tabularnewline
\hline 
\includegraphics[angle=90,scale=\scT]{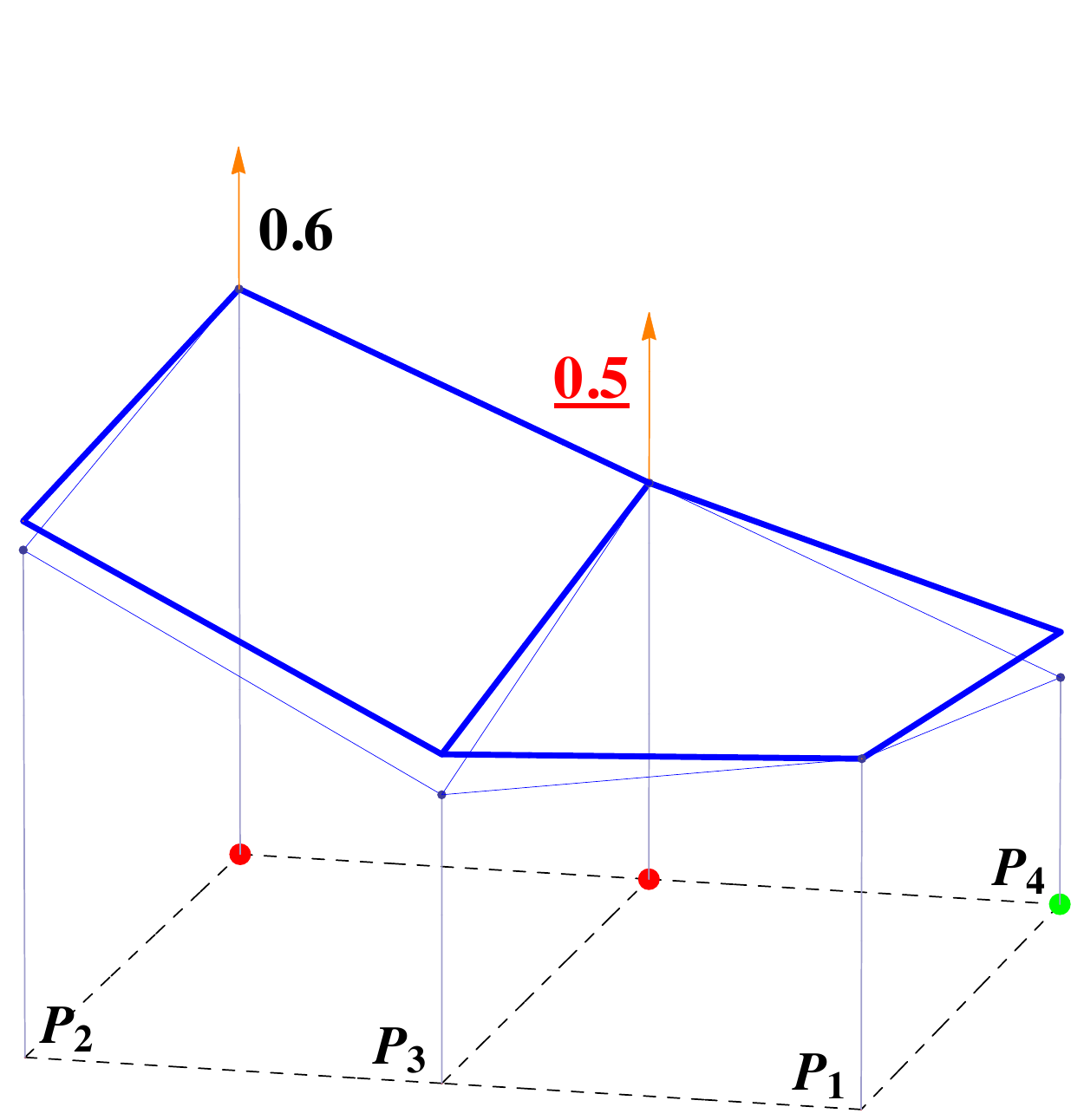} & \includegraphics[angle=90,scale=\scT]{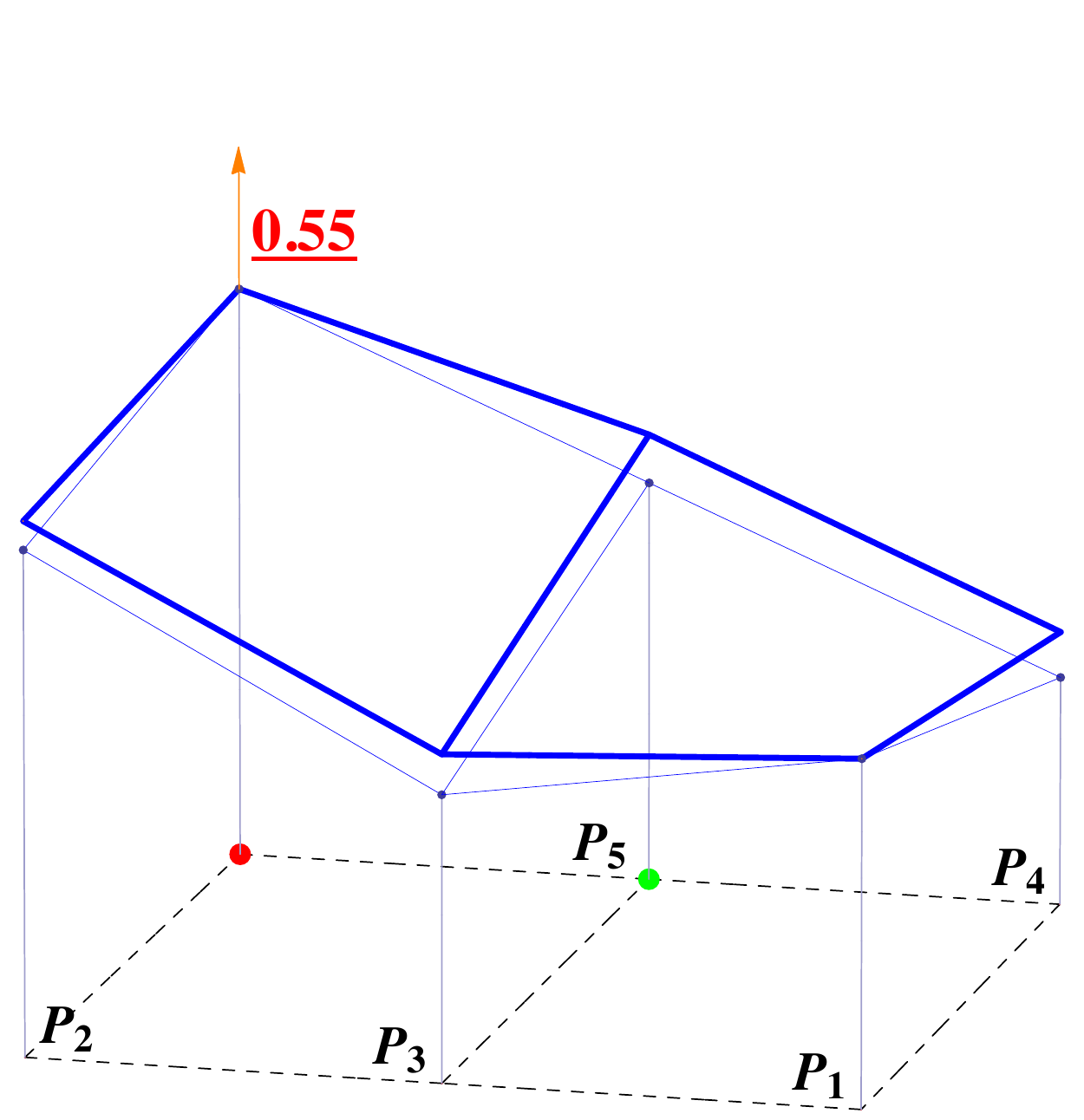} & \includegraphics[angle=90,scale=\scT]{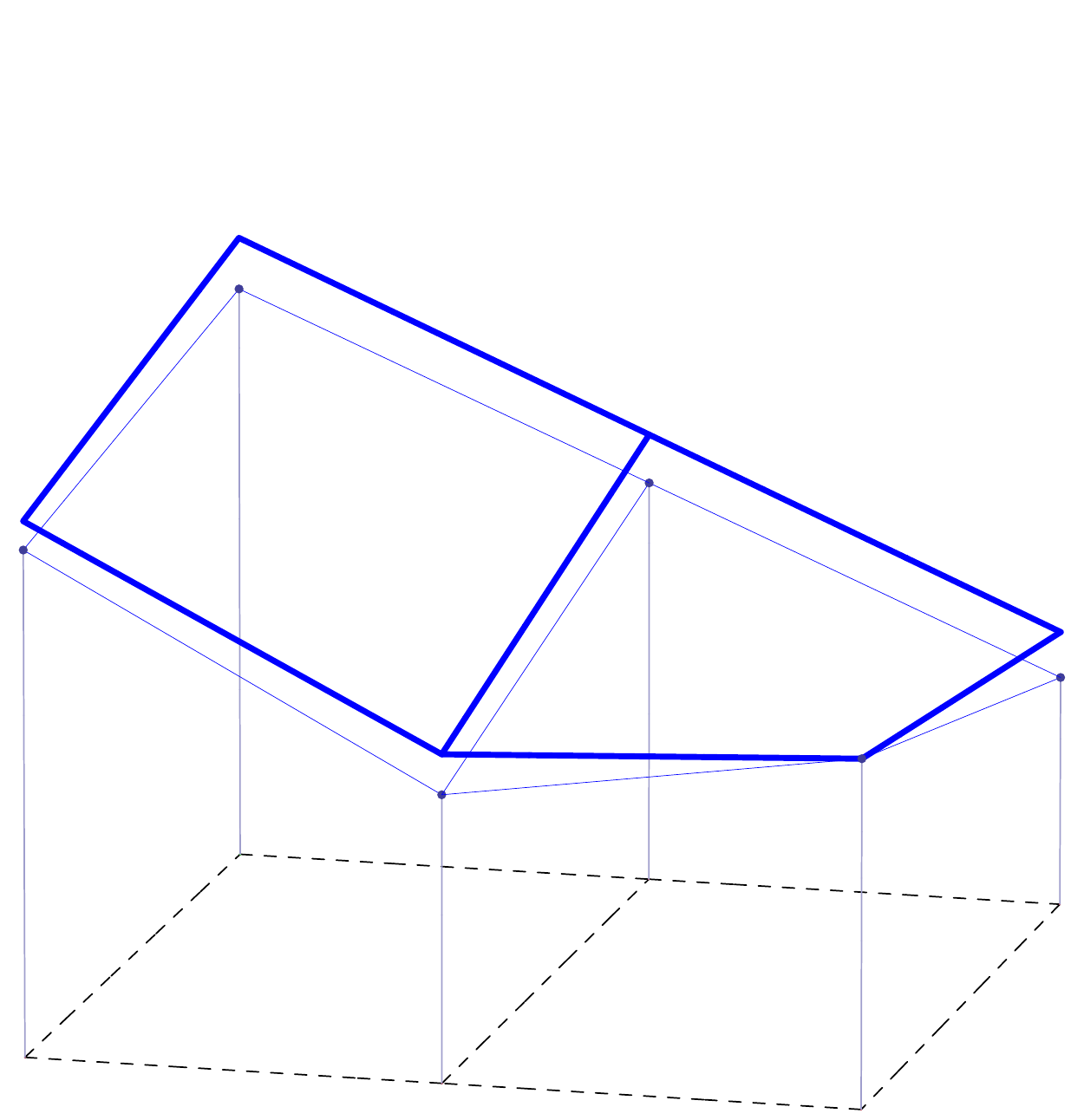}\\
\hline
\includegraphics[angle=90,scale=\scT]{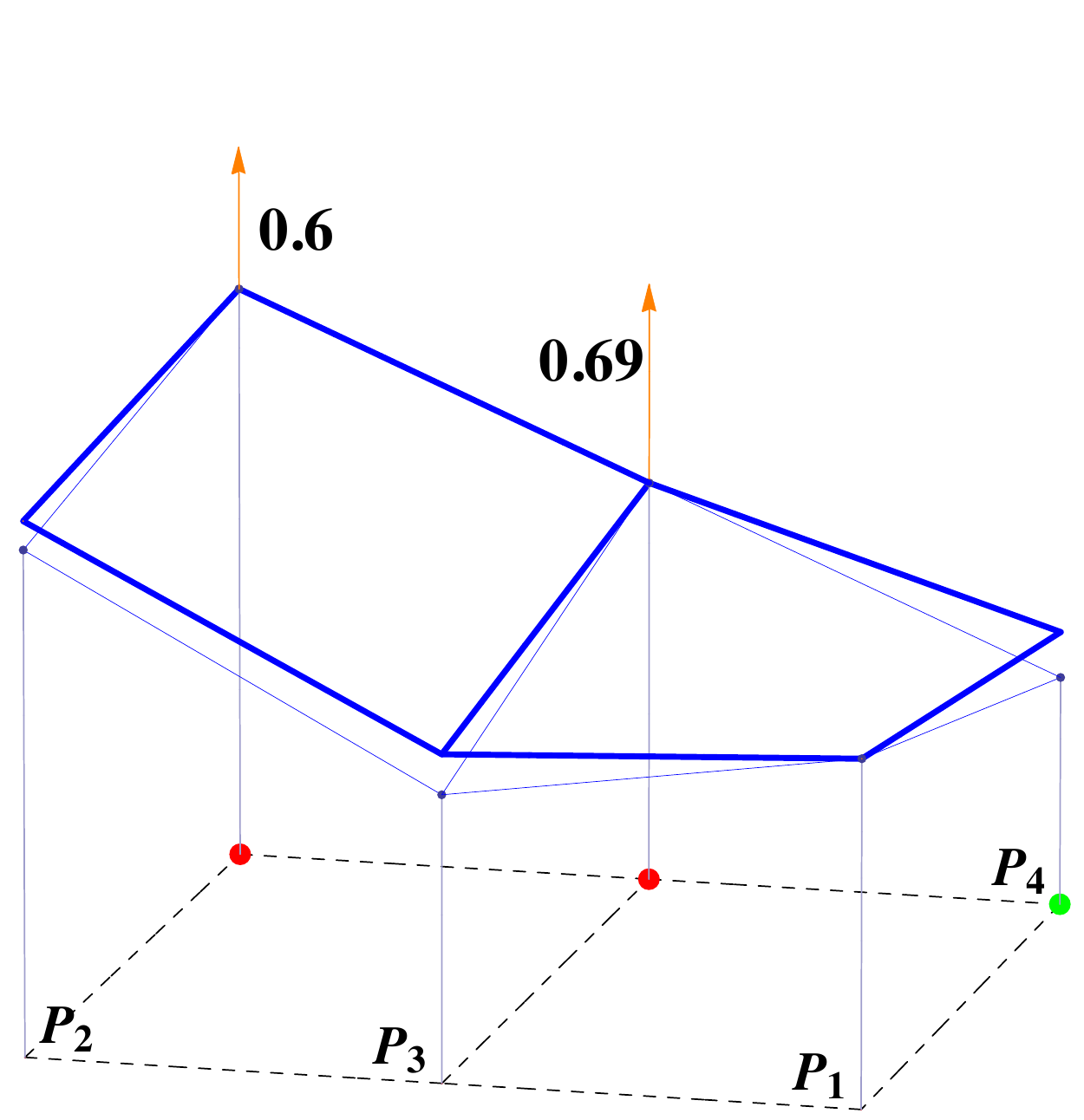} & \includegraphics[angle=90,scale=\scT]{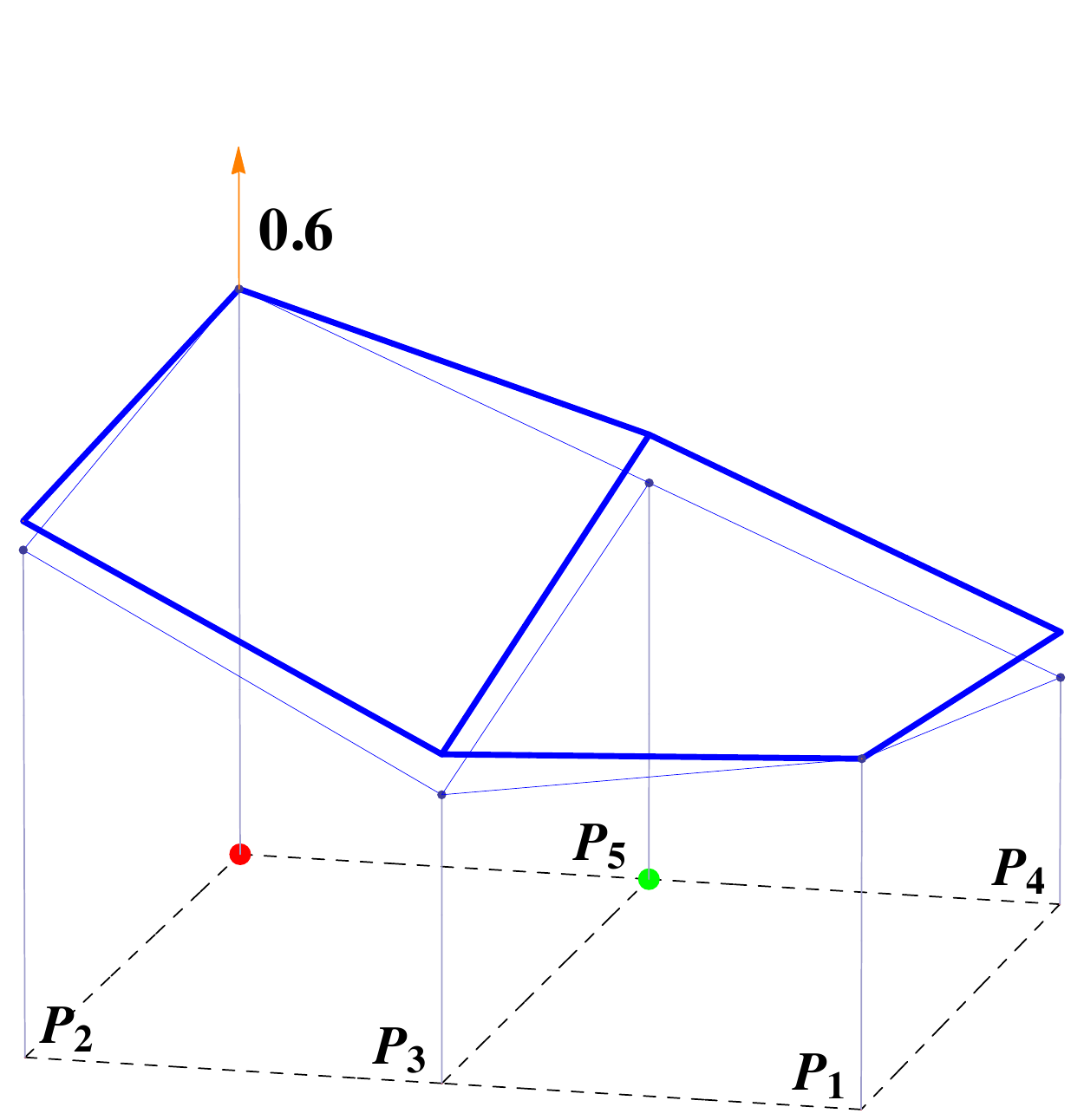} & \includegraphics[angle=90,scale=\scT]{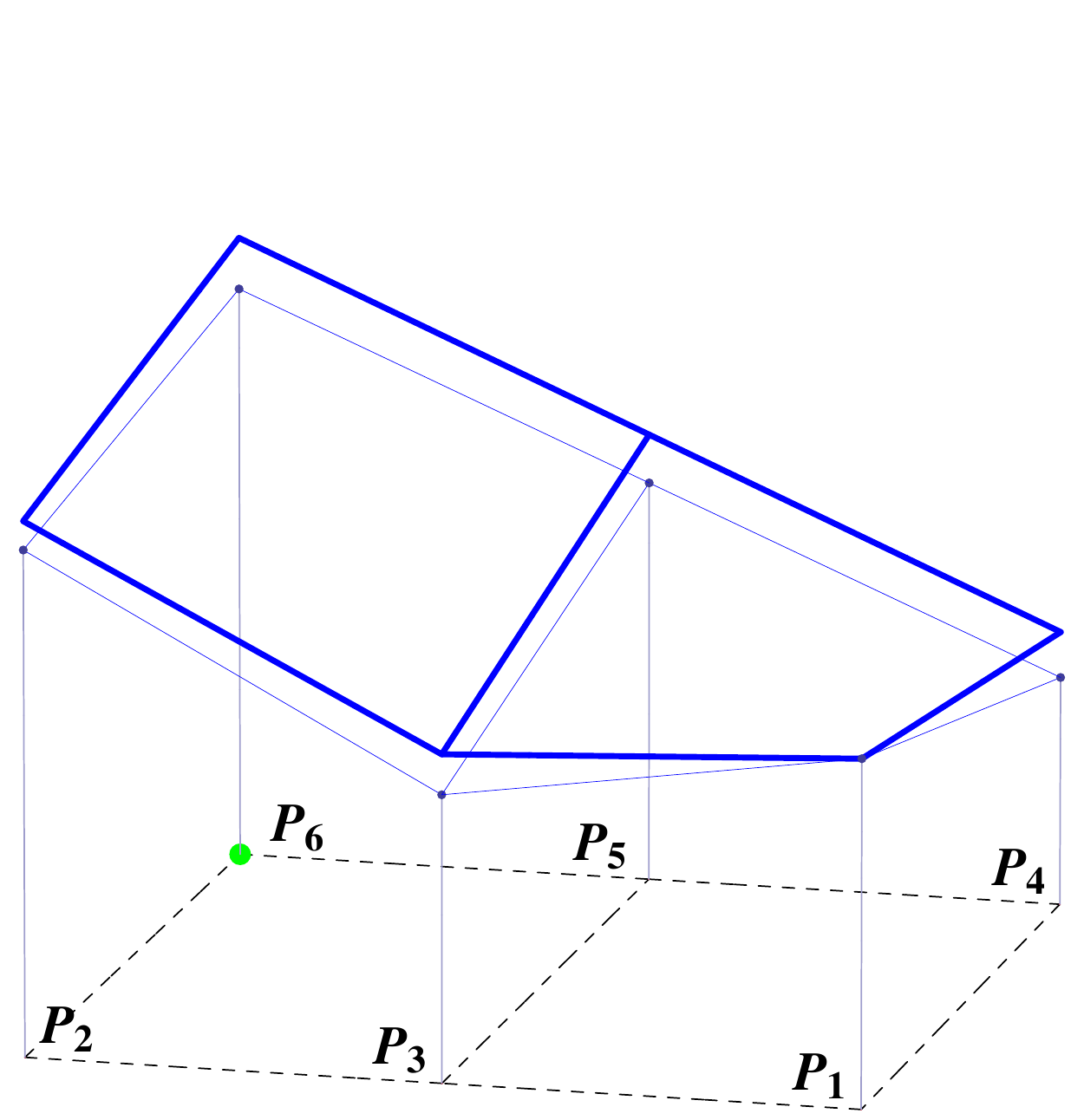}\\
\hline
\includegraphics[angle=90,scale=\scT]{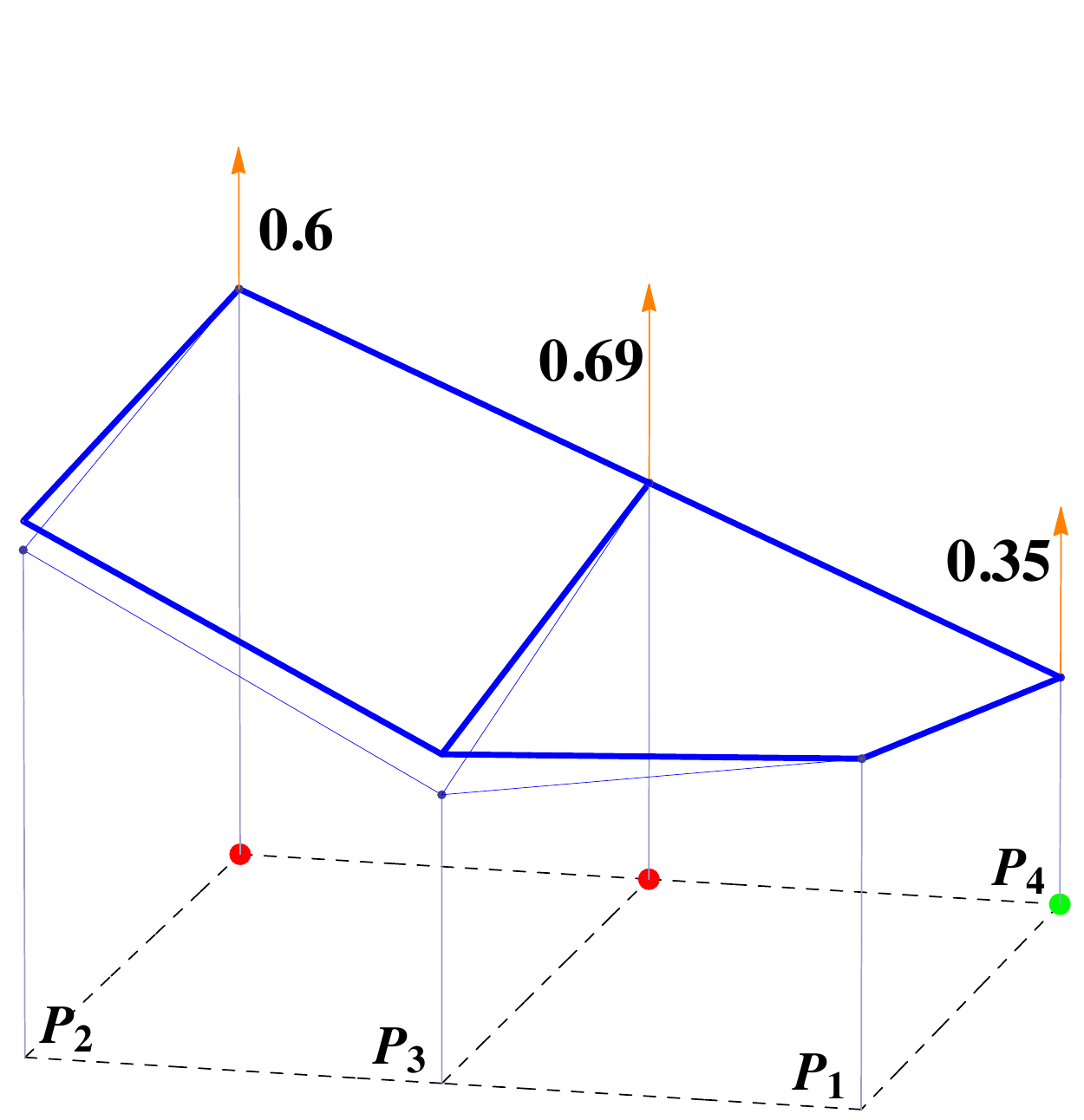} & \includegraphics[angle=90,scale=\scT]{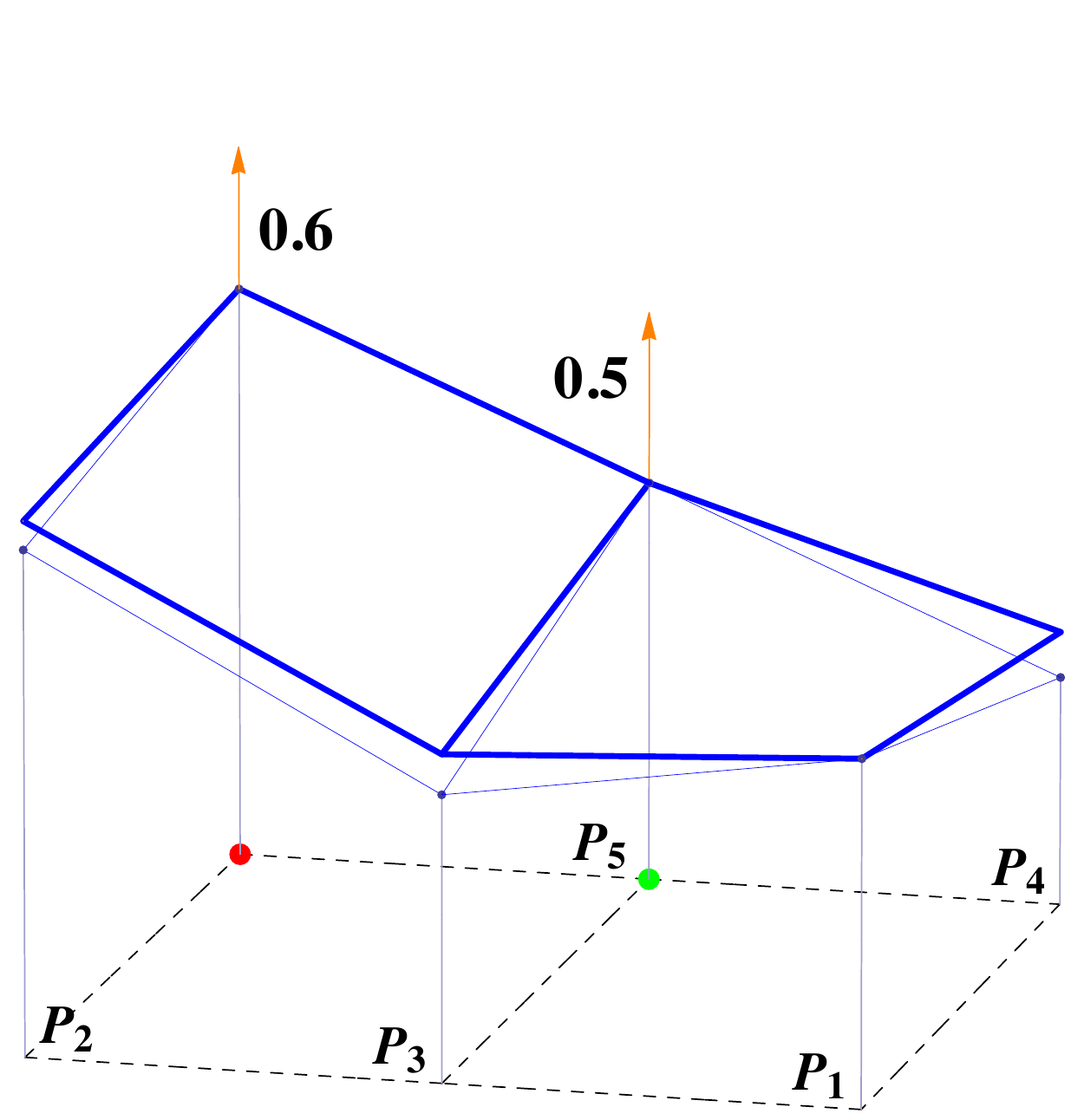} & \includegraphics[angle=90,scale=\scT]{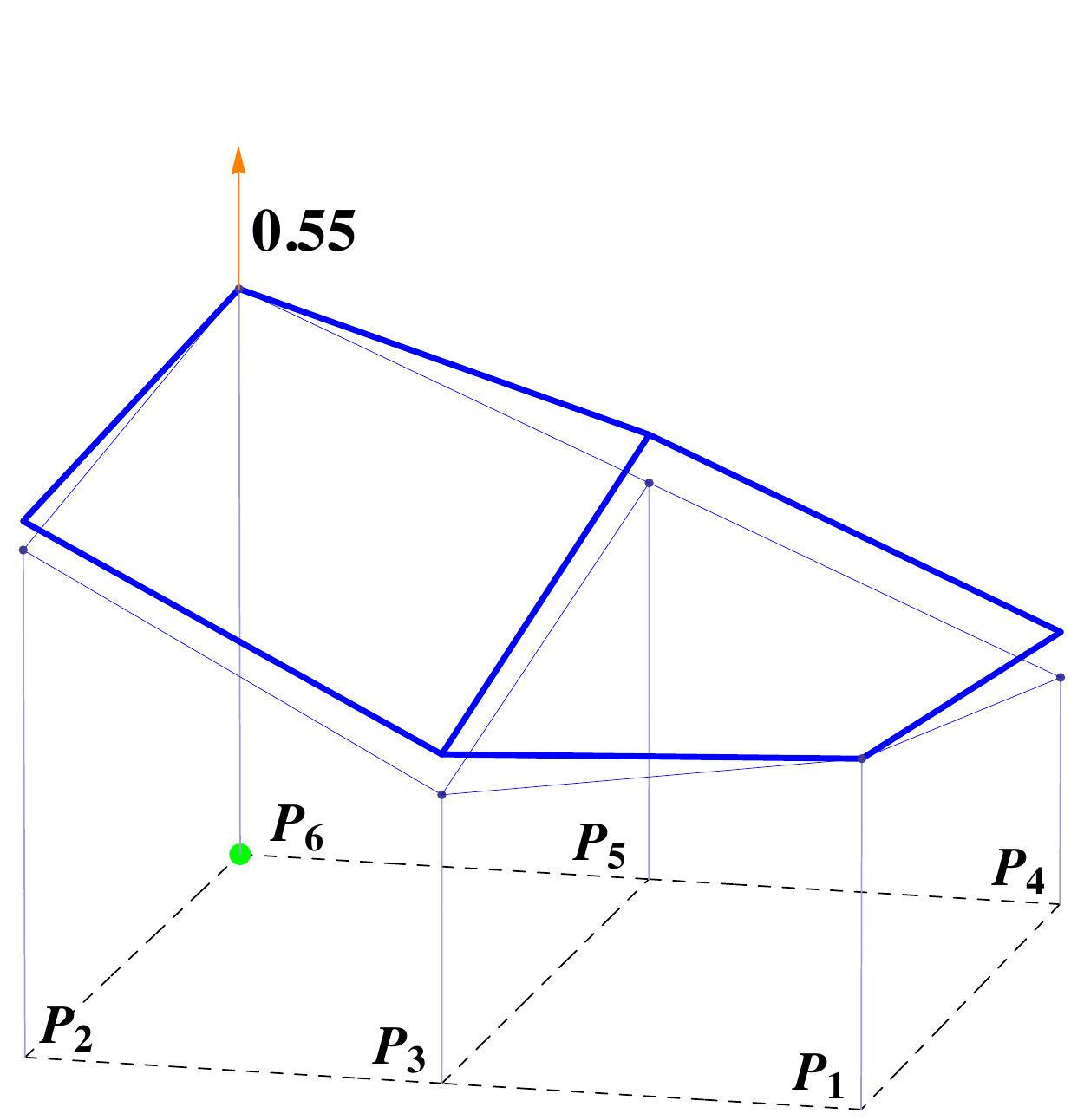}\\
\hline 
\end{tabular}   
  \end{adjustwidth}
In the next sections we verify that the mapping $T^+$ defined by
\eqref{eq:T_plus_def} satisfies the properties declared in
Section~\ref{sec:addition_algorithm_prop}.
\subsection{Increments and Lipschitz property}
In this section we verify
properties~\eqref{property:maximal_increment} and
\eqref{property:Lipschitz_preservation} from
Section~\ref{sec:addition_algorithm_prop} for $T^+$. Property
\eqref{property:maximal_increment} is an immediate consequence of
the definition~\eqref{eq:T_plus_def} of $T^+$ combined with
\eqref{eq:s_k_range} below.
\begin{lem}
\label{lem:increasingSeq} For any $\varphi\in\R^V$ we have
\begin{gather}
\tau(v) \equiv \tau_1(v,\cdot)\ge \tau_2(v,\cdot)\ge \cdots\ge
\tau_{|V|}(v,\cdot)\ge 0,\quad v\in V,\label{eq:tau_k_non_increasing}\\
s_k\in[0,\tau(P_k)],\quad 1\le k\le |V|,\label{eq:s_k_range}\\
s_1\le s_2\le \cdots\le s_{|V|}.\label{eq:s_k_non_decreasing}
\end{gather}
\end{lem}

\begin{proof}
Observe that, by \eqref{eq:m_alternative_def}, we have
\begin{equation}\label{eq:m_fcn_positivity_prop}
m_{v,h,t}\ge t.
\end{equation}
We shall prove by induction that
\begin{equation}\label{eq:s_j_induction_conclusion}
  \text{$s_k\ge 0$ for all $1\le k\le |V|$}.
\end{equation}
Assume that for some $1\le k\le|V|$ we have
\begin{equation}\label{eq:s_j_induction_assumption}
  \text{$s_j\ge 0$ for all $1\le j<k$}.
\end{equation}
Recall that the function $\tau$ is non-negative. It follows from
\eqref{eq:m_fcn_positivity_prop},
\eqref{eq:s_j_induction_assumption} and the initialization and step
\eqref{alg:update_step} of the addition algorithm that
\begin{equation}\label{eq:partial_tau_k_non_increasing}
  \tau(v) \equiv \tau_1(v,\cdot)\ge \tau_2(v,\cdot)\ge \cdots\ge
\tau_{k}(v,\cdot)\ge 0,\quad v\in V.
\end{equation}
In particular, $s_k = \tau_k(P_k, \varphi_{P_k})\ge 0$. Thus,
\eqref{eq:s_j_induction_assumption} remains true when $k$ is
replaced by $k+1$. We conclude that
\eqref{eq:s_j_induction_conclusion} holds.

It now follows, in the same way that
\eqref{eq:partial_tau_k_non_increasing} was deduced from
\eqref{eq:s_j_induction_assumption}, that
\eqref{eq:tau_k_non_increasing} is valid. Now \eqref{eq:s_k_range}
is verified upon recalling that $s_k = \tau_k(P_k,\varphi_{P_k})$.
It remains to verify \eqref{eq:s_k_non_decreasing}. Let $1\le
k<|V|$. Our choice of the point $P_k$ in step \eqref{alg:P_k_choice}
of the addition algorithm ensures that
\begin{equation}\label{eq:using_P_k_minimality}
  s_k=\tau_k(P_k, \varphi_{P_k})\le \tau_k(P_{k+1},
  \varphi_{P_{k+1}}).
\end{equation}
In addition, it follows from \eqref{eq:m_fcn_positivity_prop} that
\begin{equation*}
  s_k \le m_{P_{k+1}, \varphi_{P_k}, s_k}(h)\quad\text{for all $h\in \R$}.
\end{equation*}
Thus \eqref{eq:algorithm_update_tau} implies that
\begin{equation*}
s_{k+1} = \tau_{k+1}(P_{k+1}, \varphi_{P_{k+1}}) \ge s_k.
\end{equation*}
As $k$ is arbitrary, this establishes \eqref{eq:s_k_non_decreasing}.
\end{proof}

In the next lemma we investigate the gradient of $T^+(\varphi)$,
establishing property~\eqref{property:Lipschitz_preservation} from
Section~\ref{sec:addition_algorithm_prop} for $T^+$.
\begin{lem}
For any $\varphi\in\R^V$ and any edge $(v,w)\in E$,
\begin{align}
    &\text{if }|\varphi_v - \varphi_w|\ge 1\text{ then } T^+(\varphi)_v -
    T^+(\varphi)_w = \varphi_v - \varphi_w,\label{eq:slopes_larger_than_1_unchanged}\\
    &\text{if }|\varphi_v - \varphi_w|< 1\text{ then } |T^+(\varphi)_v -
    T^+(\varphi)_w| < 1.\label{eq:slopes_smaller_than_1_change}
\end{align}
\end{lem}
\begin{proof}
Fix an edge $(v,w)\in E$. Assume without loss of generality that
$v=P_k$ and $w=P_\ell$ for some $1\le k<\ell\le|V|$. Observe that,
by step \eqref{alg:update_step} of the addition algorithm,
\begin{equation}\label{eq:s_ell_m_estimate}
  s_\ell = \tau_\ell(w, \varphi_w) \le m_{w,
  \varphi_v, s_k}(\varphi_{w}).
\end{equation}
Now assume that $|\varphi_v - \varphi_w|\ge 1$. Then, by the
definition \eqref{eq:m_alternative_def} of $m$, we have that
\begin{equation*}
  m_{w, \varphi_v, s_k}(\varphi_w) = s_k.
\end{equation*}
Combining the last two inequalities with
\eqref{eq:s_k_non_decreasing} shows that $s_\ell = s_k$. The
equality \eqref{eq:slopes_larger_than_1_unchanged} now follows from
\eqref{eq:T_plus_def}. Assume now
that $|\varphi_v - \varphi_w|< 1$. On the one hand, by
\eqref{eq:s_k_non_decreasing},
\begin{equation*}
  T^+(\varphi)_v - T^+(\varphi)_w = \varphi_v - \varphi_w +
  s_k - s_\ell \le \varphi_v - \varphi_w< 1.
\end{equation*}
On the other hand, by \eqref{eq:s_ell_m_estimate} and the definition
\eqref{eq:m_alternative_def} of $m$,
\begin{align*}
  T^+(\varphi)_v - T^+(\varphi)_w = \varphi_v - \varphi_w +
  s_k - s_\ell &\ge \varphi_v - \varphi_w +
  s_k - m_{w,\varphi_v, s_k}(\varphi_{w}) =\\
  &= \varphi_v - \varphi_w +
  m_{w,\varphi_v,s_k}(\varphi_v + 1) - m_{w,\varphi_v,
  s_k}(\varphi_{w}).
\end{align*}
Therefore, by \eqref{eq:m_Lipschitz_constant} and our assumption
that $|\varphi_v - \varphi_w|< 1$,
\begin{equation*}
  T^+(\varphi)_v - T^+(\varphi)_w \ge \varphi_v - \varphi_w -
  \frac{1}{2}(\varphi_v + 1 - \varphi_w) = -1 + \frac{1}{2}(\varphi_v + 1 - \varphi_w) > -1.
\end{equation*}
Hence $|T^+(\varphi)_v - T^+(\varphi)_w| < 1$, establishing
\eqref{eq:slopes_smaller_than_1_change}.
\end{proof}

\subsection{Bijectivity}\label{sec:bijectivitiy}
In this section we define an inverse $(T^+)^{-1}$ to the mapping
$T^+$, thereby establishing that $T^+$ is one-to-one and onto as
claimed in property~\eqref{property:one-to-one_onto} from
Section~\ref{sec:addition_algorithm_prop}.

The definition of $(T^+)^{-1}$ uses the same graph $G=(V,E)$,
function $\tau$, constant $\eps$, total order $\preceq$ on $V$ and
family of functions $m_{v,h,t}$ as the definition of $T^+$. It is
based on the following algorithm which takes as input a function
$\tilde{\varphi}\in\R^V$ and outputs four sequences indexed by $1\le
k\le |V|$:
\begin{enumerate}
\item A sequence $(\tilde{P}_k)$ which is a ordering of the vertices $V$,
that is, $\{\tilde{P}_k\} = V$.
\item A sequence $(\tilde{s}_k)\subseteq[0,\infty)$ with $\tilde{s}_k$ representing the amount to subtract from $\tilde{\varphi}$ at vertex
$\tilde{P}_k$.
\item Two auxiliary sequences of functions, $\tilde{\tau}_k:V\times\R\to\R$ and $\tilde{D}_k:V\times\R\to\R$.
\end{enumerate}
The mapping $(T^+)^{-1}:\R^V\to\R^V$ is then defined by
\begin{equation}\label{eq:T_plus_inverse_def}
  (T^+)^{-1}(\tilde{\varphi}) :=
\varphi\;\;\text{ with }\;\;\varphi_{\tilde{P}_k} :=
\tilde{\varphi}_{\tilde{P}_k} - \tilde{s}_k,\quad 1\le k\le |V|.
\end{equation}
\textbf{Inverse addition algorithm:}\\
\noindent \underline{Initialization}. Set $\tilde{\tau}_1(v,h):=\tau(v)$ for all $v\in V$ and $h\in\R$.\\
\noindent \underline{Loop}. For $k$ between $1$ and $|V|$ do:
\begin{enumerate}
\item For each $v\in V$, define $\tilde{D}_k(v,\cdot)$ to be the inverse of the
mapping $h\mapsto h + \tilde{\tau}_k(v,h)$, which exists by
Lemma~\ref{lem:invertibility_in_inverse_alg_def} below.
\item \label{alg:P_k_choice_inv} Set $\tilde{P}_k$ to be the vertex $v$ in $V\setminus \{\tilde{P}_1, \ldots,
\tilde{P}_{k-1}\}$ which minimizes
$\tilde{\tau}_k(v,\tilde{D}_k(v,\tilde{\varphi}_v))$. If there are
multiple vertices achieving the same minimum let $\tilde{P}_k$ be
the smallest one with respect to the total order $\preceq$.
\item \label{alg:s_k_def_inv} Set $\tilde{s}_k:=\tilde{\tau}_k(\tilde{P}_k,\tilde{D}_k(\tilde{P}_k,\tilde{\varphi}_{\tilde{P}_k}))$.
\item \label{alg:update_step_inv} If $k<|V|$ set, for each $v\in V$ and
$h\in\R$,
\begin{equation}\label{eq:algorithm_update_tau_inv}
\tilde{\tau}_{k+1}(v,h):=\begin{cases} \tilde{\tau}_{k}(v,h) &
\text{ if $v\in\{\tilde{P}_1,\ldots, \tilde{P}_k\}$ or $v\not \sim
\tilde{P}_k$}\\
\min(\tilde\tau_{k}(v,h),
m_{v,\tilde{\varphi}_{\tilde{P}_{k}}-\tilde{s}_k,\tilde{s}_k}(h)) &
\text{ if $v\notin\{\tilde{P}_1,\ldots, \tilde{P}_k\}$ and }v\sim
\tilde{P}_{k}
\end{cases}.
\end{equation}
\end{enumerate}

\begin{lem}\label{lem:invertibility_in_inverse_alg_def}
  For any $\tilde{\varphi}\in\R^V$, any $v\in V$ and any $1\le k\le |V|$ the
  function $h \mapsto h + \tilde{\tau}_k(v,h)$ is continuous and strictly
  increasing from $\R$ onto $\R$. Consequently $\tilde{D}_k(v,\cdot)$ is well-defined on $\R$, is also continuous
  and strictly increasing and we have
  \begin{alignat*}{2}
    &\tilde{D}_k(v, h + \tilde{\tau}_k(v,h)) = h,\quad&& h\in\R,\\
    &\tilde{D}_k(v,\tilde{h}) +
    \tilde{\tau}_k(v,\tilde{D}_k(v,\tilde{h})) = \tilde{h},\quad&&
    \tilde{h}\in\R.
  \end{alignat*}
\end{lem}
\begin{proof}
  Fix $\tilde{\varphi}\in\R^V$ and $v\in V$. We prove the lemma by induction. Let $1\le \ell\le |V|$, suppose the algorithm is well-defined and the lemma holds for all $1\le k<\ell$ and let us prove the assertions of the lemma for $k=\ell$.
  Observe that $\tilde{\tau}_\ell(v,\cdot)$ is obtained
  by taking the minimum of $\tau(v)$ and the function
  $m_{v,h,t}(\cdot)$ with various values of $h$ and $t$. Thus, since
  $m_{v,h,t}(\cdot)$ has Lipschitz constant at most $\frac{1}{2}$ by
  \eqref{eq:m_Lipschitz_constant}, it follows that
  $\tilde{\tau}_\ell(v,\cdot)$ has Lipschitz constant at most
  $\frac{1}{2}$. Thus $h \mapsto h + \tilde{\tau}_\ell(v,h)$ is
  continuous and strictly increasing from $\R$ onto $\R$. The remaining assertions of
  the lemma are immediate consequences.
\end{proof}
We claim that $(T^+)^{-1}$ is indeed the inverse of $T^+$, that is,
that
\begin{align}
  &\text{For any $\varphi\in\R^V$,\; $(T^+)^{-1}(T^+(\varphi)) =
  \varphi$\; and}\label{eq:injectivity_of_T+}\\
  &\text{for any $\tilde{\varphi}\in\R^V$,\; $T^+((T^+)^{-1}(\tilde{\varphi})) =
  \tilde{\varphi}$}\label{eq:surjectivity_of_T+}.
\end{align}
These assertions are proved in the next two sections.
\subsubsection{Injectivity}
In this section we prove \eqref{eq:injectivity_of_T+}, showing that
$T^+$ is one-to-one.

Fix $\varphi\in\R^V$. Let $\{P_k\}, \{s_k\}, \{\tau_k\},
\{\tilde{P}_k\}, \{\tilde{s}_k\}, \{\tilde{\tau}_k\},
\{\tilde{D}_k\}$, $1\le k\le |V|$, be the sequences generated when
calculating $T^+(\varphi)$ and when calculating
$(T^+)^{-1}(\tilde{\varphi})$ with $\tilde{\varphi}:=T^+(\varphi)$.
By \eqref{eq:T_plus_def} and \eqref{eq:T_plus_inverse_def} it
suffices to show that
\begin{equation*}
  \tilde{P}_k = P_k,\; \tilde{s}_k = s_k,\; \tilde{\tau}_k =
  \tau_k,\quad 1\le k\le |V|.
\end{equation*}
We prove this claim by induction. We have $\tilde{\tau}_1 = \tau_1$
by the initialization steps of the algorithms. Fix $1\le k\le |V|$
and assume that
\begin{equation}\label{eq:injectivity_induction_hyp}
  \tilde{P}_j = P_j,\, \tilde{s}_j = s_j\;\text{ for $1\le j<k$ and
  }\; \tilde{\tau}_j = \tau_j\;\text{ for $1\le j\le k$}.
\end{equation}
We need to show that
\begin{equation}\label{eq:injectivity_induction_goal}
  \tilde{P}_k = P_k,\, \tilde{s}_k = s_k\;\text{ and, if $k<|V|$,
  }\;\tilde{\tau}_{k+1} = \tau_{k+1}.
\end{equation}
Denote
\begin{equation*}
  \Delta_v := \tau_k(v,\varphi_v)\;\text{ and }\;\tilde{\Delta}_v :=
  \tilde{\tau}_k(v,\tilde{D}_k(v,\tilde{\varphi}_v)).
\end{equation*}
These sequences need not be equal. However, they satisfy certain
relations as the following lemma clarifies.
\begin{lem}\label{lem:Delta_x_and_tilde_rel}
We have $\tilde{\Delta}_{P_k} = \Delta_{P_k}$ and
$\tilde{\Delta}_v\ge \tilde \Delta_{P_k}$ for all $v\in
V\setminus\{\tilde{P}_1,\ldots,\tilde{P}_{k-1}\}$. In addition, for
each $v\in V\setminus\{\tilde{P}_1,\ldots,\tilde{P}_{k-1}\}$ for
which $\tilde{\Delta}_v= \tilde\Delta_{P_k}$ we
  have $\Delta_v = \Delta_{P_k}$.
\end{lem}
Comparing the definitions of $P_k, s_k$ and $\tau_{k+1}$ with those
of $\tilde{P}_k, \tilde{s}_k$ and $\tilde{\tau}_{k+1}$ and using
\eqref{eq:injectivity_induction_hyp} and \eqref{eq:T_plus_def} we
deduce from the lemma that \eqref{eq:injectivity_induction_goal}
holds, completing the inductive proof.
\begin{proof}[Proof of Lemma~\ref{lem:Delta_x_and_tilde_rel}]
  Let us first show that $\tilde{\Delta}_{P_k} = \Delta_{P_k}$. By
  \eqref{eq:T_plus_def} and \eqref{eq:injectivity_induction_hyp},
  \begin{equation*}
    \tilde{\varphi}_{P_k} = \varphi_{P_k} + s_k = \varphi_{P_k} + \tau_k(P_k,\varphi_{P_k}) = \varphi_{P_k} +
    \tilde{\tau}_k(P_k,\varphi_{P_k}).
  \end{equation*}
  Thus $\tilde{D}_k(P_k, \tilde{\varphi}_{P_k}) = \varphi_{P_k}$ and
  hence, using \eqref{eq:injectivity_induction_hyp} again,
  \begin{equation*}
  \tilde{\Delta}_{P_k} = \tilde{\tau}_k(P_k, \tilde{D}_k(P_k, \tilde{\varphi}_{P_k})) = \tau_k(P_k, \varphi_{P_k}) = \Delta_{P_k}.
  \end{equation*}

  Now fix some $v\in V\setminus\{\tilde{P}_1,\ldots,\tilde{P}_{k-1}\}$ and let
us show that $\tilde{\Delta}_v\ge
\tilde{\Delta}_{P_k}=\Delta_{P_k}$. By
\eqref{eq:injectivity_induction_hyp}, $v\in
V\setminus\{P_1,\ldots,P_{k-1}\}$ so that $v = P_m$ for some $m\ge
k$. Hence we may write
\begin{equation}\label{eq:varphi_s_m}
\tilde{\varphi}_v = \varphi_v + s_m.
\end{equation}
By \eqref{eq:tau_k_non_increasing} we have
\begin{equation*}
\Delta_v = \tau_k(v,\varphi_v) \ge \tau_m(v,\varphi_v) = s_m.
\end{equation*}
Thus, using \eqref{eq:injectivity_induction_hyp} again,
\begin{equation*}
  \varphi_v + \tilde{\tau}_k(v, \varphi_v) = \varphi_v + \tau_k(v, \varphi_v) = \varphi_v +
  \Delta_v \ge \varphi_v + s_m = \tilde{\varphi}_v.
\end{equation*}
Hence, since $\tilde{D}_k(v,\cdot)$ is increasing by
Lemma~\ref{lem:invertibility_in_inverse_alg_def}, we conclude that
\begin{equation}\label{eq:rshift_D_varphi}
\tilde{D}_k(v,\tilde{\varphi}_v) \le \varphi_v.
\end{equation}
Consequently, by \eqref{eq:varphi_s_m} and
Lemma~\ref{lem:invertibility_in_inverse_alg_def},
\begin{equation}\label{eq:tilde_delta_x_shiftP_m}
  \varphi_v + s_m = \tilde{\varphi}_v = \tilde{D}_k(v,\tilde{\varphi}_v) +
  \tilde{\tau}_k(v,\tilde{D}_k(v,\tilde{\varphi}_v)) = \tilde{D}_k(v,\tilde{\varphi}_v) + \tilde{\Delta}_v \le \varphi_v + \tilde{\Delta}_v.
\end{equation}
It follows that $\tilde{\Delta}_v\ge s_m$, whence, by
\eqref{eq:s_k_non_decreasing},
\begin{equation}\label{eq:tilde_Delta_x_Delta_P_k}
  \tilde{\Delta}_v \ge s_m \ge s_k = \Delta_{P_k}
\end{equation}
as we wanted to prove.

Lastly, suppose that equality holds in
\eqref{eq:tilde_Delta_x_Delta_P_k}. It follows that equality holds
also in \eqref{eq:tilde_delta_x_shiftP_m} and hence in
\eqref{eq:rshift_D_varphi}. Thus, using
\eqref{eq:injectivity_induction_hyp}, $\Delta_{P_k} =
\tilde{\Delta}_{v} =
\tilde{\tau}_k(v,\tilde{D}_k(v,\tilde{\varphi}_{v}))) =
\tau_k(v,\varphi_{v}) = \Delta_{v}$, as required.
\end{proof}

\subsubsection{Surjectivity}
In this section we prove \eqref{eq:surjectivity_of_T+}, showing that
$T^+$ is onto. The proof is similar to the proof that $T^+$ is
one-to-one as given in the previous section.

The proof requires the following lemma, an analog of
Lemma~\ref{lem:increasingSeq} for $T^+$.
\begin{lem}
\label{lem:increasingSeq_inverse} For any $\tilde{\varphi}\in\R^V$
we have
\begin{gather}
\tau(v) \equiv \tilde{\tau}_1(v,\cdot)\ge \tilde{\tau}_2(v,\cdot)\ge
\cdots\ge
\tilde{\tau}_{|V|}(v,\cdot)\ge 0,\quad v\in V,\label{eq:tau_k_non_increasing_inverse}\\
\tilde{s}_k\in[0,\tau(\tilde{P}_k)],\quad 1\le k\le |V|,\label{eq:s_k_range_inverse}\\
\tilde{s}_1\le \tilde{s}_2\le \cdots\le
\tilde{s}_{|V|}.\label{eq:s_k_non_decreasing_inverse}
\end{gather}
\end{lem}
\begin{proof}
The proof of \eqref{eq:tau_k_non_increasing_inverse} and
\eqref{eq:s_k_range_inverse} follows in exactly the same way as the
proof of Lemma~\ref{lem:increasingSeq} with $(\tilde{P}_k)$,
$(\tilde{s}_k)$ and $(\tilde{\tau}_k)$ replacing $(P_k)$, $(s_k)$
and $(\tau_k)$.

It remains to prove \eqref{eq:s_k_non_decreasing_inverse}. We start
by showing that
\begin{equation}
\tilde{\tau}_k(v,\tilde{D}_k(v,a))\ge b\;\;\Longleftrightarrow\;\;
\tilde{\tau}_k(v, a-b)\ge b,\quad v\in V,\, 1\le k\le V,\,
a,b\in\R.\label{eq:tilde_k_inequality}
\end{equation}
To verify this, observe that by
Lemma~\ref{lem:invertibility_in_inverse_alg_def},
$\tilde{\tau}_k(v,\tilde{D}_k(v,a))\ge b$ is equivalent to
$\tilde{D}_k(v,a)\le a-b$ which, by definition of $\tilde{D}_k$ and
Lemma~\ref{lem:invertibility_in_inverse_alg_def} (the fact that $\tilde{D}_k(v,\cdot)$ is increasing), is equivalent to
$a\le a-b + \tilde{\tau}_k(v,a-b)$, as required.

Now let $1\le k<|V|$. Our choice of the point $\tilde{P}_k$ in step
\eqref{alg:P_k_choice_inv} of the inverse addition algorithm ensures
that
\begin{equation*}
  \tilde{s}_k=\tilde{\tau}_k(\tilde{P}_k, \tilde{D}_k(\tilde{P}_k,\tilde{\varphi}_{\tilde{P}_k}))\le \tilde{\tau}_k(\tilde{P}_{k+1},
  \tilde{D}_k(\tilde{P}_{k+1},\tilde{\varphi}_{\tilde{P}_{k+1}})).
\end{equation*}
We conclude by \eqref{eq:tilde_k_inequality} that
\begin{equation}\label{eq:tilde_s_k_small}
  \tilde{s}_k \le \tilde{\tau}_k(\tilde{P}_{k+1},
  \tilde{\varphi}_{\tilde{P}_{k+1}} - \tilde{s}_k).
\end{equation}
The definition \eqref{eq:m_alternative_def} of $m$
implies that
\begin{equation}\label{eq:tilde_s_k_smaller_than_m}
  \tilde{s}_k \le m_{\tilde{P}_{k+1}, \tilde{\varphi}_{\tilde{P}_{k}} - \tilde{s}_{k}, \tilde{s}_k}(h),\quad\text{for all $h\in\R$}.
\end{equation}
Putting together \eqref{eq:tilde_s_k_small} and
\eqref{eq:tilde_s_k_smaller_than_m} and recalling
\eqref{eq:algorithm_update_tau_inv} yields
\begin{equation*}
  \tilde{s}_k \le \tilde{\tau}_{k+1}(\tilde{P}_{k+1},
  \tilde{\varphi}_{\tilde{P}_{k+1}} - \tilde{s}_k),
\end{equation*}
whence, by \eqref{eq:tilde_k_inequality} again,
\begin{equation*}
  \tilde{s}_k \le \tilde{\tau}_{k+1}(\tilde{P}_{k+1},
  \tilde{D}_{k+1}(\tilde{P}_{k+1},
  \tilde{\varphi}_{\tilde{P}_{k+1}})) = \tilde{s}_{k+1}.
\end{equation*}
As $k$ is arbitrary, this establishes
\eqref{eq:s_k_non_decreasing_inverse}.
\end{proof}

Fix $\tilde{\varphi}\in\R^V$. Let $\{P_k\}, \{s_k\}, \{\tau_k\},
\{\tilde{P}_k\}, \{\tilde{s}_k\}, \{\tilde{\tau}_k\},
\{\tilde{D}_k\}$, $1\le k\le |V|$, be the sequences generated when
calculating $T^+(\varphi)$ with
$\varphi:=(T^+)^{-1}(\tilde{\varphi})$ and when calculating
$(T^+)^{-1}(\tilde{\varphi})$. To show that $T^+$ is onto it
suffices, by \eqref{eq:T_plus_def} and
\eqref{eq:T_plus_inverse_def}, to show that
\begin{equation*}
  P_k = \tilde{P}_k,\; s_k = \tilde{s}_k,\; \tau_k = \tilde{\tau}_k,\quad 1\le k\le |V|.
\end{equation*}
We prove this claim by induction. We have $\tau_1 = \tilde{\tau}_1$
by the initialization steps of the algorithms. Fix $1\le k\le |V|$
and assume that
\begin{equation}\label{eq:surjectivity_induction_hyp}
  P_j = \tilde{P}_j,\, s_j = \tilde{s}_j\;\text{ for $1\le j<k$ and
  }\;  \tau_j = \tilde{\tau}_j\;\text{ for $1\le j\le k$}.
\end{equation}
We need only show that
\begin{equation}\label{eq:surjectivity_induction_goal}
  P_k = \tilde{P}_k,\, s_k = \tilde{s}_k\;\text{ and, if $k<|V|$,
  }\;\tau_{k+1} = \tilde{\tau}_{k+1}.
\end{equation}
Denote
\begin{equation*}
  \Delta_v := \tau_k(v,\varphi_v)\;\text{ and }\;\tilde{\Delta}_v :=
  \tilde{\tau}_k(v,\tilde{D}_k(v,\tilde{\varphi}_v)).
\end{equation*}
As in the previous section, these sequences satisfy certain
relations as the following lemma clarifies.
\begin{lem}\label{lem:Delta_x_and_tilde_rel_onto}
We have $\Delta_{\tilde{P}_k} = \tilde{\Delta}_{\tilde{P}_k}$ and
$\Delta_v\ge \Delta_{\tilde{P}_k}$ for all $v\in
V\setminus\{P_1,\ldots,P_{k-1}\}$. In addition, for each $v\in
V\setminus\{P_1,\ldots,P_{k-1}\}$ for which $\Delta_v =
\Delta_{\tilde{P}_k}$ we have $\tilde{\Delta}_v=
\tilde\Delta_{\tilde{P}_k}$.
\end{lem}
Comparing the definitions of $P_k, s_k$ and $\tau_{k+1}$ with those
of $\tilde{P}_k, \tilde{s}_k$ and $\tilde{\tau}_{k+1}$ and using
\eqref{eq:surjectivity_induction_hyp} and
\eqref{eq:T_plus_inverse_def} we deduce from the lemma that
\eqref{eq:surjectivity_induction_goal} holds, completing the
inductive proof.
\begin{proof}[Proof of Lemma~\ref{lem:Delta_x_and_tilde_rel_onto}]
  Let us first show that $\Delta_{\tilde{P}_k} = \tilde{\Delta}_{\tilde{P}_k}$. By
  \eqref{eq:T_plus_inverse_def} and Lemma~\ref{lem:invertibility_in_inverse_alg_def},
  \begin{equation*}
    \varphi_{\tilde{P}_k} = \tilde{\varphi}_{\tilde{P}_k} - \tilde{s}_k = \tilde{\varphi}_{\tilde{P}_k} - \tilde{\tau}_k(\tilde{P}_k,\tilde{D}_k(\tilde{P}_k,\tilde{\varphi}_{\tilde{P}_k})) =
    \tilde{D}_k(\tilde{P}_k,\tilde{\varphi}_{\tilde{P}_k}).
  \end{equation*}
  Thus, using \eqref{eq:surjectivity_induction_hyp},
  \begin{equation*}
  \Delta_{\tilde{P}_k} = \tau_k(\tilde{P}_k, \varphi_{\tilde{P}_k}) = \tilde{\tau}_k(\tilde{P}_k, \tilde{D}_k(\tilde{P}_k, \tilde{\varphi}_{\tilde{P}_k})) = \tilde{\Delta}_{\tilde{P}_k}.
  \end{equation*}

  Now fix some $v\in V\setminus\{P_1,\ldots,P_{k-1}\}$ and let
us show that $\Delta_v\ge
\Delta_{\tilde{P}_k}=\tilde{\Delta}_{\tilde{P}_k}$. By
\eqref{eq:surjectivity_induction_hyp}, $v\in
V\setminus\{\tilde{P}_1,\ldots,\tilde{P}_{k-1}\}$ so that $v =
\tilde{P}_m$ for some $m\ge k$. Hence we may write, using
Lemma~\ref{lem:invertibility_in_inverse_alg_def},
\begin{equation}\label{eq:varphi_s_m_onto}
\varphi_v = \tilde{\varphi}_v - \tilde{s}_m = \tilde{\varphi}_v -
\tilde{\tau}_m(v, \tilde{D}_m(v,\tilde{\varphi}_v)) =
\tilde{D}_m(v,\tilde{\varphi}_v).
\end{equation}
Consequently, by \eqref{eq:surjectivity_induction_hyp},
\eqref{eq:tau_k_non_increasing_inverse} and
\eqref{eq:s_k_non_decreasing_inverse},
\begin{equation}\label{eq:delta_v_tilde_delta_inequality}
\Delta_v = \tau_k(v,\varphi_v) = \tilde{\tau}_k(v,\varphi_v) \ge
\tilde{\tau}_m(v,\varphi_v) =
\tilde{\tau}_m(v,\tilde{D}_m(v,\tilde{\varphi}_v)) = \tilde{s}_m \ge
\tilde{s}_k = \tilde{\Delta}_{\tilde{P}_k},
\end{equation}
as we wanted to show.

Finally, suppose that equality holds in
\eqref{eq:delta_v_tilde_delta_inequality}. Then, in particular,
$\tilde{s}_m = \tilde \tau_k(v,\varphi_v)$, which, by
\eqref{eq:T_plus_inverse_def}, implies that
\begin{equation*}
  \tilde{\varphi}_v = \varphi_v + \tilde\tau_k(v,\varphi_v).
\end{equation*}
The definition of $\tilde{D}_k$ now yields
\begin{equation*}
  \tilde{D}_k(v,\tilde{\varphi}_v) = \varphi_v,
\end{equation*}
from which we conclude that
\begin{equation*}
  \tilde{\Delta}_{\tilde{P}_k} = \tilde{\tau}_k(v,\varphi_v) =
  \tilde{\tau}_k(v,\tilde{D}_k(v,\tilde{\varphi}_v)) =
  \tilde{\Delta}_v,
\end{equation*}
completing the proof.
\end{proof}

\subsection{The shifts produced by the algorithm}
Our goal in this section is to analyze the shifts produced by the
addition algorithm of Section~\ref{sec:description_of_T^+} and to
give conditions under which $T^+(\varphi)_v - \varphi_v$ is
approximately equal to $\tau(v)$.
Corollary~\ref{cor:addition_lower_bound} verifies property
\eqref{property:addition_lower_bound} from
Section~\ref{sec:addition_algorithm_prop} for $T^+$.

Recall from Section~\ref{sec:addition_algorithm_prop} that
$\EC(\varphi)$ is the subgraph of edges on which $\varphi$ changes
by at least $1-\eps$, that $r(\varphi, v)$ is the radius of the
connected component of $v$ in $\EC(\varphi)$ and $M(\varphi)$ is the
diameter of the largest connected component of $\EC(\varphi)$.
Recall also the definitions of $\tau'(v,k)$ and $L(\tau,\eps)$. Depending on the choice of $\tau$ and $\eps$ the value of
$L(\tau,\eps)$ may be negative, though our theorems will be
meaningful only when this is not the case. The following is the main
proposition of this section.
\begin{prop}\label{prop:bounding_shifts}
  For any $\varphi\in\R^V$ satisfying $M(\varphi)\le
  L(\tau,\eps)$ we have
  \begin{equation*}
    \tau(v) - \tau'(v,r(\varphi, v))\le T^+(\varphi)_v - \varphi_v\le \tau(v)\quad\text{
    for all $v\in V$.}
  \end{equation*}
\end{prop}
The definitions of $M(\varphi)$ and $L(\tau,\eps)$ imply the
following corollary.
\begin{cor}\label{cor:addition_lower_bound}
For any $\varphi\in\R^V$ satisfying $M(\varphi)\le L(\tau,\eps)$ we
have
\begin{equation*}
    \tau(v) - \frac{\eps}{2}\le T^+(\varphi)_v - \varphi_v\le \tau(v)\quad\text{
    for all $v\in V$.}
  \end{equation*}
\end{cor}

\begin{proof}[Proof of Proposition~\ref{prop:bounding_shifts}]
Fix $\varphi\in\R^V$ and let $(P_k)$, $(s_k)$ and $(\tau_k)$ be the
outputs of the addition algorithm of
Section~\ref{sec:description_of_T^+} when running on the input
$\varphi$. For $v\in V$, let $k_v$ stand for that integer for which
$v=P_{k_v}$ and let
\begin{equation*}
\sigma_v:=T^+(\varphi)_v - \varphi_v = s_{k_v}
\end{equation*}
be the amount added to $\varphi_v$ by $T^+$.

The relation $\sigma_v\le \tau(v)$ holds without any assumptions, by
Lemma~\ref{lem:increasingSeq}, proving the upper bound in
Proposition~\ref{prop:bounding_shifts}. Recalling
\eqref{eq:EC_connectivity_def}, define $\EC(\varphi,v):=\left\{ w\in
V:w\xleftrightarrow{\EC(\varphi)}v\right\}$ for $v\in V$. We say
that $v$ is the first vertex visited in $\EC(\varphi,v)$ if $k_v\le
k_u$ for all $u\in \EC(\varphi,v)$. The lower bound in
Proposition~\ref{prop:bounding_shifts} is a consequence of the
following fact: For any $v\in V$,
\begin{equation} \label{eq:toProveXYZ}
    \sigma_v\geq \tau(v) - \tau'(v,r(\varphi,v))\text{ and if }v\text{ is the first vertex visited in }\EC(\varphi,v)\text{ then }  \sigma_v = \tau(v).
\end{equation}
We prove (\ref{eq:toProveXYZ}) by induction on $k_{v}$. Fix $v\in
V$. Suppose first that $v$ is the first vertex visited in
$\EC(\varphi,v)$. If all neighbors $u$ of $v$ have $k_u>k_v$ (in
particular, if $k_v = 1$) then the definition of the addition
algorithm implies that $\sigma_{v}=\tau(v)$ and
(\ref{eq:toProveXYZ}) follows. Otherwise, let $u$ be a neighbor of
$v$ with $k_{u}<k_{v}$ and note that necessarily $u\not\in
\EC(\varphi,v)$ by our assumption on $v$. Now, the induction
hypothesis \eqref{eq:toProveXYZ}, definitions
\eqref{eq:tau_prime_def}, \eqref{eq:L_tau_eps_def} and
\eqref{def:rAndm} and our assumption that $M(\varphi)\le
L(\tau,\eps)$ yield that
\begin{align*}
\tau(v)-\sigma_{u} & \leq\tau(v)-(\tau(u)-\tau'(u,r(\varphi,u))) \leq\tau'(v,r(\varphi,u)+1)\\
 & \leq \tau'(v,M(\varphi)+1) \leq \tau'(v,L(\tau,\epsilon)+1)\leq\frac{\epsilon}{2}.
\end{align*}
This, together with $|\varphi_{v}-\varphi_{u}|<1-\epsilon$ and
\eqref{eq:m_alternative_def}, imply that in step
\ref{alg:update_step} of the addition algorithm, when $k=k_u$, we
have $m_{v,\varphi_u,\sigma_{u}}(\varphi_{v})=\tau(v)$ so that
$\tau_{k_u+1}(v,\varphi_v) = \tau_{k_u}(v,\varphi_v)$. As $u$ is an
arbitrary neighbor of $v$ with $k_u<k_v$ we conclude that
$\sigma_{v}=\tau(v)$ as required in \eqref{eq:toProveXYZ}.

Now suppose that $v$ is not the first vertex visited in
$\EC(\varphi,v)$. Let $u$ be the vertex of $\EC(\varphi,v)$ with
minimal $k_{u}$. Clearly, $k_{u}<k_{v}$ and by the induction
hypothesis \eqref{eq:toProveXYZ}, $\sigma_{u}=\tau(u)$. Thus, Lemma
\ref{lem:increasingSeq} and \eqref{def:rAndm} yield that
$\sigma_{v}\geq\sigma_{u}=\tau(u)\geq\tau(v)-\tau'(v,r(\varphi,v))$,
finishing the proof of \eqref{eq:toProveXYZ}.
\end{proof}

\subsection{Jacobian definition} In this section we find a formula for the
Jacobian of the mapping $T^+$. We start with some smoothness
properties of the functions used in defining $T^+$. We write
$(P_k)$, $(s_k)$ and $(\tau_k)$ for the outputs of the addition
algorithm of Section~\ref{sec:description_of_T^+} when running on
the input $\varphi$.
\begin{lem}\label{lem:right_derivatives_and_Lipschitz_constants}
  For any $\varphi\in\R^V$, $1\le k\le |V|$ and $v\in V$, the
  function $\tau_k(v,\cdot)$ is everywhere differentiable from the right and is Lipschitz continuous with Lipschitz
  constant at most $\frac{1}{2}$.
\end{lem}
\begin{proof}
  The function $\tau_k(v,\cdot)$ is defined by taking a pointwise
  minimum of the constant function $\tau(v)$ and functions of the
  form $m_{w,h,t}(\cdot)$ for various values of the parameters $w,h$
  and $t$. The lemma follows by noting that both $\tau(v)$ and $m_{w,h,t}(\cdot)$ are everywhere differentiable from the
  right and Lipschitz continuous with Lipschitz constant at most
  $\frac{1}{2}$ (see \eqref{eq:m_alternative_def} and \eqref{eq:m_Lipschitz_constant}) and these properties are preserved under taking pointwise
  minimum (it follows, in fact, that $\tau_k(v,\cdot)$ is piecewise linear with all
  slopes of size at most $\frac{1}{2}$).
\end{proof}

Let $J^+:\R^V\to(0,\infty)^V$ be defined by
\begin{equation}\label{eq:J_plus_def}
  J^+(\varphi):=\prod_{k=1}^{|V|} \left(1 + \partial_2\tau_k(P_k,
  \varphi_{P_k})\right)
\end{equation}
where the notation $\partial_2\tau_k(P_k, \varphi_{P_k})$ stands for
the right derivative of $\tau_k$ with respect to its second variable
(which exists by
Lemma~\ref{lem:right_derivatives_and_Lipschitz_constants}),
evaluated at $(P_k, \varphi_{P_k})$.
Lemma~\ref{lem:right_derivatives_and_Lipschitz_constants} ensures also
that the factors in the product are positive.

Recall the definition of the partition $V_0, V_1$ of $V$ and the
measure $d\mu_\theta$ from \eqref{eq:V_0_V_1_partition} and
\eqref{eq:Lebesgue_with_delta_measure_properties_section}.
\begin{lem}\label{lem:changeOfVariables}
For any $\theta:V_0\to\R$ and any function $g:\R^V\to\R$ integrable
with respect to $d\mu_\theta$ the function
$g(T^+(\varphi))J^+(\varphi) $ is integrable with respect to
$d\mu_\theta$ and
\begin{equation}\label{eq:Jacobian_formula}
\int g(T^+(\varphi))J^+(\varphi) d\mu_\theta(\varphi) = \int
g(\varphi) d\mu_\theta(\varphi).
\end{equation}
\end{lem}
We remark that $T^+$ is clearly Borel measurable by its definition in
Section~\ref{sec:description_of_T^+} and hence the integrand on the
left-hand side of \eqref{eq:Jacobian_formula} is measurable. The
rest of the section is devoted to proving this lemma.

We need the following basic facts about Lipschitz continuous maps.
Let $d\ge 1$ be an integer. First, by Rademacher's theorem a
Lipschitz continuous map $T:\R^d\to \R^d$ is almost everywhere
differentiable. Second, the following change of variables formula
holds for any integrable $h: \R^d\to\R$ (see \cite[Section
3.3.3]{Evans:1992fk}),
\begin{equation}
\int h(\varphi)|\!\det(\nabla T(\varphi))|d\varphi=\int
\Bigg[\sum_{\varphi\in
T^{-1}(\psi)}h(\varphi)\Bigg]d\psi,\label{eq:change_of_variables_formula}
\end{equation}
where we have written $d\varphi$ for the Lebesgue measure on $\R^d$.
Here, as remarked in \cite{Evans:1992fk}, $T^{-1}(\psi)$ is at most
countable for almost every $\psi$.

Now, let $\Sigma$ stand for the set of bijections
$\sigma:\{1,\ldots, |V|\}\to V$. For each $\sigma\in\Sigma$ define
the set
\begin{equation}\label{eq:A_sigma_def}
A^{\sigma}:=\{\varphi\in\R^V\colon P_{k}=\sigma(k)\text{ for $1\le
k\le |V|$}\}.
\end{equation}
Referring back to the definition of the addition algorithm in
Section~\ref{sec:description_of_T^+} we see that each $A^\sigma$ is
measurable, possibly empty, and $\R^V = \cup_{\sigma\in\Sigma}
A^\sigma$. For each $\sigma\in\Sigma$ we define a version of the
addition algorithm in which the points are taken in the order
$\sigma$. More precisely, we define an algorithm taking as input a
function $\varphi\in\R^V$ and outputting two sequences indexed by
$1\le k\le |V|$:
\begin{enumerate}
\item A sequence $(s_k^\sigma)\subseteq[0,\infty)$.
\item A sequence $(\tau_k^\sigma)$ of functions, $\tau_k^\sigma:V\times\R\to\R$.
\end{enumerate}
\textbf{Addition algorithm with order $\sigma$:}\\
\noindent \underline{Initialization}. Set $\tau_1^\sigma(v,h):=\tau(v)$ for all $v\in V$ and $h\in\R$.\\
\noindent \underline{Loop}. For $k$ between $1$ and $|V|$ do:
\begin{enumerate}
\item \label{alg:s_k__sigma_def} Set $s_k^\sigma:=\tau_k^\sigma(\sigma(k),\varphi_{\sigma(k)})$.
\item \label{alg:update_step_sigma} If $k<|V|$ set, for each $v\in V$ and
$h\in\R$,
\begin{equation}\label{eq:algorithm_update_tau_sigma}
\tau_{k+1}^\sigma(v,h):=\begin{cases} \tau_{k}^\sigma(v,h) & \text{
if $v\in\{\sigma(1),\ldots, \sigma(k)\}$ or $v\not \sim
\sigma(k)$}\\
\min(\tau_{k}^\sigma(v,h), m_{v,\varphi_{\sigma(k)},s_k^\sigma}(h))
& \text{ if $v\notin\{\sigma(1),\ldots, \sigma(k)\}$ and }v\sim
\sigma(k)
\end{cases}.
\end{equation}
\end{enumerate}
We then define a mapping $T^\sigma:\R^V\to\R^V$ by
\begin{equation}\label{eq:T_sigma_def}
  T^\sigma(\varphi) :=
\tilde{\varphi}^\sigma\;\;\text{ with
}\;\;\tilde{\varphi}^\sigma_{\sigma(k)} := \varphi_{\sigma(k)} +
s_k^\sigma,\quad 1\le k\le |V|.
\end{equation}
Comparing the definitions of $T^+$ and $T^\sigma$ we conclude that
\begin{equation}\label{eq:T^+_on_A^sigma}
  T^+(\varphi) =
T^\sigma(\varphi),\, s_k = s_k^\sigma\text{ and }\tau_k =
\tau_k^\sigma\;\text{ for $\varphi\in A^\sigma$}.
\end{equation}

Fix a $\theta:V_0\to\R$ and let
\begin{equation*}
  X:=\{\varphi\in\R^V\colon \forall v\in V_0,\,\varphi_v =
  \theta_v\}.
\end{equation*}
Observe that $T^+$ maps $X$ bijectively onto $X$ by properties
\eqref{property:one-to-one_onto} and
\eqref{property:maximal_increment} (see Section
\ref{sec:addition_algorithm_prop}) and the definition of $V_0$. The
measure $d\mu_\theta$ is supported on $X$; identifying $X$ with
$\R^{V_1}$ in the natural way it coincides with the Lebesgue measure
on $X$.


By \eqref{eq:m_alternative_def}, the function  $m_{v,h,t}(h')$ is
Lipschitz continuous as a function of $h,t$ and $h'$, for every
fixed $v$. In addition, the composition and pointwise minimum of
Lipschitz continuous functions is also Lipschitz continuous. It
follows that for every $v$ and $k$, the function
$\tau_k^\sigma(v,h)$ is Lipschitz continuous as a function of $h$
and $\varphi$ (i.e., as an implicit function of $\varphi_w$ for
every $w\in V$). We thus deduce from the definition of $s_k^\sigma$
and \eqref{eq:T_sigma_def} that $T^\sigma$ is a Lipschitz continuous
map. We also note that $T^\sigma$ maps $X$ into $X$ since
\begin{equation}\label{eq:tau_k_sigma_on_V_0}
\tau_k^\sigma(v,\cdot)\equiv0\quad\text{for all $v\in V_0$},
\end{equation}
as follows by induction on $k$ using the fact that $m_{v,h,t}\ge t$
by \eqref{eq:m_alternative_def}. Thus we may apply the formula
\eqref{eq:change_of_variables_formula} (by identifying $X$ with
$\R^{V_1}$ and $d\mu_\theta$ with the Lebesgue measure on
$\R^{V_1}$) to obtain that
\begin{equation}\label{eq:T_sigma_change_of_variables}
  \int_X h(\varphi)|\!\det(\nabla_{V_1} T^\sigma(\varphi))|d\mu_\theta(\varphi)=\int_X \Bigg[\sum_{\varphi\in
  (T^\sigma)^{-1}(\psi)}h(\varphi)\Bigg]d\mu_\theta(\psi)
\end{equation}
for every $\sigma\in \Sigma$ and $h:X\to\R$ integrable with respect
to $d\mu_\theta$. Here and below, we denote by $\nabla_W T^\sigma$,
$W\subseteq V$, the matrix-valued function
\begin{equation*}
\nabla_W T^\sigma(\varphi):=\Big(\frac{\partial
T^\sigma(\varphi)_v}{\partial \varphi_w}\Big)_{v,w\in W}.
\end{equation*}
We continue to find a formula for $|\!\det(\nabla_{V_1}
T^\sigma(\varphi))|$. We note first that $\nabla_{V_1}
T^\sigma(\varphi)$ exists for $d\mu_\theta$-almost every $\varphi\in
X$ as, by the above discussion, $T^\sigma$ is Lipschitz continuous
from $X$ to $X$. By construction of $T^\sigma$, $\nabla_V T^\sigma$
has a triangular form when its rows and columns are sorted in the
order of $\sigma$. Hence the definition of $s_k^\sigma$,
\eqref{eq:T_sigma_def} and \eqref{eq:tau_k_sigma_on_V_0} yield that
for $d\mu_\theta$-almost every $\varphi\in X$ we have
\begin{equation}\label{eq:T_sigma_Jacobian_on_V_1_and_V}
|\!\det(\nabla_{V_1} T^{\sigma}(\varphi))| = \prod_{\substack{1\le
k\le |V|\\
\sigma(k)\in
V_1}}\left|1+\partial_{2}\tau_k^\sigma(\sigma(k),\varphi_{\sigma(k)})\right|
=
\prod_{k=1}^{|V|}\left|1+\partial_{2}\tau_k^\sigma(\sigma(k),\varphi_{\sigma(k)})\right|.
\end{equation}
Now let $h:\R^V\to\R$ be a function integrable with respect to
$d\mu_\theta$ and define
\begin{equation*}
  h^\sigma(\varphi):=h(\varphi) 1_{(\varphi\in
A^{\sigma})},\quad\sigma\in\Sigma.
\end{equation*}
Putting together \eqref{eq:J_plus_def}, the fact that $J^+\ge 0$,
\eqref{eq:A_sigma_def}, \eqref{eq:T^+_on_A^sigma},
\eqref{eq:T_sigma_Jacobian_on_V_1_and_V}
and \eqref{eq:T_sigma_change_of_variables} we have
\begin{equation*}
\begin{split}
  \int_X h(\varphi)J^+(\varphi) d\mu_\theta(\varphi) &= \sum_{\sigma\in
  \Sigma} \int_X h^\sigma(\varphi)\prod_{k=1}^{|V|} \left|1 + \partial_2\tau_k(P_k,
  \varphi_{P_k})\right| d\mu_\theta(\varphi) =\\
  &= \sum_{\sigma\in\Sigma} \int_X
h^\sigma(\varphi)\prod_{k=1}^{|V|} \left|1 +
\partial_2\tau_k^\sigma(\sigma(k),
  \varphi_{\sigma(k)})\right| d\mu_\theta(\varphi) =\\
  &= \sum_{\sigma\in\Sigma} \int_X
h^\sigma(\varphi)|\!\det(\nabla_{V_1} T^{\sigma}(\varphi))|
d\mu_\theta(\varphi) = \sum_{\sigma\in\Sigma}\int_X
\Bigg[\sum_{\varphi\in
(T^\sigma)^{-1}(\psi)}h^\sigma(\varphi)\Bigg]d\mu_\theta(\psi).
  \end{split}
\end{equation*}
Finally, $T^+$ is invertible by Section~\ref{sec:bijectivitiy} and
$T^+ = T^\sigma$ on $A^\sigma$ by \eqref{eq:T^+_on_A^sigma}. Hence
$T^\sigma$ restricted to $A^\sigma$ is one-to-one. Thus, since
$h^\sigma(\varphi)=0$ when $\varphi\notin A^\sigma$, we may continue
the last equality to obtain
\begin{equation*}
  \int_X h(\varphi)J^+(\varphi) d\mu_\theta(\varphi) = \sum_{\sigma\in\Sigma}\int_X h^\sigma((T^+)^{-1}(\psi)) d\mu_\theta(\psi) = \int_X h((T^+)^{-1}(\psi))
  d\mu_\theta(\psi).
\end{equation*}
This equality is obtained for any $h:\R^V\to\R$ integrable with
respect to $d\mu_\theta$. Letting $g:\R^V\to\R$ be integrable with
respect to $d\mu_\theta$, Lemma~\ref{lem:changeOfVariables} now
follows by substituting $h$ with $g(T^+(\varphi))$. Formally, this
is done by using the above equality to approximate $g(T^+(\varphi))$
with $h$ which are integrable with respect to $d\mu_\theta$.

\subsection{Properties of $T^-$} The relation
\eqref{eq:T_+_T_-_relation} defines a mapping $T^-:\R^V\to\R^V$ by
\begin{equation}\label{eq:T_minus_def}
  T^-(\varphi):=2\varphi - T^+(\varphi).
\end{equation}
In this section we establish that $T^-$ satisfies similar properties
to those proved for $T^+$, as claimed in
Section~\ref{sec:addition_algorithm_prop}.

In this section, to emphasize the dependence on $\varphi$, we write
$(P_k^\varphi)$, $(s_k^\varphi)$ and $(\tau_k^\varphi)$ for the
outputs of the addition algorithm of
Section~\ref{sec:description_of_T^+} when running on the input
$\varphi$. Putting together \eqref{eq:T_plus_def} and
\eqref{eq:T_minus_def} we see that
\begin{equation}\label{eq:T_minus_second_def}
  T^-(\varphi) =
\tilde{\varphi}\;\;\text{ with }\;\;\tilde{\varphi}_{P_k^\varphi} :=
\varphi_{P_k^\varphi} - s_k^\varphi,\quad 1\le k\le |V|.
\end{equation}
We claim that, due to the symmetry of the function $f$ of
\eqref{eq:bump_function_def},
\begin{equation}\label{eq:T_minus_third_def}
  T^-(\varphi) = -T^+(-\varphi)\quad\text{ for all $\varphi\in\R^V$}.
\end{equation}
To see this observe first that the symmetry of $f$ and
\eqref{eq:m_alternative_def} imply
\begin{equation*}
  m_{v,-h,t}(-h') = m_{v,h,t}(h')\quad\text{ for all $v\in V$ and
  $h,h',t\in\R$}.
\end{equation*}
Thus, examining the addition algorithm of
Section~\ref{sec:description_of_T^+} we conclude that
\begin{equation}\label{eq:addition_algorithm_outputs_for_minus_phi}
  P_k^{-\varphi} = P_k^{\varphi}, s_k^{-\varphi} =
  s_k^{\varphi}\text{ and }\tau_k^{-\varphi}(v,-h) =
  \tau_k^{\varphi}(v,h)\quad\text{ for all $1\le k\le |V|$, $v\in V$, $h\in\R$ and
  $\varphi\in\R^V$}.
\end{equation}
Together with \eqref{eq:T_plus_def}, this equality implies
\eqref{eq:T_minus_third_def}.

Now, the fact that $T^-$ satisfies properties
\eqref{property:one-to-one_onto},
\eqref{property:maximal_increment},
\eqref{property:Lipschitz_preservation} and
\eqref{property:addition_lower_bound} in
Section~\ref{sec:addition_algorithm_prop} follows immediately from
\eqref{eq:T_minus_def}, \eqref{eq:T_minus_third_def} and the fact
that $T^+$ satisfies these properties. We now show that $T^-$ also
satisfies \eqref{eq:Jacobian_formula_properties_section}. Define
$J^-:\R^V\to(0,\infty)^V$ by
\begin{equation}\label{eq:J_minus_def}
  J^-(\varphi):=\prod_{k=1}^{|V|} \left(1 - \partial_2\tau_k^\varphi(P_k^\varphi,
  \varphi_{P_k^\varphi})\right),
\end{equation}
analogously to \eqref{eq:J_plus_def}. Observe that $J^-(\varphi) =
J^+(-\varphi)$ by \eqref{eq:T_minus_def} and
\eqref{eq:addition_algorithm_outputs_for_minus_phi}. Recall the
definition of the measure $d\mu_\theta$ from
\eqref{eq:Lebesgue_with_delta_measure_properties_section}. Using
\eqref{eq:T_minus_third_def} and the equality
\eqref{eq:Jacobian_formula_properties_section} for $T^+$ we have for
every $\theta:V_0\to\R$ and every $g:R^V\to[0,\infty)$, integrable
with respect to $d\mu_\theta$,
\begin{equation*}
\begin{split}
  \int &g(T^-(\varphi))J^-(\varphi) d\mu_\theta(\varphi) = \int
  g(-T^+(-\varphi))J^+(-\varphi) d\mu_\theta(\varphi) = \\
  &= \int g(-T^+(\varphi))J^+(\varphi)
  d\mu_{-\theta}(\varphi) =\int g(-\varphi) d\mu_{-\theta}(\varphi) = \int g(\varphi) d\mu_\theta(\varphi).
\end{split}
\end{equation*}

We remark that the symmetry of the function $f$ of
\eqref{eq:bump_function_def}, while essential for establishing
\eqref{eq:T_minus_third_def}, is not necessary for establishing the
properties of $T^-$ described in
Section~\ref{sec:description_of_T^+}. These properties may also be
obtained without using \eqref{eq:T_minus_third_def} by repeating the
proofs used for $T^+$.

\subsection{The geometric average of the Jacobians}
In this section we provide an estimate for the geometric average of
the Jacobians $J^+$ and $J^-$ in terms of the connectivity
properties of the subgraph $\EC(\varphi)$ and the Lipschitz
properties of the function $\tau$. This estimate establishes
property~\eqref{property:Jacobian_estimate} from
Section~\ref{sec:addition_algorithm_prop}.
\begin{lem}
\label{lem:lowerDenistyBound}For any $\varphi\in\R^V$ satisfying
$M(\varphi)\le L(\tau,\eps)$ we have
\begin{equation*}
\sqrt{J^+(\varphi)J^-(\varphi)} \ge
\exp\left(-\frac{1}{\eps^2}\sum_{v\in V} \tau'\left(v,1+\max_{w\sim
v}r(\varphi,w)\right)^2\right).
\end{equation*}
\end{lem}
\begin{proof}
Fix $\varphi\in\R^V$ satisfying $M(\varphi)\le L(\tau,\eps)$. Write
$(P_k)$, $(s_k)$ and $(\tau_k)$ for the outputs of the addition
algorithm of Section~\ref{sec:description_of_T^+} when running on
the input $\varphi$. Denote $\sigma_v := T^+(\varphi)_v - \varphi_v$
for $v\in V$. By \eqref{eq:J_plus_def} and \eqref{eq:J_minus_def} we
get
\begin{equation}\label{eq:hola}
\log\left(\sqrt{J^+(\varphi)J^-(\varphi)}\right)\geq\frac{1}{2}\sum_{k=1}^{|V|}\log\left(1
- \left(\partial_2\tau_k(P_k,
  \varphi_{P_k})\right)^2\right)\ge -\sum_{k=1}^{|V|} \left(\partial_2\tau_k(P_k,
  \varphi_{P_k})\right)^2,
\end{equation}
where we have used that
$|\partial_2\tau_k(P_k,\varphi_{P_k})|\leq1/2$ for all $k$ according
to Lemma~\ref{lem:right_derivatives_and_Lipschitz_constants}.
Examination of the addition algorithm of
Section~\ref{sec:description_of_T^+} reveals that $\tau_k(v,h)$ is
the minimum of $\tau(v)$ and $m_{v, \varphi_w, \sigma_w}(h)$ where
$w$ ranges over a (possibly empty) subset of the neighbors of $v$.
Observing that the Lipschitz constant of $m_{v,h,t}$ is at most
$\max\rbr{\frac{1}{\eps}(\tau(v)-t),0}$ by \eqref{eq:m_alternative_def}, we see
that
\begin{equation}\label{eq:tau_k_derivative_estimate}
  |\partial_2\tau_k(v, h)|\le \max\left(\frac{1}{\eps}\left(\tau(v) -
  \min_{w\sim v} \sigma_w\right), 0\right).
\end{equation}
Now, using our assumption that $M(\varphi)\le L(\tau,\eps)$,
Proposition~\ref{prop:bounding_shifts} yields that
\begin{equation}\label{eq:tau_k_derivative_estimate2}
  \tau(v) - \min_{w\sim v} \sigma_w \le \tau(v) - \min_{w\sim v}
  (\tau(w) - \tau'(w,r(\varphi, w))) \le
  \tau'\left(v, 1+\max_{w\sim v} r(\varphi, w)\right).
\end{equation}
Plugging \eqref{eq:tau_k_derivative_estimate2} into
\eqref{eq:tau_k_derivative_estimate} shows that
\begin{equation*}
  |\partial_2\tau_k(v, h)|\le \frac{1}{\eps}\tau'\left(v, 1+\max_{w\sim v}
  r(\varphi, w)\right).
\end{equation*}
The lemma follows by substituting this estimate in \eqref{eq:hola}.
\end{proof}

\section{Reflection positivity for random surfaces}\label{sec:reflection_positivity}
Recall the random surface measure $\mu_{\T_n^2, \zero, U}$, defined
in \eqref{eq:mu_T_n_2_U_measure_def}, corresponding to a potential
$U$. In this section we estimate the probability that the random
surface has many edges with large slopes.

We start by explaining why the measure $\mu_{\T_n^2, \zero, U}$ is
well-defined under our assumptions.
\begin{lem}\label{lem:measure_well_defined}
The measure $\mu_{\T_n^2, \zero, U}$ is well-defined for any
potential $U$ satisfying condition~\eqref{eq:U_integral_cond}. In
addition, there exists a constant $c(U)>0$ for which
\begin{equation}\label{eq:partition_function_lower_bound}
  Z_{\T_n^2,\zero, U}\ge c(U)^{|V(\T_n^2)|}.
\end{equation}
\end{lem}
\begin{proof}
  Let $U$ be a potential satisfying condition~\eqref{eq:U_integral_cond}. In order that $\mu_{\T_n^2, \zero, U}$ be well-defined it suffices that
  \begin{equation}\label{eq:Z_def}
    Z_{\T_n^2,\zero, U} = \int \exp\Bigg(-\sum_{(v,w)\in E(\T_n^2)}
  U(\varphi_v - \varphi_w)\Bigg) \delta_0(d\varphi_{\zero})\prod_{v\in V(\T_n^2)\setminus\{\zero\}}
  d\varphi_v
  \end{equation}
  satisfies $0<Z_{\T_n^2,\zero, U}<\infty$.

  We first show that $Z_{\T_n^2,\zero, U}<\infty$. Let $S$ be a spanning tree of
$\T_n^2$, regarded here as a subset of edges. Then
\[
Z_{\T_n^2,\zero, U}\leq C_1(U)^{|E(\T_n^2)\setminus S|}\int
\exp\Bigg(-\sum_{(v,w)\in S}
  U(\psi_v - \psi_w)\Bigg) \delta_0(d\psi_{\zero})\prod_{v\in V(\T_n^2)\setminus\{\zero\}}
  d\psi_v
\]
where $C_1(U)\eqdef \sup_x\, \exp(-U(x))<\infty$ by
\eqref{eq:U_integral_cond}. By integrating the vertices in
$V(\T_n^2)\setminus\{\zero\}$ leaf by leaf according to the spanning
tree $S$ the integral above equals $\left(\int
\exp(-U(x))dx\right)^{|S|}$, which is finite by
\eqref{eq:U_integral_cond}.

  We now prove \eqref{eq:partition_function_lower_bound}, implying in particular that $Z_{\T_n^2,\zero, U}>0$.
  Condition~\eqref{eq:U_integral_cond} implies the existence of some
  $\alpha<\infty$ for which the set $A:=\{x\colon U(x)\le\alpha\}$
  has positive measure. The Lebesgue density theorem now yields the
  existence of a point $a\in A$ and an $\eps>0$ such that
  \begin{equation*}
    |[a-2\eps, a+2\eps] \cap A| \ge 0.9\cdot4\eps,
  \end{equation*}
  where we write $|B|$ for the Lebesgue measure of a set $B\subseteq\R$. This
  implies that
  \begin{equation}\label{eq:difference_in_A_1}
    \inf_{x,y,z,w\in[-\eps,\eps]}|\{u\in [a-\eps,a+\eps]\colon u-x,u-y,u-z,u-w\in
    A\}|\ge 0.4\eps
  \end{equation}
  and, using that $U(x)=U(-x)$, the analogous statement
  \begin{equation}\label{eq:difference_in_A_2}
    \inf_{x,y,z,w\in[a-\eps,a+\eps]}|\{u\in [-\eps,\eps]\colon u-x,u-y,u-z,u-w\in
    A\}|\ge 0.4\eps.
  \end{equation}
  Denote
  by $(V_{\text{even}}, V_{\text{odd}})$ a bipartition of the
  vertices of the bipartite graph $\T_n^2$, with $\zero\in V_{\text{even}}$, and define the following
  set of configurations,
  \begin{equation*}
    \Omega:=\{\varphi:V(\T_n^2)\to\R\colon
    \varphi(V_{\text{even}})\subseteq[-\eps,\eps],
    \varphi(V_{\text{odd}})\subseteq[a-\eps,a+\eps]\}.
  \end{equation*}
  We conclude from the definition of $A$, \eqref{eq:difference_in_A_1} and \eqref{eq:difference_in_A_2} that the integral in \eqref{eq:Z_def}, restricted to
  the set $\Omega$, is at least
  $\left(0.4\eps\exp(-\alpha)\right)^{|V(\T_n^2)\setminus\{\zero\}|}>0$.
  This can be seen by again fixing a spanning tree of
$\T_n^2$ and integrating the vertices in
$V(\T_n^2)\setminus\{\zero\}$ leaf by leaf according to it.

  As a side note we remark that the fact that $\T_n^2$ is bipartite was essential
  for showing that $Z_{\T_n^2,\zero, U}>0$. If $\T_n^2$ is replaced by a triangle
  graph on $3$ vertices then the analogous quantity
  to $Z_{\T_n^2,\zero, U}$ is zero when, say, $\{x\colon
  U(x)<\infty\} = [-3,-2]\cup[2,3]$. However, the above argument can
  be easily modified to work for all graphs if $\{x\colon
  U(x)<\infty\}$ contains an interval around $0$.
\end{proof}

For $0<L<\infty$ and $0<\delta<1$ we say a potential $U$ \emph{has
$(\delta, L)$-controlled gradients on $\T_n^2$} if the following
holds:
\begin{enumerate}
  \item There exists some $K>L$ such that $U(x)<\infty$ for $|x|<K$.
  \item If $\varphi$ is randomly sampled from the measure $\mu_{\T_n^2,\zero,
  U}$ and if we define the random subgraph $\EC(\varphi, L)$ of $\T_n^2$ by
  \begin{equation}\label{eq:E_varphi_L_random_graph}
    \EC(\varphi, L):=\{(v,w)\in E(\T_n^2)\colon |\varphi_v-\varphi_w|\ge
    L\}
  \end{equation}
  then
  \begin{equation*}
    \P(e_1, \ldots, e_k\in \EC(\varphi,L))\le \delta^k\quad\text{for all
    $k\ge 1$ and distinct $e_1,\ldots, e_k\in E(\T_n^2)$}.
  \end{equation*}
\end{enumerate}
\begin{thm}\label{thm:controlledGradients}
Suppose a measurable $U:\R\to(-\infty,\infty]$ satisfies
$U(x)=U(-x)$, condition~\eqref{eq:U_integral_cond} and the
condition:
\begin{equation}\label{eq:interval_potential}
\begin{aligned}
    &\text{Either $U(x)<\infty$ for all $x$ or there exists some $0<K<\infty$ such}\\
    &\text{that $U(x)<\infty$ when $|x|<K$ and
    $U(x)=\infty$ when $|x|>K$}.
\end{aligned}
\end{equation}
Then for any $0<\delta<1$ there exists an $0<L<\infty$ such that for
all $n\ge 1$, $U$ has $(\delta,L)$-controlled gradients on $\T_n^2$.
\end{thm}
This theorem is proved in the following sections, making use of
reflection positivity and the chessboard estimate.

\subsection{Reflection positivity} We start by reviewing the basic definitions
pertaining to our use of reflection positivity and the chessboard estimate. Our treatment is based on \cite[Section 5]{Biskup:2009fk}.

Let $n\ge 1$. For $-n+1\le j\le n$ the vertical plane of reflection
$P_j^{\text{ver}}$ (passing through vertices) is the set of vertices
\begin{equation*}
  P_j^{\text{ver}}:=\{(j,k)\in V(\T_n^2)\colon -n+1\le k\le n\}.
\end{equation*}
The plane $P_j^{\text{ver}}$ divides $\T_n^2$ into two overlapping
parts, $P_j^{\text{ver},+}$ and $P_j^{\text{ver},-}$, according to
\begin{align*}
&P_j^{\text{ver}, +}:=\{(j+m, k)\in V(\T_n^2)\colon 0\le m\le n, -n+1\le k\le n\},\\
&P_j^{\text{ver},-}:=\{(j-m, k)\in V(\T_n^2)\colon 0\le m\le n,
-n+1\le k\le n\},
\end{align*}
where here and below, arithmetic operations on vertices of $T_n^2$
are performed modulo $2n$ (in the set $\{-n+1, -n+2, \ldots, n-1,n\}$). The parts $P_j^{\text{ver},+}$ and
$P_j^{\text{ver},-}$ overlap in
\begin{equation*}
  \bar{P}_j^{\text{ver}}:=P_j^{\text{ver}}\cup\{(j+n,k)\in V(\T_n^2)\colon -n+1\le k\le
  n\}.
\end{equation*}
The reflection $\theta_{P_j^{\text{ver}}}$ is the mapping
$\theta_{P_j^{\text{ver}}}:V(\T_n^2)\to V(\T_n^2)$ defined by
\begin{equation*}
  \theta_{P_j^{\text{ver}}}(\ell,k) = (2j-\ell, k),
\end{equation*}
which exchanges $P_j^{\text{ver},+}$ and $P_j^{\text{ver},-}$. We
also define horizontal planes of reflection $P_j^{\text{hor}}$ and
their associated $P_j^{\text{hor},+}$, $P_j^{\text{hor},-}$,
$\bar{P}_j^{\text{hor}}$ and $\theta_{P_j^{\text{hor}}}$ in the same
manner by switching the role of the two coordinates of vertices in
$\T_n^2$. We write simply $P, P^+, P^-, \bar{P}$ and $\theta_P$ when
the plane of reflection $P$ is one of the planes $P_j^\text{ver}$ or
$P_j^\text{hor}$ which is left unspecified.

Denote by $\F$ the set of all measurable functions
$f:\R^{V(\T_n^2)}\to\R$ satisfying
\begin{equation}\label{eq:translation_invariant_f_def}
f(\varphi) = f(\varphi + c)\quad\text{for all
$\varphi\in\R^{V(\T_n^2)}$ and $c\in\R$}.
\end{equation}
Equivalently, $\F$ is the set of all measurable functions depending
only on the gradient of $\varphi$. For a plane of reflection $P$ we
write $\F_P^+$ (respectively $\F_P^-$) for the set of $f\in\F$ for
which $f(\varphi)$ depends only on $\varphi_v$, $v\in P^+$
(respectively $v\in P^-$). We extend the definition of $\theta_P$ to
act on $\R^{V(\T_n^2)}$ and $\F$ by
\begin{equation*}
  (\theta_P \varphi)_v:=\varphi_{\theta_P(v)}\quad\text{and}\quad(\theta_P f)(\varphi):=f(\theta_P \varphi).
\end{equation*}
When $\varphi$ is randomly sampled from a probability measure on
$\R^{V(\T_n^2)}$ we will regard a function $f\in\F$ as a random
variable (taking the value $f(\varphi)$) and write $\E f$ for its
expectation.
\begin{defn} Let $\varphi$ be randomly sampled from a probability measure $\P$ on $\R^{V(\T_n^2)}$. We say that $\mathbb{P}$ is reflection positive with respect to $\F$ if for any plane
of reflection $P$ and any two bounded $f, g\in \F_P^+$,
\begin{equation}
\mathbb{E}(f\,\theta_P g)=\mathbb{E}(g\,\theta_P
f)\label{eq:reflectionIsCool}
\end{equation}
and
\begin{equation}
\mathbb{E}(f\,\theta_{P}f)\geq0.\label{eq:reflectionIsPositive}
\end{equation}
\end{defn}
We call a function $f\in \F$ a block function at $(j,k)\in
V(\T_n^2)$ if
\begin{equation}\label{eq:block_function_def}
  f(\varphi) =
  f_0(\varphi_{(j,k)},\varphi_{(j+1,k)},\varphi_{(j,k+1)},\varphi_{(j+1,k+1)})
\end{equation}
for some $f_0:\R^4\to\R$. For $t=(t_1,t_2)\in V(\T_n^2)$ we define a
reflection operator $\vartheta_t$ acting on block functions as
follows. If $f$ is a block function at $(j,k)\in V(\T_n^2)$ then
$\vartheta_t f$ is the function obtained from $f$ by performing the
reflections which map the block at $(j,k)$ to the block at $(j+t_1,
k+t_2)$. Explicitly, if $f$ is defined by
\eqref{eq:block_function_def} then $\vartheta_t f$ is the block
function at $(j+t_1, k+t_2)$ defined by
\begin{equation}\label{eq:reflection_operator_def}
(\vartheta_t f)(\varphi):=\begin{cases}
    f_0(\varphi_{(j+t_1,k+t_2)},\varphi_{(j+t_1+1,k+t_2)},\varphi_{(j+t_1,k+t_2+1)},\varphi_{(j+t_1+1,k+t_2+1)})&\text{$t_1, t_2$ even}\\
    f_0(\varphi_{(j+t_1+1,k+t_2)},\varphi_{(j+t_1,k+t_2)},\varphi_{(j+t_1+1,k+t_2+1)},\varphi_{(j+t_1,k+t_2+1)})&\text{$t_1$ odd, $t_2$ even}\\
    f_0(\varphi_{(j+t_1,k+t_2+1)},\varphi_{(j+t_1+1,k+t_2+1)},\varphi_{(j+t_1,k+t_2)},\varphi_{(j+t_1+1,k+t_2)})&\text{$t_1$ even, $t_2$ odd}\\
    f_0(\varphi_{(j+t_1+1,k+t_2+1)},\varphi_{(j+t_1,k+t_2+1)},\varphi_{(j+t_1+1,k+t_2)},\varphi_{(j+t_1,k+t_2)})&\text{$t_1, t_2$ odd}\\
  \end{cases}.
\end{equation}

\begin{thm} (Chessboard estimate)
\label{thm:checkboard} Let $\varphi$ be randomly sampled from a
probability measure $\P$ on $\R^{V(\T_n^2)}$. Suppose that $\P$ is
reflection positive with respect to $\F$. Then for any $1\le m\le
|V(\T_n^2)|$, any $f_1,\ldots, f_m$, bounded block functions at
$(0,0)$, and any distinct $t_1,\ldots, t_m\in V(\T_n^2)$ we have
\begin{equation}\label{eq:chessboard_estimate}
  \left|\E\left(\prod_{i=1}^m \vartheta_{t_i}
  f_i\right)\right|^{|V(\T_n^2)|}\le \prod_{i=1}^m
  \E\left(\prod_{t\in \T_n^2} \vartheta_t f_i\right).
\end{equation}
In particular, the right-hand side is non-negative.
\end{thm}
For completeness, we provide a short proof of the chessboard
estimate in Section~\ref{sec:proof_of_chessboard_estimate} below. We remark that the same proof shows that if $\P$ is reflection
positive with respect to all measurable functions on
$\R^{V(\T_n^{2})}$ then it also satisfies the chessboard estimate
with respect to this class. We restrict here to the class $\F$ in
view of our application to random surface measures, see
Proposition~\ref{prop:random_surface_RP} below.

\subsection{Controlled gradients property}
In this section we prove Theorem~\ref{thm:controlledGradients}. We
start by proving that our random surface measures are reflection
positive.
\begin{prop}\label{prop:random_surface_RP}
Suppose a measurable $U:\R\to(-\infty,\infty]$ satisfies
$U(x)=U(-x)$ and the condition~\eqref{eq:U_integral_cond}. Then for
any $n\ge 1$ the measure $\mu_{\T_n^2, \zero, U}$ is reflection
positive with respect to $\F$.
\end{prop}
\begin{proof}
Suppose $\varphi$ is randomly sampled from $\mu_{\T_n^2, \zero, U}$.
Fix a plane of reflection $P$, a vertex $v_0\in P$ and suppose
$\tilde{\varphi}$ is randomly sampled from $\mu_{\T_n^2, v_0, U}$
(the measure $\mu_{\T_n^2,v_0,U}$ is obtained by replacing $\zero$
with $v_0$ in \eqref{eq:mu_T_n_2_U_measure_def}). We write
$\mathbb{E}_{\mu_{\T_n^2,\zero,U}}$ and
$\mathbb{E}_{\mu_{\T_n^2,v_0,U}}$ for the expectation operators
corresponding to $\varphi$ and $\tilde{\varphi}$, respectively.
Observe that
\begin{equation}\label{eq:change_base_point}
f(\varphi) \,{\buildrel d \over =}\,
f(\tilde{\varphi})\quad\text{for any $f\in\F$}
\end{equation}
since the induced measure on the gradient of $\varphi$ is
translation invariant. In addition, by symmetry,
\begin{equation}\label{eq:reflection_symmetry_of_measure}
  \tilde{\varphi} \,{\buildrel d \over =}\, \theta_P
  \tilde{\varphi}.
\end{equation}
For two bounded $f, g\in \F_P^+$ the relation
\eqref{eq:reflectionIsCool} now follows from
\eqref{eq:change_base_point} and
\eqref{eq:reflection_symmetry_of_measure} by
\begin{equation*}
  \mathbb{E}_{\mu_{\T_n^2,\zero,U}}(f\,\theta_P g)=\mathbb{E}_{\mu_{\T_n^2,v_0,U}}(f\,\theta_P
  g) = \mathbb{E}_{\mu_{\T_n^2,v_0,U}}(\theta_P(f\,\theta_P g)) = \mathbb{E}_{\mu_{\T_n^2,v_0,U}}(g\,\theta_P f) = \mathbb{E}_{\mu_{\T_n^2,\zero,U}}(g\,\theta_P
  f).
\end{equation*}
To see the relation \eqref{eq:reflectionIsPositive} observe that, by
the domain Markov property and symmetry, conditioned on
$(\tilde{\varphi}_v)_{v\in\bar{P}}$ the configurations
$(\tilde{\varphi}_v)_{v\in P^+}$ and $((\theta_P
\tilde{\varphi})_v)_{v\in P^+}$ are independent and identically
distributed. Thus, for any $f\in\F_P^+$ we have
\begin{align*}
  \mathbb{E}_{\mu_{\T_n^2,\zero,U}}(f\,\theta_P f) &= \mathbb{E}_{\mu_{\T_n^2,v_0,U}}(f\,\theta_P
  f) = \mathbb{E}_{\mu_{\T_n^2,v_0,U}}\left(\mathbb{E}_{\mu_{\T_n^2,v_0,U}}\left(f\,\theta_P
  f\,|\,(\tilde{\varphi}_v)_{v\in\bar{P}}\right)\right) =\\
  &= \mathbb{E}_{\mu_{\T_n^2,v_0,U}}\left(\mathbb{E}_{\mu_{\T_n^2,v_0,U}}\left(f\,|\,(\tilde{\varphi}_v)_{v\in\bar{P}}\right)\mathbb{E}_{\mu_{\T_n^2,v_0,U}}\left(\theta_P
  f\,|\,(\tilde{\varphi}_v)_{v\in\bar{P}}\right)\right) =\\
  &= \mathbb{E}_{\mu_{\T_n^2,v_0,U}}\left(\mathbb{E}_{\mu_{\T_n^2,v_0,U}}\left(f\,|\,(\tilde{\varphi}_v)_{v\in\bar{P}}\right)^2\right)\ge
  0.\qedhere
\end{align*}
\end{proof}
We now prove Theorem~\ref{thm:controlledGradients}. Fix
$0<\delta<1$, $n\ge 1$ and suppose $\varphi$ is randomly sampled
from $\mu_{\T_n^2, \zero, U}$. Let $K$ be the constant from
\eqref{eq:interval_potential}, where we write $K=\infty$ if
$U(x)<\infty$ for all $x$. Recall the definition of the random graph $\EC(\varphi,L)$ from
\eqref{eq:E_varphi_L_random_graph}. For an edge $e=(v,w)\in
E(\T_n^2)$ and $0<L<\infty$ define the function $f_{e,L}\in \F$ by
\begin{equation*}
  f_{e,L}(\psi):=1_{(|\psi_v - \psi_w|\ge L)}.
\end{equation*}
We need to show that there exists some $0<L<K$, independent of $n$,
such that
\begin{equation*}
  \E\left(\prod_{i=1}^k f_{e_i,L}\right)\le \delta^k\quad\text{for all $k\ge 1$ and distinct $e_1,\ldots, e_k\in E(\T_n^2)$}.
\end{equation*}
Fix some $k\ge 1$ and distinct $e_1,\ldots, e_k\in E(\T_n^2)$.
Define four block functions at $(0,0)$ by
\begin{align*}
  f_L^{\text{hor,0}}(\psi)&:=1_{(|\psi_{(1,0)} - \psi_{(0,0)}|\ge L)},\qquad f_L^{\text{ver,0}}(\psi):=1_{(|\psi_{(0,1)} - \psi_{(0,0)}|\ge
  L)},\\
  f_L^{\text{hor,1}}(\psi)&:=1_{(|\psi_{(1,1)} - \psi_{(0,1)}|\ge L)},\qquad f_L^{\text{ver,1}}(\psi):=1_{(|\psi_{(1,1)} - \psi_{(1,0)}|\ge L)}.
\end{align*}
The definition \eqref{eq:reflection_operator_def} of the reflection
operators $(\vartheta_t)$ implies that there exist
$k_1,k_2,k_3,k_4\ge 0$ with $k_1+k_2+k_3+k_4=k$ and, for each $1\le
j\le 4$, distinct $(t_{j,i})_{1\le i\le k_j}\subseteq V(\T_n^2)$
such that
\begin{equation*}
  \E\left(\prod_{i=1}^k f_{e_i,L}\right) = \E\left(\prod_{i=1}^{k_1}
  \vartheta_{t_{1,i}} f_L^{\text{hor,0}}\prod_{i=1}^{k_2}
  \vartheta_{t_{2,i}} f_L^{\text{hor,1}}\prod_{i=1}^{k_3}
  \vartheta_{t_{3,i}} f_L^{\text{ver,0}}\prod_{i=1}^{k_4}
  \vartheta_{t_{4,i}} f_L^{\text{ver,1}}\right).
\end{equation*}
Assume, without loss of generality, that $k_1\ge k/4$ (as the cases
that $k_j\ge k/4$ for some $2\le j\le 4$ follow analogously). Then,
by the chessboard estimate, Theorem~\ref{thm:checkboard},
\begin{equation*}
  \E\left(\prod_{i=1}^k f_{e_i,L}\right)\le \E\left(\prod_{i=1}^{k_1}
  \vartheta_{t_{1,i}} f_L^{\text{hor}}\right)\le \Bigg(\E\Bigg(\prod_{t\in \T_n^2} \vartheta_t
  f_L^{\text{hor}}\Bigg)\Bigg)^{\frac{k_1}{|V(\T_n^2)|}}
\end{equation*}
and thus it suffices to show that there exists some $0<L<K$,
independent of $n$, such that
\begin{equation}\label{eq:no_horizontal_extremal_edges_estimate}
  \E\left(\prod_{t\in \T_n^2} \vartheta_t f_L^{\text{hor}}\right)\le
  \delta^{4|V(\T_n^2)|}.
\end{equation}
We note that
\begin{equation}\label{eq:E_L_introduction}
  \E\left(\prod_{t\in \T_n^2} \vartheta_t f_L^{\text{hor}}\right) =
  \P(\varphi\in E_L)
\end{equation}
where
\begin{equation*}
  E_L:=\{\psi\in \R^{V(\T_n^2)}\colon |\psi_{(j+1, k)} - \psi_{j,k}|\ge
  L\text{ for all $-n+1\le j\le n$ and all even $-n+1\le k\le n$}\}.
\end{equation*}
Thus, recalling \eqref{eq:mu_T_n_2_U_measure_def}, we have
\begin{equation}\label{eq:E_L_explicit_probability}
  \P(\varphi\in E_L) = \frac{1}{Z_{\T_n^2,\zero, U}} \int_{E_L}\exp\Bigg(-\sum_{(v,w)\in E(\T_n^2)}
  U(\psi_v - \psi_w)\Bigg) \delta_0(d\psi_{\zero})\prod_{v\in V(\T_n^2)\setminus\{\zero\}}
  d\psi_v=:\frac{Z_{\T_n^2,\zero, U}(E_L)}{Z_{\T_n^2,\zero, U}}.
\end{equation}
We estimate the numerator and denominator in the last fraction
separately. First, we have already shown a lower bound on
$Z_{\T_n^2,\zero, U}$ in \eqref{eq:partition_function_lower_bound}.
Second, denote by $H$ the subset of edges $((j, k),(j+1,k))\in
E(\T_n^2)$ for which $k$ is even. Let $S$ be a spanning tree of
$\T_n^2$, regarded here as a subset of edges, satisfying
\begin{equation}\label{eq:S_and_H_intersection}
  |S\cap H| \ge \frac{1}{10}|E(\T_n^2)|.
\end{equation}
Then
\[
Z_{\T_n^2,\zero, U}(E_L)\leq C_1(U)^{|E(\T_n^2)\setminus
S|}\int_{E_L}\exp\Bigg(-\sum_{(v,w)\in S}
  U(\psi_v - \psi_w)\Bigg) \delta_0(d\psi_{\zero})\prod_{v\in V(\T_n^2)\setminus\{\zero\}}
  d\psi_v
\]
where $C_1(U)\eqdef \sup_x\, \exp(-U(x))<\infty$ by
\eqref{eq:U_integral_cond}. The integral above can be estimated by
integrating the vertices in $V(\T_n^2)\setminus\{\zero\}$ leaf by
leaf according to the spanning tree $S$. Recalling the definition of
$E_L$, two cases arise depending on whether or not the edge
connecting a leaf to the remaining tree belongs to $H$. Thus we
obtain
\begin{equation*}
  Z_{\T_n^2,\zero, U}(E_L)\le C_1(U)^{|E(\T^{2})\setminus
S|} C_2(U)^{|S\setminus H|} C_3(U,L)^{|S\cap H|}
\end{equation*} where
\begin{equation*}
C_2(U)\eqdef\int \exp\rbr{-U(x)}\dd x\quad\text{and}\quad
C_3(U,L)\eqdef\int 1_{|x|\geq L}\exp\rbr{-U(x)}\dd x.
\end{equation*}
Condition~\eqref{eq:U_integral_cond} ensures that $C_2(U)<\infty$
and the definition of $K$ gives that $\lim_{L\uparrow K} C_3(U,L) =
0$. Thus, using \eqref{eq:S_and_H_intersection}, for every $\eps>0$
there exists an $0<L<K$, independent of $n$, for which
\begin{equation*}
  Z_{\T_n^2,\zero, U}(E_L)\le \eps^{|V(\T_n^2)|}.
\end{equation*}
This inequality, together with \eqref{eq:E_L_introduction},
\eqref{eq:E_L_explicit_probability} and
\eqref{eq:partition_function_lower_bound}, implies that we may
choose an $0<L<K$, independent of $n$, so that
\eqref{eq:no_horizontal_extremal_edges_estimate} holds, as we wanted
to show.

\subsection{Proof of the chessboard
estimate}\label{sec:proof_of_chessboard_estimate} In this section we
prove Theorem~\ref{thm:checkboard}.

Let $\varphi$ be randomly sampled from the given measure $\P$.
Reflection positivity of $\P$ with respect to $\F$ implies that for
each plane of reflection $P$, the bilinear form $E(g\theta_P h)$ is
a degenerate inner product on bounded $g,h\in \F_{P^+}$. In
particular, we have the Cauchy-Schwartz inequality,
\begin{equation}\label{eq:C-S_reflection_positivity}
  |E g\theta_P h| \le \sqrt{\E(g\theta_P g)\E(h\theta_P h)},\quad \text{for all bounded }g,h\in \F_{P^+}.
\end{equation}

For a function $f\in \F$ of the form
\begin{equation}\label{eq:product_block_functions}
  f(\varphi) = \prod_{t\in V(\T_n^2)} \vartheta_t f_t\quad\text{for some $(f_t)$, bounded block functions at
  $(0,0)$}
\end{equation}
and a plane of reflection $P$, define two functions, the ``parts of
$f$ in $P^-$ and $P^+$'', by
\begin{equation*}
  f_{P^-}:=\prod_{t\in P^-\setminus P} \vartheta_t
  f_t\in\F_{P^-}\quad\text{and}\quad f_{P^+}:=\prod_{t\in P^+\setminus (\bar{P}\setminus P)} \vartheta_t
  f_t\in\F_{P^+}.
\end{equation*}
Define also the function $\rho_P f\in \F$ by
\begin{equation*}
  \rho_P f:= f_{P^+} \theta_P f_{P^+}
\end{equation*}
and note that $\E(\rho_P f)\ge 0$ by
\eqref{eq:reflectionIsPositive}. Observe that
\begin{equation*}
  f = f_{P^+} f_{P^-} = f_{P^+} \theta_P (\theta_P f_{P^-}).
\end{equation*}
Thus, using the Cauchy-Schwartz inequality
\eqref{eq:C-S_reflection_positivity} with $g=f_{P^+}$ and $h =
\theta_P f_{P^-}$ we have
\begin{equation}\label{eq:Cauchy-Schwartz_inequality_RP}
  |\E(f)| \le \sqrt{\E(f_{P^+} \theta_P f_{P^+}) \E (f_{P^-} \theta_P
  f_{P^-})}
  = \sqrt{\E(\rho_P f) \E(\rho_{\bar{P}\setminus P} f)}.
\end{equation}

Our first goal is to show that starting with a function of the form
\eqref{eq:product_block_functions}, one may iteratively apply the
operator $\rho_P$ with different planes of reflection $P$ to reach a
function of the form \eqref{eq:product_block_functions} with all the
block functions identical.
\begin{prop}\label{prop:homogenization_prop}
  For each $s\in V(\T_n^2)$ there exists a sequence of planes of
  reflection $P_1,\ldots, P_m$ such that for each $f$ of the form
  \eqref{eq:product_block_functions} we have
  \begin{equation*}
    \rho_{P_m}\rho_{P_{m-1}}\cdots\rho_{P_1} f = \prod_{t\in
    V(\T_n^2)} \vartheta_t f_s.
  \end{equation*}
\end{prop}
\begin{proof}
  Let $s = (j,k)\in V(\T_n^2)$. Define the vertical planes of reflection $(Q_i)$, $0\le i\le \lceil\log_2(n)\rceil$, by $Q_i:=P_{j_i}^{\text{ver}}$ for $j_i
:= j + 1 - 2^i$ modulo $2n$. One may verify directly that
\begin{equation*}
  \rho_{Q_{\lceil\log_2(n)\rceil}}\cdots\rho_{Q_1}\rho_{Q_0} f = \prod_{t\in
    V(\T_n^2)} \vartheta_t f_{\pi(t)}
\end{equation*}
for some $\pi:V(\T_n^2)\to V(\T_n^2)$ satisfying that $\pi((a,b))
  = (j,b)$ for all $-n+1\le a\le n$. In the same manner, one may
  now take the horizontal planes of reflection $(R_i)$, $0\le i\le \lceil\log_2(n)\rceil$, defined by $R_i:=P_{k_i}^{\text{hor}}$ for $k_i
:= k + 1 - 2^i$ modulo $2n$, and conclude that
\begin{equation*}
  \rho_{R_{\lceil\log_2(n)\rceil}}\cdots\rho_{R_1}\rho_{R_0}\rho_{Q_{\lceil\log_2(n)\rceil}}\cdots\rho_{Q_1}\rho_{Q_0} f = \prod_{t\in
    V(\T_n^2)} \vartheta_t f_{s},
\end{equation*}
as required.
\end{proof}

For a bounded block function $f_0$ at $(0,0)$ define
\begin{equation*}
  \|f_0\|:=\Bigg(\E\Bigg(\prod_{t\in V(\T_n^2)} \vartheta_t
  f_0\Bigg)\Bigg)^{\frac{1}{|V(\T_n^2)|}}
\end{equation*}
which is well-defined and non-negative by \eqref{eq:reflectionIsPositive}. Let $f$
have the form \eqref{eq:product_block_functions}. With the above
notation, the chessboard estimate \eqref{eq:chessboard_estimate}
becomes the inequality
\begin{equation}\label{eq:chessboard_estimate_inequality_in_proof}
  |\E(f)| \le \prod_{s\in V(\T_n^2)} \|f_s\|,
\end{equation}
where we note that in Theorem~\ref{thm:checkboard} we may assume
that $m=|V(\T_n^2)|$ by taking some of the block functions to be
constant.

Consider first the case that
\begin{equation}\label{eq:some_norm_0}
\|f_s\| = 0\quad\text{for some $s\in V(\T_n^2)$}.
\end{equation}
Let $P_1,\ldots, P_m$ be the planes of reflection corresponding to
$s$ as given by Proposition~\ref{prop:homogenization_prop}. By
iteratively applying the Cauchy-Schwartz inequality
\eqref{eq:Cauchy-Schwartz_inequality_RP} with the planes $(P_i)$ we
may obtain that $|\E(f)|$ is bounded by a product in which
$\|f_s\|$, raised to some positive power, is one of the factors.
Thus we conclude from \eqref{eq:some_norm_0} that $\E(f) = 0$,
establishing \eqref{eq:chessboard_estimate_inequality_in_proof} in
this case.

Second, assume that \eqref{eq:some_norm_0} does not hold. Define
\begin{equation*}
  g_s:=\frac{f_s}{\|f_s\|},\quad s\in V(\T_n^2).
\end{equation*}
Let $h\in\F$ be an (arbitrary) function maximizing $|\E(h)|$ among
all functions of the form
\begin{equation}\label{eq:form_of_h}
  h = \prod_{t\in V(\T_n^2)} \vartheta_t h_t\quad\text{with each $h_t$ being one of the
  $(g_s)$}.
\end{equation}
Observe that, by the Cauchy-Schwartz inequality
\eqref{eq:Cauchy-Schwartz_inequality_RP} and the definition of $h$,
we have
\begin{equation*}
  |\E(h)|\le \sqrt{\E(\rho_P h) \E(\rho_{\bar{P}\setminus P} h)}\le \sqrt{\E(\rho_P h) |\E(h)|}\quad\text{for any plane of reflection $P$}.
\end{equation*}
Thus,
\begin{equation}\label{eq:E_h_bound}
  |\E(h)|\le \E(\rho_P h)\quad\text{for any plane of reflection $P$}.
\end{equation}
In particular, $\E(\rho_P h)$ also maximizes $|\E(h)|$ among
functions of the form \eqref{eq:form_of_h} (so that equality holds
in the last inequality). Let $P_1,\ldots, P_m$ be the planes of
reflection corresponding to $s = \zero$ as given by
Proposition~\ref{prop:homogenization_prop}. By iteratively applying
\eqref{eq:E_h_bound} with these planes we obtain that
\begin{equation*}
  |\E(h)|\le \E(\rho_{P_m}\rho_{P_{m-1}}\cdots\rho_{P_1} h) =
  \|h_\zero\|^{|V(\T_n^2)|} = 1
\end{equation*}
since $\|g_s\|=1$ for all $s$ and $h$ has the form
\eqref{eq:form_of_h}. Finally, the definition of $h$ now shows that
\begin{equation*}
  \frac{|\E(f)|}{\prod_{s\in V(\T_n^2)} \|f_s\|} \le |\E(h)| \le 1
\end{equation*}
implying \eqref{eq:chessboard_estimate_inequality_in_proof} and
finishing the proof of Theorem~\ref{thm:checkboard}.

\section{Lower bound for random surface fluctuations in two
dimensions}\label{sec:proofs_for_main_theorems} Recall the
definition of the controlled gradients property from
Section~\ref{sec:reflection_positivity}. Throughout the section we
fix $n\ge 2$ and a potential $U$ with the following properties:
\begin{itemize}
  \item There exists an $0<\eps\le 1/2$ for which $U$ has $(1/8, 1-\eps)$-controlled gradients on
$\T_n^2$.
  \item $U$ restricted to $[-1,1]$ is twice continuously differentiable.
\end{itemize}
We fix $\eps$ to the value given by the first property. Write $\zero:=(0,0)$. For the rest of the section we suppose that
$\varphi$ is a random function sampled from the probability
distribution $\mu_{\T_n^2, \zero, U}$ defined in
\eqref{eq:mu_T_n_2_U_measure_def}. For a vertex $v=(v_1,v_2)$ of
$\T_n^2$ we write $\|v\|_1:=|v_1| + |v_2|$.
\begin{thm}\label{thm:variance_lower_bound}
There exist constants $C(U),c(U)>0$ such that for any $v\in
V(\T_n^2)$ with $\|v\|_1\ge (\log n)^2$ we have
\begin{align}
  &\var(\varphi_v) \ge c(U) \log(1+\|v\|_1),\label{eq:variance_estimate}\\
  &\P(|\varphi_v|\le r)\le C(U)\left(\frac{r}{\sqrt{\log(1+\|v\|_1)}}\right)^{2/3},\quad r\ge 1,\label{eq:anti-concentration_estimate}\\
  &\P(|\varphi_v| \ge t\sqrt{\log(1+\|v\|_1)})\ge
  c(U)e^{-C(u)t^2},\quad 1\le t\le \frac{1+\sqrt{\|v\|_1}}{1+\log n}\label{eq:large_deviation_estimate}.
\end{align}
\end{thm}
The theorem establishes lower bounds for the variance and large
deviation probabilities of $\varphi_v$ as well as upper bounds on
the probability that $\varphi_v$ is atypically small. The lower
bound on the variance is expected to be sharp up to the value of
$c(U)$.

The theorem is not optimal in several ways. One expects the results
to hold for all $v\in V(\T_n^2)$ without the restriction on
$\|v\|_1$, one expects that the exponent $2/3$ may be replaced by
$1$ and that the restrictions on $r$ and $t$ may be relaxed. We
believe that further elaboration of our methods may address some of
these issues. However, since our main focus is on vertices $v$ for
which $\|v\|_1$ is of order $n$ and on estimating the variance of
$\varphi_v$ we prefer to present simpler proofs.

\begin{thm}\label{thm:maximum_lower_bound}
There exists a constant $c(U)>0$ such that
\begin{equation*}
  \P\left(\max_{v\in V(\T_n^2)} |\varphi_v|\ge c(U)\log n\right) \ge \frac{1}{2}.
\end{equation*}
\end{thm}
Again, this estimate is expected to be sharp up to the value of
$c(U)$.

\subsection{Tools}
In this section we let $\tau:V(\T_n^2)\to[0,\infty)$ be an arbitrary
function satisfying $\tau(\zero)=0$. We let $T^+, T^-$ be the
functions defined in Section~\ref{sec:Shift-Algorithm-and} acting on
the graph $\T_n^2$ with the given $\tau$ function and constant
  $\eps$. We also recall the notation $J^+, J^-, M(\varphi)$ and $L(\tau,\eps)$ from Section~\ref{sec:addition_algorithm_prop}. Our main tool for lower
bounding the fluctuations of $\varphi$ is the following lemma.
\begin{lem} \label{lem:lowerboundForShiftedConfiguration}
Denote
\begin{equation*}
  V_0:=\{v\in V(\T_n^2)\colon \tau(v) = 0\}
\end{equation*}
and let $\F_0$ be the sigma-algebra generated by $(\varphi_v)$,
$v\in V_0$. There exists a constant $c(U)>0$ such that for any $a,s>0$, any $u\in V(\T_n^2)$ and any event
$A\in\F_0$ we have
\begin{align}\label{eq:shiftLowerbound}
\bigg[\P\bigg(&\Big\{|\varphi_{u} - \tau(u)|\le a + \frac{\eps}{2}\Big\}\cap A\bigg)\P\bigg(\Big\{|\varphi_{u} + \tau(u)|\le a + \frac{\eps}{2}\Big\}\cap A\bigg)\bigg]^{1/2}\ge\nonumber \\
 &\ge
c(U)s\left[\P\left(\{|\varphi_u|\le a\}\cap A\right) -
\P\left(\left(\{J^+(\varphi)J^-(\varphi) <
s^2\}\cup\{M(\varphi)>L(\tau,\eps)\}\right)\cap A\right)\right].
\end{align}
\end{lem}
\begin{proof}
  Write
  \begin{equation*}
    g(\psi):=\frac{1}{Z_{\T_n^2, \zero, U}} \exp\Bigg(-\sum_{(v,w)\in E(\T_n^2)}
  U(\psi_v - \psi_w)\Bigg)
  \end{equation*}
  for the density of the measure $\mu_{\T_n^2, \zero, U}$. Fix a
  function $\theta:V_0\to \R$ satisfying $\theta(\zero)=0$ and denote by $d\lambda$ the measure
  \begin{equation*}
    d\lambda(\psi) = \delta_{0}(d\psi_\zero) \prod_{v\in V\setminus\{\zero\}} d\psi_v.
  \end{equation*}
  Define the event
  \begin{equation*}
    E:=\{\psi\in\R^{V(\T_n^2)}\colon |\psi_u|\le a, J^+(\psi)J^-(\psi)\ge s^2, M(\psi)\le L(\tau,\eps)\}
  \end{equation*}
  and the quantity
  \begin{equation*}
    I:=\int_{E\cap A}
    \sqrt{g(T^+(\psi))g(T^-(\psi))J^+(\psi)J^-(\psi)}d\lambda(\psi).
  \end{equation*}
  We wish to bound $I$ from below and from above. We start with the
  bound from below.

  Since $U$ restricted to $[-1,1]$ is twice continuously differentiable there exists some $0<c(U)\le 1$ such that
  \begin{equation}\label{eq:U_twice_differentiable_property}
    \exp\left(-\frac{1}{2}(U(x+r)+U(x-r))\right)\ge c(U)\exp(-U(x))
  \end{equation}
  for all $x,r\in\R$ for which $x+r, x-r\in [-1,1]$.

  Abbreviate $\sigma_v:=T^+(\psi)_v - \psi_v = \psi_v - T^-(\psi)_v$ (using \eqref{eq:T_+_T_-_relation}) and observe that
  \begin{align*}
    \sqrt{g(T^+(\psi))g(T^-(\psi))} &= \frac{1}{Z_{\T_n^2, \zero, U}} \exp\Bigg(-\frac{1}{2}\sum_{(v,w)\in E(\T_n^2)}
  U(\psi_v - \psi_w + \sigma_v - \sigma_w) + U(\psi_v -
  \psi_w - \sigma_v + \sigma_w)\Bigg)\ge\\
  &\ge \frac{c(U)}{Z_{\T_n^2, \zero, U}}\exp\Bigg(-\sum_{(v,w)\in E(\T_n^2)}
  U(\psi_v - \psi_w)\Bigg) = c(U)g(\psi),
  \end{align*}
  where we have used
  property~\eqref{property:Lipschitz_preservation} from
  Section~\ref{sec:addition_algorithm_prop} to justify our use of
  \eqref{eq:U_twice_differentiable_property}.
  Together with the definition of the event $E$ this implies that
  \begin{equation}\label{eq:I_lower_bound}
    I \ge c(U)s\int_{E\cap A} g(\psi)d\lambda(\psi).
  \end{equation}

  To bound $I$ from above we use the Cauchy-Schwartz inequality and the Jacobian identity in \eqref{eq:Jacobian_formula_properties_section} to
  obtain
  \begin{align}
    I&\le \left(\int_{E\cap A}
    g(T^+(\psi))J^+(\psi)d\lambda(\psi)\int_{E\cap A}
    g(T^-(\psi))J^-(\psi)d\lambda(\psi)\right)^{\frac{1}{2}} =\nonumber\\
    &=\left(\int_{T^+(E\cap A)} g(\psi)d\lambda(\psi)\int_{T^-(E\cap A)}
    g(\psi)d\lambda(\psi)\right)^{\frac{1}{2}}.\label{eq:I_upper_bound}
  \end{align}
Comparing \eqref{eq:I_lower_bound} and \eqref{eq:I_upper_bound} and
recalling that $\varphi$ is sampled from the probability
distribution $\mu_{\T_n^2, \zero, U}$ we conclude that
\begin{equation}\label{eq:main_use_of_transformation_with_events}
  \left[\P\big(\varphi\in T^+(E\cap A)\big) \P\big(\varphi\in T^-(E\cap A)\big)\right]^{\frac{1}{2}}\ge c(U)s\P\big(\varphi\in E\cap A\big).
\end{equation}
We continue by noting that by the definition of $E$,
\begin{equation*}
  \P\big(\varphi\in E\cap A\big) \ge \P\big(\{|\varphi_u|\le a\}\cap A\big) - \P\big((\{J^+(\varphi)J^-(\varphi)<s^2\}\cup\{M(\varphi)> L(\tau,\eps)\})\cap A\big).
\end{equation*}
In addition, we recall from
properties~\eqref{property:maximal_increment} and
\eqref{property:addition_lower_bound} of $T^+$ in
Section~\ref{sec:addition_algorithm_prop} that if $\psi$ satisfies
$|\psi_u|\le a$ and $M(\psi)\le L(\tau,\eps)$ then $-a-
\frac{\eps}{2}\le T^+(\psi)_u-\tau(u)\le a$ and a similar relation
for $T^-$ by \eqref{eq:T_+_T_-_relation}. In addition, since
$A\in\F_0$, properties (1) and (2) imply that $A = T^+(A) = T^-(A)$.
Therefore, using that $T^+$ and $T^-$ are one-to-one,
\begin{multline*}
  \P\big(\varphi\in T^+(E\cap A)\big) \P\big(\varphi\in T^-(E\cap A)\big)=\\
  =\P\big(\varphi\in T^+(E)\cap A\big) \P\big(\varphi\in T^-(E)\cap A\big)\le\\
  \le \P\big(\{|\varphi_{u} -
\tau(u)|\le a + \frac{\eps}{2}\}\cap A\big) \P\big(\{|\varphi_{u} +
\tau(u)|\le a + \frac{\eps}{2}\}\cap A\big).
\end{multline*}
Combining the last two inequalities with
\eqref{eq:main_use_of_transformation_with_events} establishes the
lemma.
\end{proof}

Our next lemma bounds the error terms appearing on the right-hand
side of \eqref{eq:shiftLowerbound}.
\begin{lem}\label{lem:Jacobian_and_bad_edge_diameter}
  For any $s>0$ we have
  \begin{equation*}
    \P\left(\{J^+(\varphi)J^-(\varphi) <
s^2\}\cup\{M(\varphi)>L(\tau,\eps)\}\right)\le (2n)^2
2^{-L(\tau,\eps)} + \frac{4\sum_{v\in V(\T_n^2)} \sum_{k=0}^{\infty}
2^{-k}
\tau'\left(v,k+1\right)^2}{\eps^2\log\left(\frac{1}{s}\right)}.
  \end{equation*}
\end{lem}
\begin{proof}
  Given a vertex $v\in V(\T_n^2)$ and $k\ge 1$ denote by $\mathcal{P}_{v,k}$ the set of all simple paths in $\T_n^2$ starting at $v$ and having length $k$.
  Here, by such a path we mean a vector $(e_1,
  \ldots, e_k)\subseteq E(\T_n^2)$ of distinct edges with $e_i = (v_i, v_{i+1})$ and $v=v_1$.
  Observe that, trivially, $|\mathcal{P}_{v,k}|\le 4^k$ for all $v$
  and $k$. Now note that since $U$ has $(1/8, 1-\eps)$-controlled
  gradients on $\T_n^2$ we have for each $v\in V(\T_n^2)$ and $k\ge 1$,
  \begin{equation}\label{eq:bad_edges_connectivity_estimate}
    \P(r(\varphi, v)\ge k) \le \sum_{(e_1,\ldots, e_k)\in \mathcal{P}_{v,k}}
    \P(e_1,\ldots, e_k\in \EC(\varphi))\le 4^k \left(\frac{1}{8}\right)^k = 2^{-k}.
  \end{equation}

  Observe that
  \begin{multline}\label{eq:Jacobian_and_diameter_estimate}
    \P\left(\{J^+(\varphi)J^-(\varphi) <
s^2\}\cup\{M(\varphi)>L(\tau,\eps)\}\right) =\\
= \P\left(M(\varphi)>L(\tau,\eps)\right) +
\P\left(\{J^+(\varphi)J^-(\varphi) < s^2\}\cap\{M(\varphi)\le
L(\tau,\eps)\}\right).
  \end{multline}
  We estimate each of the terms on the right-hand side separately.

  First, using \eqref{eq:bad_edges_connectivity_estimate} we have
  \begin{equation}\label{eq:M_L_comparison_estimate}
    \P\left(M(\varphi)>L(\tau,\eps)\right)\le |V(\T_n^2)|
    2^{-L(\tau,\eps)} \le (2n)^2 2^{-L(\tau,\eps)},
  \end{equation}
  observing that the inequality holds trivially if $L(\tau,\eps)$ is
  zero or negative.

  Second, using property~\eqref{property:Jacobian_estimate} from
  Section~\ref{sec:addition_algorithm_prop} we see that
  \begin{equation}\label{eq:using_the_Jacobian_estimate}
    \P\left(\{J^+(\varphi)J^-(\varphi) < s^2\}\cap\{M(\varphi)\le
L(\tau,\eps)\}\right)\le \P\left(\sum_{v\in V(\T_n^2)}
\tau'\left(v,1+\max_{w\sim
v}r(\varphi,w)\right)^2>\eps^2\log\left(\frac{1}{s}\right)\right).
  \end{equation}
  Now,
  \begin{align*}
    \E\left(\sum_{v\in V(\T_n^2)}
\tau'\left(v,1+\max_{w\sim v}r(\varphi,w)\right)^2\right) &\le \E
\left(\sum_{v\in V(\T_n^2)} \sum_{w\sim
v}\tau'\left(v,1+r(\varphi,w)\right)^2\right)=\\
&= \sum_{v\in V(\T_n^2)} \sum_{w\sim v} \sum_{k=0}^{\infty}
\tau'\left(v,1+k\right)^2\P(r(\varphi,w)=k)
  \end{align*}
  and using again \eqref{eq:bad_edges_connectivity_estimate} we conclude
  that
  \begin{equation*}
    \E\left(\sum_{v\in V(\T_n^2)}
\tau'\left(v,1+\max_{w\sim v}r(\varphi,w)\right)^2\right) \le
4\sum_{v\in V(\T_n^2)} \sum_{k=0}^{\infty} 2^{-k}
\tau'\left(v,1+k\right)^2.
  \end{equation*}
  Thus, Markov's inequality and
  \eqref{eq:using_the_Jacobian_estimate} show that
  \begin{equation*}
    \P\left(\{J^+(\varphi)J^-(\varphi) < s^2\}\cap\{M(\varphi)\le
L(\tau,\eps)\}\right)\le \frac{4\sum_{v\in V(\T_n^2)}
\sum_{k=0}^{\infty} 2^{-k}
\tau'\left(v,1+k\right)^2}{\eps^2\log\left(\frac{1}{s}\right)}.
  \end{equation*}
  The lemma follows by combining this estimate with
  \eqref{eq:Jacobian_and_diameter_estimate} and
  \eqref{eq:M_L_comparison_estimate}.
\end{proof}

\subsection{Fluctuation bounds}
In this section we prove Theorem~\ref{thm:variance_lower_bound}.

Fix $v\in V(\T_n^2)\setminus\{\zero\}$. Define the increasing
function $h:[0,\infty)\to[0,\infty)$ by
\begin{equation*}
  h(x) := \frac{\log(1 + x)}{\sqrt{\log(1 + \|v\|_1)}}
\end{equation*}
and the function $\eta:V(\T_n^2)\to[0,\infty)$ by
\begin{equation}\label{eq:eta_def}
\eta(w):=\begin{cases}0&\|w\|_1\le
\sqrt{\|v\|_1}\\h(\|w\|_1) - h(\sqrt{\|v\|_1})&\sqrt{\|v\|_1}\le \|w\|_1 \le \|v\|_1\\
  h(\|v\|_1) - h(\sqrt{\|v\|_1})& \|w\|_1\ge \|v\|_1\end{cases}.
\end{equation}
We aim to use the lemmas of the previous section with the $\tau$
function a constant multiple of $\eta$. The above definition is
chosen so that we may control the quantities appearing in
Lemma~\ref{lem:Jacobian_and_bad_edge_diameter}. The first case
allows us to lower bound the function $L$ while the second and third
cases ensure that $\eta$ is slowly varying. The next lemma
formalizes these ideas.
Write, as in \eqref{eq:tau_prime_def},
\begin{equation}\label{eq:eta_prime_def}
\eta'(w,k):=\max(\eta(w) - \eta(u)\colon u\in V(\T_n^2),\,
d_{\T_n^2}(w,u)\le k\},\quad w\in V(\T_n^2),\, k\ge 1.
\end{equation}
\begin{lem}\label{lem:eta_Jacobian_and_bad_edge_bound}
  There exists an absolute constant $C>0$ such that
  \begin{equation*}
    \sum_{w\in V(\T_n^2)} \sum_{k=0}^{\infty} 2^{-k}
\eta'\left(w,k+1\right)^2 \le C.
  \end{equation*}
  For any $\alpha>0$ we have
  \begin{equation}\label{eq:L_with_alpha_and_eta}
    L(\alpha\cdot\eta,\eps)\ge
    \left\lfloor(1+\sqrt{\|v\|_1})\left(\exp\left(\frac{\eps\sqrt{\log(1
    + \|v\|_1)}}{2\alpha}\right) - 1\right)\right\rfloor - 1.
  \end{equation}
\end{lem}
\begin{proof}
The fact that $\eta(w)$ depends only on $\|w\|_1$ and $\eta(w_1)\ge
\eta(w_2)$ when $\|w_1\|_1\ge \|w_2\|_1$ shows that for each $w\in
V(\T_n^2)$ and $k\ge 0$ we have
\begin{equation*}
  0\le \eta'(w,k+1)\le \begin{cases}
    0&\|w\|_1 > \|v\|_1 + k+1\\
    h(\|w\|_1) - h(\|w\|_1 - (k+1))& k+1\le\|w\|_1
\le \|v\|_1 + k+1\\
    h(\|w\|_1) - h(0)& \|w\|_1 < k+1
  \end{cases}.
\end{equation*}
By considering separately the latter two cases in the above
inequality we have
\begin{align*}
    &\sum_{w\in V(\T_n^2)} \sum_{k=0}^{\infty} 2^{-k}
\eta'\left(w,k+1\right)^2 \le\\
&\le 4\sum_{t=0}^{\|v\|_1}\sum_{k=0}^\infty 2^{-k}(t+k+1)(h(t+k+1) -
h(t))^2 + 4\sum_{m=1}^\infty \sum_{k=m}^\infty 2^{-k}m
(h(m)-h(0))^2,
\end{align*}
where we have also used that there are at most $4m$ vertices $w\in
V(\T_n^2)$ with $\|w\|_1 = m$ (strict inequality is possible when
$m\ge n$). Continuing the last inequality we obtain

%

%
  \begin{align*}
    &\sum_{w\in V(\T_n^2)} \sum_{k=0}^{\infty} 2^{-k}
\eta'\left(w,k+1\right)^2 \le\\
&\le 4\sum_{t=0}^{\|v\|_1}\sum_{k=0}^\infty
2^{-k}(t+k+1)(h(t+k+1) - h(t))^2 + 4\sum_{m=1}^\infty 2^{-m+1}m
(h(m)-h(0))^2\le\\
&\le 8\sum_{t=0}^{\|v\|_1}\sum_{k=0}^\infty 2^{-k}(t+k+1)(h(t+k+1) -
h(t))^2 =\\
&= \frac{8}{\log(1 + \|v\|_1)} \sum_{t=0}^{\|v\|_1}\sum_{k=0}^\infty
2^{-k}(t+k+1)\log^2\left(\frac{t + k + 2}{t + 1}\right)\le\\
&\le \frac{C'}{\log(1 +
\|v\|_1)}\sum_{t=0}^{\|v\|_1}(t+1)\log^2\left(\frac{t+2}{t+1}\right)\le
\frac{C'}{\log(1 + \|v\|_1)}\sum_{t=0}^{\|v\|_1}\frac{1}{t+1}\le C.
  \end{align*}
  for some absolute constants $C, C'>0$.

  We note that for any $x,s,k\ge 0$ we have that
  \begin{equation*}
    h(x + k) - h(x) \le s\quad\text{if and only if}\quad k\le
    (1+x)\big(\exp\big(s\sqrt{\log(1+\|v\|_1)}\big)-1\big).
  \end{equation*}
  Thus, \eqref{eq:L_with_alpha_and_eta} follows from the definitions
  \eqref{eq:L_tau_eps_def} and \eqref{eq:eta_def} of $L(\tau,\eps)$
  and $\eta$.
\end{proof}
\begin{proof}[Proof of Theorem~\ref{thm:variance_lower_bound}]
Assume that $\|v\|_1\ge (\log n)^2$. It suffices to prove
\eqref{eq:anti-concentration_estimate} and
\eqref{eq:large_deviation_estimate} as \eqref{eq:variance_estimate}
is an immediate consequence of the case $t=1$ of
\eqref{eq:large_deviation_estimate} and the fact that $\E\varphi_v =
0$ by symmetry.

Let $N(U)>0$ be large enough for the following derivations. We first
claim that choosing $c(U)$ sufficiently small and $C(U)$
sufficiently large the theorem holds when $n\le N(U)$. Indeed, this
is clear for \eqref{eq:anti-concentration_estimate} as we may make
the right-hand side greater than $1$ by choosing $C(U)$
appropriately. To see this for \eqref{eq:large_deviation_estimate}
first note that our assumption that the potential $U$ restricted to
$[-1,1]$ is bounded away from infinity implies that
$\P(|\varphi_v|\ge 0.99\|v\|_1)>0$. Thus it suffices to check that
$\frac{(1+\sqrt{\|v\|_1})\sqrt{\log(1+\|v\|_1)}}{1+\log n}\le
0.99\|v\|_1$ and this follows, using our assumption that $n\ge 2$,
as $\frac{(1 + \sqrt{x})\sqrt{\log(1+x)}}{1 + \log 2}\le 0.99 x$ for
$x\ge 1$.

%

Assume for the rest of the proof that $n>N(U)$. Consequently, since
$\|v\|_1\ge (\log n)^2$, we have
\begin{equation}\label{eq:eta_v_lower_bound}
  \eta(v) \ge \frac{1}{4}\left(\sqrt{\log(1 + \|v\|_1)} + 1\right).
\end{equation}

We start with the proof of \eqref{eq:large_deviation_estimate}. Let
$1\le t\le \frac{1+\sqrt{\|v\|_1}}{\log n}$. If $\P(|\varphi_v|\ge
t\sqrt{\log(1+\|v\|_1)})\ge \frac{1}{2}$ there is nothing to prove.
Thus we suppose that $\P(|\varphi_v|\le t\sqrt{\log(1+\|v\|_1)})\ge
\frac{1}{2}$. Pick the function $\tau:=8t\cdot\eta$ so that, since
$\eps\le \frac{1}{2}$, we have $\tau(v) \ge 2t\sqrt{\log(1 +
\|v\|_1)} + \frac{\eps}{2}$ by \eqref{eq:eta_v_lower_bound}.
Combining the arithmetic-geometric mean inequality with
Lemma~\ref{lem:lowerboundForShiftedConfiguration}, taking $A$ to be
the full event, we have
\begin{align}
  \frac{1}{2}&\P\left(|\varphi_v|\ge t\sqrt{\log(1+\|v\|_1)}\right) = \frac{1}{2}\left[\P\left(\varphi_v\ge t\sqrt{\log(1+\|v\|_1)}\right) +
  \P\left(\varphi_v\le -t\sqrt{\log(1+\|v\|_1)}\right)\right] \ge\nonumber\\
  &\ge \left[\P\left(\varphi_v\ge t\sqrt{\log(1+\|v\|_1)}\right) \P\left(\varphi_v\le
  -t\sqrt{\log(1+\|v\|_1)}\right)\right]^{1/2}\ge\nonumber\\
  &\ge \left[\P\left(|\varphi_v - \tau(v)|\le t\sqrt{\log(1+\|v\|_1)}+\frac{\eps}{2}\right) \P\left(|\varphi_v + \tau(v)|\le
  t\sqrt{\log(1+\|v\|_1)}+\frac{\eps}{2}\right)\right]^{1/2} \ge\nonumber\\
  &\ge c(U) s
  \left[\P\left(|\varphi_v|\le t\sqrt{\log(1+\|v\|_1)}\right) -
\P\left(\{J^+(\varphi)J^-(\varphi) <
s^2\}\cup\{M(\varphi)>L(\tau,\eps)\}\right)\right] \ge\nonumber\\
&\ge c(U)s\left[\frac{1}{2} - \P\left(\{J^+(\varphi)J^-(\varphi) <
s^2\}\cup\{M(\varphi)>L(\tau,\eps)\}\right)\right],\label{eq:usage_of_lower_bound_lemma}
\end{align}
where $s>0$ is arbitrary. By
Lemma~\ref{lem:Jacobian_and_bad_edge_diameter} and
Lemma~\ref{lem:eta_Jacobian_and_bad_edge_bound} we have
\begin{equation*}
  \P\left(\{J^+(\varphi)J^-(\varphi) <
s^2\}\cup\{M(\varphi)>L(\tau,\eps)\}\right) \le (2n)^2
2^{-L(\tau,\eps)} + \frac{256C t^2}{\eps^2 \log(1/s)}.
\end{equation*}
Furthermore, our assumption that $t\le \frac{1+\sqrt{\|v\|_1}}{\log
n}$ and $\|v\|_1\ge (\log n)^2$ combined with
\eqref{eq:L_with_alpha_and_eta} yields that
\begin{equation}\label{eq:L_tau_eps_usage_bound}
  (2n)^2
2^{-L(\tau,\eps)} \le \frac{1}{2n}\le \frac{1}{8}
\end{equation}
when $N(U)$ is sufficiently large. Thus, choosing $s =
\exp\left(-\frac{1024 C t^2}{\eps^2}\right)$ and combining the last
inequalities we conclude that
\begin{equation*}
  \P\left(|\varphi_v|\ge t\sqrt{\log(1+\|v\|_1)}\right)\ge
  c'(U)\exp(-C(U)t^2)
\end{equation*}
for some $c'(U),C(U)>0$ depending only on $U$.

We now prove \eqref{eq:anti-concentration_estimate}. We may suppose
that $r\le \frac{1}{4}\sqrt{\log(1 +\|v\|_1)}$ since otherwise
\eqref{eq:anti-concentration_estimate} is trivial. Fix $1\le r\le
\frac{1}{4}\sqrt{\log(1 +\|v\|_1)}$. For $k\ge 1$ choose
$\tau:=\frac{5rk}{\eta(v)}\cdot \eta$, so that $\tau(v) = 5rk$, to
obtain similarly to \eqref{eq:usage_of_lower_bound_lemma},
\begin{align*}
  \frac{1}{2}&\left[\P(|\varphi_v - 5rk|\le r+1) + \P(|\varphi_v+5rk|\le r+1)\right] \ge\\
  &\ge \left[\P(|\varphi_v - 5rk|\le r+1)\P(|\varphi_v+5rk|\le
  r+1)\right]^{1/2}\ge\\
  &\ge \left[\P\left(|\varphi_v - \tau(v)|\le r+\frac{\eps}{2}\right) \P\left(|\varphi_v + \tau(v)|\le
  r+\frac{\eps}{2}\right)\right]^{1/2} \ge\\
  &\ge \frac{1}{2}c(U)
  \left[\P\left(|\varphi_v|\le r\right) -
\P\left(\{J^+(\varphi)J^-(\varphi) <
\frac{1}{4}\}\cup\{M(\varphi)>L(\tau,\eps)\}\right)\right].
\end{align*}
Assume $k\le \frac{\eta(v)}{r}$ (recalling that $\eta(v)>r$ by our
assumption that $r\le \frac{1}{4}\sqrt{\log(1 +\|v\|_1)}$ and
\eqref{eq:eta_v_lower_bound}), so that by
\eqref{eq:L_with_alpha_and_eta} and our assumption that $N(U)$ is
large we have that \eqref{eq:L_tau_eps_usage_bound} holds. Hence by
Lemma~\ref{lem:Jacobian_and_bad_edge_diameter},
Lemma~\ref{lem:eta_Jacobian_and_bad_edge_bound} and the fact that
$r,k\ge 1$,
\begin{equation*}
  \P\left(\{J^+(\varphi)J^-(\varphi) <
\frac{1}{4}\}\cup\{M(\varphi)>L(\tau,\eps)\}\right) \le \frac{1}{2n}
+ C'(U)\left(\frac{r k}{\eta(v)}\right)^2 \le C''(U)\left(\frac{r
k}{\eta(v)}\right)^2
\end{equation*}
for some constants $C'(U), C''(U)>0$ (depending on $U$ through
$\eps$). Thus, in particular, for all $1\le k\le
\left(\frac{\eta(v)}{r}\right)^{2/3}$ we have
\begin{equation*}
  \frac{1}{2}\left[\P(|\varphi_v - 5rk|\le r+1) + \P(|\varphi_v+5rk|\le
  r+1)\right]\ge \frac{1}{2}c(U)\left[\P\left(|\varphi_v|\le
  r\right) - C''(U)\left(\frac{r}{\eta(v)}\right)^{2/3}\right].
\end{equation*}
Summing over $k$ and using that the sum of probabilities of disjoint
events is at most one yields
\begin{equation*}
  c(U)\bigg\lfloor \left(\frac{\eta(v)}{r}\right)^{2/3}\bigg\rfloor \left[\P\left(|\varphi_v|\le
  r\right) - C''(U)\left(\frac{r}{\eta(v)}\right)^{2/3}\right]\le 1.
\end{equation*}
Since $\eta(v)>r$ by our assumption that $r\le
\frac{1}{4}\sqrt{\log(1 +\|v\|_1)}$ and \eqref{eq:eta_v_lower_bound}
it follows that
\begin{equation*}
  \P\left(|\varphi_v|\le
  r\right) \le \left(\frac{2}{c(U)} +
  C''(U)\right)\left(\frac{r}{\eta(v)}\right)^{2/3}.
\end{equation*}
Together with \eqref{eq:eta_v_lower_bound} this proves
\eqref{eq:anti-concentration_estimate}.
\end{proof}

\subsection{Maximum}
In this section we prove Theorem~\ref{thm:maximum_lower_bound}.

Let $\rho(U)>0$ be a constant to be chosen later, depending only on
$U$ and small enough for the following derivations. We may choose
$c(U)$ sufficiently small so that the theorem holds when $n\le
\exp(1/\rho(U)^2)$ and thus we assume that
\begin{equation}\label{eq:n_rho_relation}
n\ge \exp\left(\frac{1}{\rho(U)^2}\right).
\end{equation}
Fix a collection of arbitrary vertices $u_1,\ldots, u_n\in
V(\T_n^2)$ satisfying $\|u_i\|_1\ge \frac{n}{2}$ and
$d_{\T_n^2}(u_i, u_j)>2 n^{1/3}$ when $i\neq j$. Define the events,
for $1\le i\le n$,
\begin{align*}
  B_i&:=\{|\varphi_{u_i}| \ge \rho(U) \log n\},\\
  A_i&:=\cap_{1\le j<i} B_j^c,
\end{align*}
where we mean that $A_1$ is the full event. We have
\begin{equation}\label{eq:B_i_union_expression}
  \P\left(\max_{v\in V(\T_n^2)} |\varphi_v|\ge \rho(U)\log n\right)\ge \P(\cup_{i=1}^n B_i) = \P(B_1) + \P(B_2\cap A_2) + \cdots + \P(B_n\cap A_n)
\end{equation}
and we aim to use Lemma~\ref{lem:lowerboundForShiftedConfiguration}
to estimate the summands on the right-end side. Let $v_0:=(\lfloor
n^{1/3}\rfloor, 0)$ and let $\eta:V(\T_n^2)\to[0,\infty)$ be the
function defined by \eqref{eq:eta_def} with $v=v_0$. Noting that
$\eta$ takes its maximal value at $v_0$ we may define
$\eta_i:V(\T_n^2)\to[0,\infty)$, $1\le i\le n$, by
\begin{equation}\label{eq:eta_i_def}
  \eta_i(w):=\eta(v_0) - \eta(w - u_i),
\end{equation}
where $w-u_i$ is the vertex in $\T_n^2$ obtained by doing the
coordinate-wise difference modulo $2n$. We define also the functions
$\tau_i:V(\T_n^2)\to[0,\infty)$ by
\begin{equation*}
  \tau_i(w):= 16\rho(U)\sqrt{\log n}\cdot \eta_i(w).
\end{equation*}
\begin{lem}
  For all $1\le i\le n$ we have
  \begin{equation}\label{eq:tau_i_disjoint_support}
    \tau_i(w) = 0\quad\text{when}\quad d_{\T_n^2}(w,u_i)\ge \lfloor
    n^{1/3} \rfloor.
  \end{equation}
  In addition, if $\rho(U)$ is sufficiently small then
  \begin{align}
    \tau_i(u_i)&\ge 2\rho(U)\log n + \frac{\eps}{2}\label{eq:tau_i_u_i_bound},\\
    L(\tau_i, \eps)&\ge n^{1/6}\label{eq:tau_i_bad_diameter_bound}
  \end{align}
  and
  \begin{equation}
    \sum_{w\in V(\T_n^2)} \sum_{k=0}^{\infty} 2^{-k} \tau_i'\left(w,k+1\right)^2 \le C\rho(U)^2\log
    n\label{eq:tau_i_Jacobian_preliminary_bound}
  \end{equation}
  for some absolute constant $C>0$.
\end{lem}
\begin{proof}
  Property \eqref{eq:tau_i_disjoint_support} is an immediate
  consequence of the fact that $\eta(v) = \eta(v_0)$ for all vertices $v$ with $\|v\|_1\ge \|v_0\|_1$ and the definition of $\tau_i$.

  To see \eqref{eq:tau_i_u_i_bound}, recall \eqref{eq:n_rho_relation} and observe that
  \begin{equation}\label{eq:eta_i_u_i_bound}
    \eta_i(u_i) = \eta(v_0)\ge \frac{1}{8}(\sqrt{\log n} + 1),
  \end{equation}
  when $\rho(U)$ is sufficiently small, as in
  \eqref{eq:eta_v_lower_bound}. Now use the definition of $\tau_i$
  and the fact that $\eps\le \frac{1}{2}$.

  Since \eqref{eq:eta_i_def} defines $\eta_i$ via $\eta$
  we may use Lemma~\ref{lem:eta_Jacobian_and_bad_edge_bound}, taking
$\rho(U)$ sufficiently small, to obtain
\eqref{eq:tau_i_bad_diameter_bound}. Finally,
\eqref{eq:tau_i_Jacobian_preliminary_bound} follows from a similar
derivation as in the proof of
Lemma~\ref{lem:eta_Jacobian_and_bad_edge_bound}.
\end{proof}
We may now apply Lemma~\ref{lem:lowerboundForShiftedConfiguration}
with $\tau_i$ playing the role of $\tau$ and $A_i$ playing the role
of $A$, noting that by \eqref{eq:tau_i_disjoint_support} and our
choice of the $u_i$, $A_i$ is indeed measurable with respect to the
sigma algebra generated by $\{\varphi_v\colon \tau_i(v) = 0\}$.
Using also the arithmetic-geometric mean inequality and
\eqref{eq:tau_i_u_i_bound} we have
\begin{align}
  \frac{1}{2}&\P\left(B_i\cap A_i\right) = \frac{1}{2}\left[\P\left(\left\{\varphi_{u_i}\ge \rho(U)\log n\right\}\cap A_i\right) +
  \P\left(\left\{\varphi_{u_i}\le -\rho(U)\log n\right\}\cap A_i\right)\right] \ge\nonumber\\
  &\ge \left[\P\left(\left\{\varphi_{u_i}\ge \rho(U)\log n\right\}\cap A_i\right) \P\left(\left\{\varphi_{u_i}\le
  -\rho(U)\log n\right\}\cap A_i\right)\right]^{1/2}\ge\nonumber\\
  &\ge \left[\P\left(\left\{|\varphi_{u_i} - \tau_i(u_i)|\le \rho(U)\log n+\frac{\eps}{2}\right\}\cap A_i\right) \P\left(\left\{|\varphi_{u_i} + \tau_i(u_i)|\le
  \rho(U)\log n+\frac{\eps}{2}\right\}\cap A_i\right)\right]^{1/2} \ge\nonumber\\
  &\ge c(U) s
  \left[\P\left(\left\{|\varphi_{u_i}|\le \rho(U)\log n\right\}\cap A_i\right) -
\P\left(\{J^+(\varphi)J^-(\varphi) <
s^2\}\cup\{M(\varphi)>L(\tau_i,\eps)\}\cap A_i\right)\right]\ge\nonumber\\
&\ge c(U) s
  \left[1 - \P(B_i\cup A_i^c) - \P\left(\{J^+(\varphi)J^-(\varphi) <
s^2\}\cup\{M(\varphi)>L(\tau_i,\eps)\}\right)\right],\label{eq:B_i_union_estimate}
\end{align}
where $s>0$ is arbitrary. Combining Lemma~\ref{lem:Jacobian_and_bad_edge_diameter} with
\eqref{eq:tau_i_bad_diameter_bound} and
\eqref{eq:tau_i_Jacobian_preliminary_bound} we have
\begin{equation*}
  \P\left(\{J^+(\varphi)J^-(\varphi) <
s^2\}\cup\{M(\varphi)>L(\tau_i,\eps)\}\right)\le (2n)^2 2^{-n^{1/6}}
+ \frac{4C\rho(U)^2\log n}{\eps^2\log\left(\frac{1}{s}\right)}.
\end{equation*}
Choosing $s:=\exp(-20C\rho(U)^2\log n / \eps^2)$, taking $\rho(U)$
small enough and using \eqref{eq:n_rho_relation} yields
\begin{equation*}
  \P\left(\{J^+(\varphi)J^-(\varphi) <
s^2\}\cup\{M(\varphi)>L(\tau_i,\eps)\}\right)\le \frac{1}{4}.
\end{equation*}
Plugging back into \eqref{eq:B_i_union_estimate} and summing over
$i$ using \eqref{eq:B_i_union_expression} gives
\begin{equation*}
  \P(\cup_{i=1}^n B_i)\ge 2c(U) n^{-20C\rho(U)^2 / \eps^2} \sum_{i=1}^n\left[\frac{3}{4} - \P(B_i\cup A_i^c)\right].
\end{equation*}
Finally, choosing $\rho(U)$ sufficiently small this implies that
\begin{equation*}
  1\ge \P(\cup_{i=1}^n B_i)\ge \frac{8}{n} \sum_{i=1}^n\left[\frac{3}{4} - \P(B_i\cup
  A_i^c)\right].
\end{equation*}
It follows that there exists some $1\le i\le n$ for which
$\P(B_i\cup A_i^c)\ge \frac{1}{2}$, whence, by the definition of
$A_i$, $\P(\cup_{i=1}^n B_i) \ge \frac{1}{2}$ and the theorem
follows.

\section{Proof of main theorem}\label{sec:main_theorem_proof} In
this section we prove Theorem~\ref{thm:main}.

Let $n\ge 2$. Let $U:\R\to(-\infty,\infty]$ satisfy $U(x)=U(-x)$ and
conditions~\eqref{eq:U_integral_cond} and
\eqref{eq:potential_condition}. Then $U$ satisfies the conditions of
Theorem~\ref{thm:controlledGradients}, whence $U$ has $(1/8,
L)$-controlled gradients on $\T_n^2$ for some $0<L=L(U)<\infty$
which is independent of $n$. By the definition of the controlled
gradients property and condition~\eqref{eq:potential_condition} it
follows that there exists some $K'=K'(U)>0$, independent of $n$,
such that $L<K'\le 2L$ and $U$ is twice continuously differentiable
on $[-K',K']$. Define
\begin{equation}\label{eq:U_tilde_U_relation}
  \tilde{U}(x):=U(K'\cdot x).
\end{equation}
Let $\varphi$ be randomly sampled from $\mu_{\T_n^2, \zero, U}$ and
let $\tilde{\varphi}$ be randomly sampled from $\mu_{\T_n^2, \zero,
\tilde{U}}$. The relation~\eqref{eq:U_tilde_U_relation} implies that
\begin{equation}\label{eq:varphi_tilde_varphi_relation}
  \tilde{\varphi} \,{\buildrel d \over =}\, \varphi/K'.
\end{equation}
Thus $\tilde{U}$ has $(1/8,1-\eps)$-controlled gradients on
$\T_n^2$, where $\eps:=1 - L/K'\in(0,1/2]$, and $\tilde{U}$ is twice
continuously differentiable on $[-1,1]$. We conclude that
$\tilde{U}$ satisfies the conditions in the beginning of
Section~\ref{sec:proofs_for_main_theorems}. The validity of
Theorem~\ref{thm:main} for $U$ now follows from
\eqref{eq:varphi_tilde_varphi_relation} and the validity of
Theorems~\ref{thm:variance_lower_bound}
and~\ref{thm:maximum_lower_bound} for $\tilde{U}$ by noting the
following points: First, $\var(\varphi_v) =
K'^2\var(\tilde{\varphi})$. Second, $\P(|\varphi_v|\le
\delta\sqrt{\log(1 + \|v\|_1)}) = \P(|\tilde{\varphi}_v|\le
\delta\sqrt{\log(1 + \|v\|_1)} / K')$ and we may consider separately
the cases $K'\le 1$ and $K'>1$ when applying
Theorem~\ref{thm:variance_lower_bound}. Third, $\P(|\varphi_v| \ge
c(U)t\sqrt{\log(1+\|v\|_1)}) = \P(|\tilde{\varphi}_v| \ge
c(U)t\sqrt{\log(1+\|v\|_1)} / K')$ and we may choose $c(U) \le K'$
in order to use Theorem~\ref{thm:variance_lower_bound}. Finally,
$\max_{v\in V(\T_n^2)} |\varphi_v| \,{\buildrel d \over =}\, K'
\max_{v\in V(\T_n^2)} |\tilde{\varphi}_v|$.

\section{Discussion and open questions}\label{sec:open_prob}
In this work we prove lower bounds for the fluctuations of
two-dimensional random surfaces. Specifically, we investigate random
surface measures of the form \eqref{eq:mu_T_n_2_U_measure_def} based
on a potential $U$ satisfying the
conditions~\eqref{eq:U_integral_cond} and
\eqref{eq:potential_condition}. These conditions allow for a wide
range of potentials including the hammock potential, when $U(x) = 0$
for $|x|\le 1$ and $U(x) = \infty$ for $|x|>1$, double well and
oscillating potentials. We prove that such random surfaces
delocalize, with the variance of their fluctuations being at least
logarithmic in the side-length of the torus. We also establish related
bounds on the maximum of the surface and on large deviation and
small ball probabilities. In this section we discuss related
research directions and open questions.

{\bf Upper bound on the fluctuations}. It is expected that under
  mild conditions on the potential there holds an upper bound of matching order on the fluctuations of the
  random surface. For instance, that if $\varphi$ is randomly sampled from the
  measure \eqref{eq:mu_T_n_2_U_measure_def} then
  $\var(\varphi_{(n,n)})\le C(U) \log n$ for some $C(U)<\infty$ and
  all $n\ge 2$. One may well speculate the result to hold for all potentials
  satisfying \eqref{eq:U_integral_cond} and
  \eqref{eq:potential_condition} and indeed even in greater generality. Certain potentials are known to satisfy such a bound but it appears that even
  the case of the potential $U(x) = x^4$ has not yet been settled \cite[Remark
  6 and open problem 1]{Velenik:2006kx}.

{\bf Reflection positivity}. Our work relies crucially on
  reflection positivity and the chessboard estimate to establish
  what we called the controlled gradients property, see the
  beginning of Section~\ref{sec:reflection_positivity}. This restricts our results in ways which are probably not
  essential. Specifically, we may handle only random surface measures
  on a torus with even side length and we must normalize such measures at a single
  point. It is desirable to lift these restrictions, by possibly
  arriving at a more illuminating proof of the controlled gradients
  property. This will allow to treat random surface measures on other
  graphs as well as on the graph $\T_n^2$ with other boundary
  conditions. For instance, one would expect our results to hold for
  zero boundary conditions, when $\varphi_v$ is normalized to
  zero at all $v=(v_1,v_2)$ with $\max(|v_1|, |v_2|) = n$.

  With regards to this we put forward that the controlled gradients property possibly holds for any finite, connected
  graph $G$ and any potential $U$, satisfying the conditions \eqref{eq:U_integral_cond} and
  \eqref{eq:potential_condition}, say. Precisely, let $G$ and $U$ be such a graph and potential. Write $K:=\sup\{x\colon U(x)<\infty\}\in(0,\infty]$ and let $\varphi$
  be randomly sampled from the probability measure
  \begin{equation}
    d\mu_{G, v_0, U}(\varphi) := \frac{1}{Z_{G,v_0, U}} \exp\Bigg(-\sum_{(v,w)\in E(G)}
    U(\varphi_v - \varphi_w)\Bigg) \delta_0(d\varphi_{v_0})\prod_{v\in V(G)\setminus\{v_0\}}
    d\varphi_v,
  \end{equation}
  for some vertex $v_0\in V(G)$. Then it may be that for any
  $0<\delta<1$ there exists some $0<L<K$ such that $L$ depends only on $\delta$ and $U$ (and not on $G$) and if we define the random subgraph $\EC(\varphi, L)$ of $G$ by
  \begin{equation}\label{eq:E_varphi_L_random_graph_1}
    \EC(\varphi, L):=\{(v,w)\in E(G)\colon |\varphi_v-\varphi_w|\ge
    L\}
  \end{equation}
  then
  \begin{equation*}
    \P(e_1, \ldots, e_k\in \EC(\varphi,L))\le \delta^k\quad\text{for all
    $k\ge 1$ and distinct $e_1,\ldots, e_k\in E(G)$}.
  \end{equation*}

{\bf More general random surfaces}. One may try to extend
  the applicability of our results in several directions. First, one
  may try and relax the condition \eqref{eq:potential_condition} to
  allow for singular potentials. Ioffe, Shlosman and Velenik \cite{Ioffe:2002fk} introduced a
  technique for proving lower bounds on fluctuations for potentials
  which are small perturbations, in some sense, of smooth
  potentials. These ideas were also incorporated in the work of Richthammer
  \cite{Richthammer:2007fk} upon which our addition algorithm is based. It is a promising avenue for future research to try and combine the
  techniques of \cite{Ioffe:2002fk} with our technique. This may allow to
  treat all continuous (not necessarily differentiable)
  potentials as well as certain classes of discontinuous
  potentials.

  Second, one may try and extend the results to integer-valued random surface
  models. For instance, to probability measures on configurations
  $\varphi:\T_n^2\to\mathbb{Z}$ (rather than $\varphi:\T_n^2\to\R$) with $\varphi(\zero)=0$ for
  which the probability of $\varphi$ is proportional to $\exp\Bigg(-\sum_{(v,w)\in E(\T_n^2)}
  U(\varphi_v - \varphi_w)\Bigg)$. This direction seems much more
  challenging as our technique is based on an argument
  which relies crucially on the continuous nature of the model.
  We mention that while it is expected that many integer-valued random surface
  models have fluctuations with variance of logarithmic order this
  has been established only in two cases: when $U(x) = \beta|x|$
  and $U(x)=\beta x^2$, both with $\beta$ sufficiently small. This result is by Fr\"ohlich and Spencer \cite{Frohlich:1981yq}. It is also
  known that if $\beta$ is large then these models become localized,
  having fluctuations with bounded variance, a transition which is
  called the roughening transition. As specific examples of surfaces for which delocalization is expected but remains unproved
  we mention integer-valued analogs of the hammock potential, when
  $U(x) = 0$ for $x\in\{-1,1\}$ and otherwise $U(x)=\infty$ (the
  graph-homomorphism or homomorphism height function model) or when
  $U(x)=0$ for $x\in\{-M,-M+1,\ldots, M\}$ and otherwise $U(x) =
  \infty$ (the $M$-Lipschitz model). The former of these models can
  be used as a height function representation for the square-ice or $6$-vertex models and is also related to the zero temperature 3-state
  antiferromagnetic Potts model (i.e., uniformly chosen proper colorings of $\T_n^2$ with
  $3$ colors). For more on these models we refer to \cite{Peled:2010fj} where it is proved that the homomorphism height function and $1$-Lipschitz
  models are localized in sufficiently high dimensions.

{\bf Scaling limits and Gibbs states}. The study of various limits for random surface models has received a great deal of attention in the literature. Infinite volume Gibbs states fail to exist for the random surface itself due to its delocalization but may exist for its gradients. Funaki and Spohn \cite{Funaki:1997fk} proved that for uniformly convex potentials $U$, i.e., potentials
satisfying  $0<c\le U''(x)\le C<\infty$, a unique infinite volume
gradient Gibbs measure exists for any value of `tilt'. Another direction studied in \cite{Funaki:1997fk} was to consider the
Langevin dynamics of the random surface. Under hydrodynamic scaling, convergence to
a solution of a PDE, the so-called motion by mean curvature dynamics, was established. Naddaf
and Spencer \cite{Naddaf:1997yq}, with an alternative scaling, proved the convergence of the model to a
continuous Gaussian free field. Further, Giacomin, Olla and Spohn \cite{Giacomin:2001rt} extended these results to
the Langevin dynamics obtaining convergence to an infinite-dimensional Ornstein-Uhlenbeck
process. Under similar convexity assumptions Miller \cite{Miller11} extended the scaling limit results to handle
various choices of boundary conditions. Finally, we mention deep connections with the SLE theory.
Schramm and Sheffield  \cite{Schramm:2009uq} discovered that, in the scaling limit, appropriately defined contour lines
of the two-dimensional discrete Gaussian free field converge to an SLE curve with parameter $\kappa=4$.
It was conjectured that this is a universal phenomena independent of potential details. A significant
contribution in this area has been made by Miller \cite{Miller2} who resolved the conjecture for a large class
of uniformly convex potentials.

It is expected that the results described in this section hold under mild assumptions on
the potential $U$. As a first step, one may let $\varphi$ be randomly sampled from the random surface
model \eqref{eq:mu_T_n_2_U_measure_def} with the potential $U (x) = x^4$
or the hammock potential and try to prove that the law
of $\varphi_{(n,n)}$, suitably normalized, converges to a Gaussian distribution. The above-mentioned works used uniform convexity via the Brascamp-Lieb inequality, Helffer-Sj\"ostrand representation or homogenization techniques and novel techniques may be required to extend the results beyond this setting. The question of unicity for gradient Gibbs states seems more delicate as Biskup and Koteck\'{y} \cite{Biskup:2007fk} gave an example of a non-convex potential admitting multiple
gradient Gibbs states with the same `tilt'.

{\bf Maximum in high dimensions}. Our work establishes that the
expected maximum of the random
  surfaces we consider is of order at least $\log n$ and it is
  expected that this is the correct order of magnitude. A curious question regards the maximum in
  higher dimensions. For instance, denote by $\T_{n}^d$ the
$d$-dimensional discrete torus with vertex set
$\{-n+1,-n+2,\ldots,n-1,n\}^d$ and let $\varphi$ be randomly sampled
from the random surface measure \eqref{eq:mu_T_n_2_U_measure_def}
with $\T_n^2$ replaced by $\T_n^d$ for some $d\ge 3$. It is known
that for the discrete Gaussian free field, when $U(x) = x^2$, the
maximum of the field is typically of order $\sqrt{\log n}$ as $n$
tends to infinity. However, it may well be that the behavior of the
maximum is now potential-specific. How would the maximum behave for
the hammock potential, i.e., for a uniformly chosen Lipschitz
function? Observe that if a Lipschitz function is at height $t$ at a
given vertex then it is at height at least $t/2$ in a ball of radius
$t/2$ around that vertex, a ball containing order $t^d$ vertices.
This raises the possibility that the probability of a random
Lipschitz function to attain height $t$ at a given vertex decays as
$\exp(-c t^d)$. This bound would imply that the typical maximal
height is of order at most $(\log n)^{1/d}$, as $n$ tends to
infinity. Is this the correct order of magnitude? The technique of
Benjamini, Yadin and Yehudayoff \cite{BenjaminiYadinYehudayoff07}
may lead to a lower bound of this order. For the integer-valued
models of Lipschitz functions mentioned above, the homomorphism
height function and 1-Lipschitz models, an upper bound of order
$(\log n)^{1/d}$ on the expected maximum was established in
\cite{Peled:2010fj} in sufficiently high dimensions. We mention also
the works \cite{PeledSamotijYehudayoff13,
PeledSamotijYehudayoff13_2} where the maximum of such Lipschitz
function models is studied on expander and tree graphs.

{\bf Decay of correlations}. Let $\varphi$ be randomly sampled from
the random surface measure \eqref{eq:mu_T_n_2_U_measure_def}. Our
results focus on estimating $\var(\varphi_v)$ for various vertices
$v$, i.e., the diagonal elements of the covariance matrix of
$\varphi$. How do the off-diagonal elements behave? How fast do the
values of $\varphi$ decorrelate? A related question is to study the
decay of correlations for the gradient of $\varphi$. Sufficiently
fast decay of gradient correlations will lead to an upper bound on
$\var(\varphi_v)$, by writing $\varphi_v$ as the sum of the
gradients of $\varphi$ on a path leading from $\zero$ to $v$ and
averaging over many such paths. With regards to this we mention the
results of Aizenman \cite{Aizenman:1994fk} and Pinson
\cite{Pinson98}, following ideas of Patrascioiu and Seiler
\cite{Patrascioiu:1992fk}, who give a lower bound, in a certain
sense, for the decay of correlations for the Hammock potential and
for the integer-valued homomorphism height function model mentioned
above.

{\bf High-dimensional convex geometry}. The case that the potential
$U$ is the hammock potential is natural also from a geometric point
of view. In this case the measure \eqref{eq:mu_T_n_2_U_measure_def}
is the uniform measure on the high-dimensional convex polytope of
Lipschitz functions defined by
\begin{equation*}
  \text{Lip}:=\cbr{\varphi:\T_n^2\to\R\colon \varphi_{\zero}=0\text{ and }|\varphi_v-\varphi_w|\leq1\text{ when $v\sim w$}}.
\end{equation*}
The field of convex geometry is highly developed and we mention here
the central limit theorem of Klartag \cite{Klartag07} which states
that uniform measures on high-dimensional convex bodies have many
projections which are approximately Gaussian. It would be
interesting to use this point of view to obtain new results for the
random surface with the hammock potential.

\section*{Acknowledgments} We wish to thank people whose support we enjoyed during the research connected with this paper. First of all we thank Yvan Velenik who introduced PM to the problem, put the two authors together and suggested using the techniques developed in \cite{Richthammer:2007fk}.
Further, we thank Senya Shlosman for the suggestion to use
reflection positivity and Marek Biskup and Roman Koteck\'{y} for
useful discussions. Finally, we thank an anonymous referee whose
many suggestions greatly improved the paper.

\bibliographystyle{plain}

\end{document}